\documentclass[oneside,reqno,11pt]{amsart}

\usepackage[top=1in, bottom=1.25in, left=1.25in, right=1.25in]{geometry}
\usepackage{amsmath}
\usepackage{amsfonts}
\usepackage{amssymb}
\usepackage{amsthm}
\usepackage{hyperref}
\usepackage{xcolor}
\definecolor{revisionblue}{RGB}{20,55,155}
\newif\ifshowrevisions

\definecolor{neighborhoodpurple}{RGB}{110,35,135}

\hypersetup{hidelinks}
\usepackage{mathrsfs}
\usepackage{tikz-cd}
\usepackage{multicol}
\usepackage{xypic}
\usepackage[font=small,labelfont=bf]{caption}
\usepackage{mathtools}
\usepackage{bbm}

\newtheorem{thm}{Theorem}[section]
\newtheorem{lemma}[thm]{Lemma}

\newtheorem{cor}[thm]{Corollary}
\newtheorem{prop}[thm]{Proposition}
\newtheorem{defn}[thm]{Definition}
\newtheorem{example}[thm]{Example}
\newtheorem{remark}[thm]{Remark}
\newtheorem{conjecture}[thm]{Conjecture}
\newtheorem{assumption}[thm]{Assumption}

\newtheorem*{claim}{Claim}

\newcommand{\Z}{\mathbb{Z}} 
\newcommand{\C}{\mathbb{C}} 

\newcommand{\bL}{\mathbb{L}} 
\newcommand{\unit}{1_\bL} 
\newcommand{\cF}{\mathcal{F}}  
\newcommand{\cA}{\mathcal{A}}  
\newcommand{\CF}{\mathrm{CF}} 
  
\newcommand{\Hom}{\mathrm{Hom}}  
\newcommand{\R}{\mathbb{R}}  
\newcommand{\bP}{\mathbb{P}}  
\newcommand{\bS}{\mathbb{S}}  
\newcommand{\bb}{\mathbbm{b}} 
\newcommand{\bv}{\mathbf{v}}
\newcommand{\Id}{\mathrm{Id}}  
\newcommand{\cO}{\mathcal{O}}  
\newcommand{\dg}{\mathrm{dg}}
\newcommand{\cG}{\mathcal{G}}
\newcommand{\cS}{\mathcal{S}}
\newcommand{\cY}{\mathcal{Y}}

\newcommand{\fr}{\mathrm{fr}}
\newcommand{\Fuk}{\mathrm{Fuk}}

\newcommand{\sslash}{/\!/}

\numberwithin{equation}{section}

\allowdisplaybreaks[1]

\begin{document}
	
\title{Mirror functor for deformed preprojective algebras}

\author{Hansol Hong}
\address{Department of Mathematics\\ Yonsei University\\ South Korea}
\email{hansolhong@yonsei.ac.kr}

\author{Siu-Cheong Lau}
\address{Department of Mathematics and Statistics\\ Boston University\\ USA}
\email{lau@math.bu.edu}
\author{Ju Tan}
\address{The IBS Center for Geometry and Physics\\ South Korea}
\email{jutan@ibs.re.kr}

\begin{abstract}

We construct noncommutative mirrors from the bulk-deformed Floer theory of compact immersed Lagrangians equipped with higher rank flat bundles, whose Floer theory can have non-zero curvature terms due to contributions of stable constant discs. For a nodal union of Lagrangian spheres, codimension-two bulk cycles meeting the compact core produce the additive deformed preprojective relations and the corresponding framed Maurer-Cartan spaces at nonzero complex moment-map levels. In particular, our localized mirror functors produce noncommutative ADHM and Nakajima monadic complexes. We obtain local charts and gluing maps of the noncommutative mirror varieties via explicit computations of Lagrangian deformation spaces and Fukaya isomorphisms between them. These constructions recover noncommutative spaces studied by Kapustin-Kuznetsov-Orlov, Baranovsky-Ginzburg-Kuznetsov, and Kawamata, with further examples in dimension three.

We also establish the categorical framework in the unobstructed setting. For a compact $\mathbb Z$-graded unobstructed immersed Lagrangian $\bL$ without nonpositive-degree immersed generators, Koszul duality identifies the derived Fukaya subcategory generated by $\bL$ with the derived category of modules with finite-dimensional total cohomology over its extended localized mirror $\tilde{\cA}_{\bL}$. This construction starts from the Floer algebra of the compact reference brane and does not require a global wrapped generator. Applications to plumbings recover known split-generation results for compact objects.

\end{abstract}

\maketitle

\tableofcontents

\section{Introduction} \label{sec:intro}

Deformed preprojective algebras form one of the basic bridges between quiver representation theory, Kleinian singularities, and noncommutative geometry.  They were introduced by Crawley-Boevey and Holland~\cite{CBH98} as noncommutative deformations of the usual preprojective algebra, which are especially important for quotients of $\C^2$ by finite subgroups of $\mathrm{SL}_2(\C)$.

\begin{defn}[\cite{CBH98}]\label{def:dPPA}
	Let $K$ be a field and let $Q$ be a finite quiver with vertex set $I$.
	The deformed preprojective algebra of $Q$ with parameter $\lambda \in K^I$ is
	\[
		\Pi(Q)^\lambda
		:= K\bar{Q}\Big/\left\langle \sum_{a \in Q_1}[a,a^*]-\lambda \right\rangle, \quad \lambda= \sum_{i \in I} \lambda_i e_i
	\]
	where $\bar{Q}$ is the double quiver of $Q$ and $Q_1$ is the set of original arrows of
	$Q$ and $e_i$ is the vertex idempotent.  When $\lambda=0$, this recovers the ordinary preprojective algebra.
\end{defn}

The relation between preprojective algebras and symplectic topology has appeared in several forms. Etg\"u-Lekili~\cite{EL19} describe wrapped Fukaya categories of plumbings in terms of derived multiplicative preprojective algebras. This gives an algebraic description of the wrapped Fukaya category.
	
Our starting point is instead the Floer deformation theory of a compact Lagrangian immersion in a possibly non-exact symplectic manifold. In the framed setting, \cite{HLT24, LT26} constructed localized mirror functors from framed plumbings to Nakajima quiver varieties~\cite{Nak94,Nak98,Nak01}. Under these functors, the monadic complexes, which play an important role in Nakajima's theory, arise as images of framed Lagrangian immersions.  

The localized mirror constructions of \cite{HLT24,LT26} treat the case of zero complex moment-map level, corresponding to $\lambda=0$. In Nakajima's construction, the deformation parameter $\lambda$ specifies the level of the complex moment map in the hyperK\"ahler quotient. A guiding question in this paper is to identify the mirror-symmetric origin of the deformation parameter $\lambda$ on the symplectic side.

Our main geometric calculations concern a small symplectic neighborhood of a nodal compact Lagrangian $\bL$, that is, an immersion with transverse double points. Boundary and bulk deformations lead to stable pearl trajectories with non-trivial constant disc contributions that, to our knowledge, have not previously been described in the literature. For a nodal union of spheres, we realize nonzero complex moment-map levels directly from the resulting noncommutative (nc) Maurer-Cartan equations. Codimension-two bulk cycles meeting the Lagrangian immersion supply the constant terms in the preprojective relations. In particular, the resulting localized mirror functors produce commutative and noncommutative ADHM monadic complexes. The following one-node example illustrates the origin of the bulk parameter.

\begin{example} \label{ex:ncC2}
Consider the algebra
$$ \C\langle x,y \rangle / \langle xy-yx+\lambda\rangle .$$
For the one-vertex, one-loop quiver $Q$, this is the deformed preprojective algebra $\Pi(Q)^{-\lambda}$ with parameter $-\lambda$ in the Convention of Definition \ref{def:dPPA}.  At $\lambda=0$, it is the coordinate ring of the commutative affine plane $\C^2$. In~\cite{HKL23}, the case $\lambda=0$ was realized as the Maurer-Cartan	deformation space of a nodal Lagrangian sphere $\mathscr{S}$ with one node. If $b=\tilde{x}X+yY$ is the formal boundary deformation at the node, with $\tilde{x}$ and $y$ treated as noncommutative variables, then, after a certain change of coordinate from $\tilde{x}$ to $x$, its Maurer-Cartan equation becomes $m_0^b=(xy-yx)P_2$. Here $P_2$ is the degree-two Morse generator.  See Figure~\ref{fig:ncADHM}.
	
We now explain how the constant term $\lambda$ appears on the symplectic side.  We work in a symplectic neighborhood of $\mathscr{S}$, and take a cotangent fiber $C$ through a generic  point $p \in \mathscr{S}$ away from the node. We then consider the bulk deformation of Floer theory of $\mathscr{S}$ by $\bb=\lambda C$.  Although an exact Lagrangian cannot bound a nonconstant disc without corners, the bulk insertion stabilizes the constant disc through $p$.  In the Morse model, this constant disc is followed by a gradient trajectory to the degree-two critical point $P_2$, and hence contributes the term $\lambda P_2$ to $m_0^\bb$. 
	
The first interior insertion is essential for stability: forgetting it leaves a constant disc with only the output boundary marked point, which is unstable. Thus this contribution is not ruled out by the zero intersection number of the constant class. For further interior insertions, the divisor axiom applies to the stable forgetful maps and gives zero, since $C$ has codimension two and the constant class has zero intersection with it. Hence $m_0^\bb=\lambda P_2$, and the bulk-deformed Maurer-Cartan equation becomes
$$ m_0^{\bb,b}=(xy-yx+\lambda)P_2. $$
\end{example}

\begin{figure}[h]
	\begin{center}
		\includegraphics[scale=0.7]{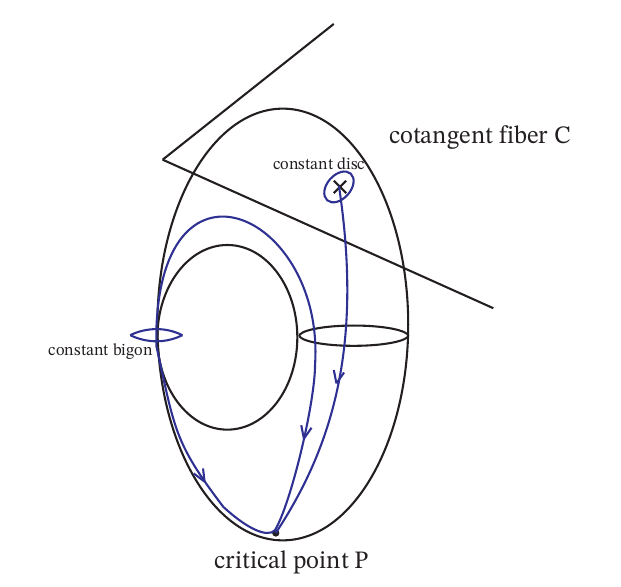}
		\caption{The bulk contribution for a nodal sphere. A constant disc constrained to the cotangent fiber $C$, followed by a Morse trajectory to the degree-two generator $P_2$, contributes $\lambda P_2$. The configurations at the node give the commutator term after the coordinate normalization.}\label{fig:ncADHM}
	\end{center}
\end{figure}

For plumbings of two-spheres, the same construction yields the deformed preprojective algebra associated with the underlying plumbing quiver. More precisely, we prove the following result using a Morse model for the bulk-deformed Floer theory.

\begin{thm}[see Theorem \ref{thm: dpre} for the precise statement] \label{thm: intro_dpre}
For a plumbing of two-spheres, there are bulk deformations for which the bulk-deformed Maurer-Cartan algebra of the zero-section (the union of two-spheres) $\bL$ is isomorphic to the deformed preprojective algebra of the quiver canonically obtained from the plumbing graph.
	
Similarly, let $(\bL^{\fr},\mathcal{E})$ be the framed Lagrangian brane obtained by adjoining to $\bL$ a cotangent fiber in each sphere component, and let $\mathcal{E}$ be the trivial bundle with rank vector $(\vec{n},\vec{d})$.  For each vector of deformation parameters	 $\vec{\delta}$, there are bulk deformations for which the Maurer-Cartan space of $\bL^{\fr}$ is isomorphic to the corresponding Nakajima quiver variety at	 complex moment-map level $\mu_\C=\vec{\delta}$.
\end{thm}

This result gives a direct mirror-symmetric interpretation of the deformation parameters: they record the weights of codimension-two bulk cycles meeting the compact core. The role of the compact reference immersion becomes particularly important when these weights are nonzero. In the Morse model used below, an individual core sphere has curvature $m_0^{\bb}(L_v)=\delta_v P_v$. For $\delta_v\neq0$, the sphere with zero bounding cochain is therefore not an object of the uncurved bulk-deformed Fukaya category. We nevertheless retain its curved Floer operations as part of the deformation problem for the whole immersion. The degree-one immersed generators provide boundary-deformation variables, and the Maurer-Cartan equations specify when their contributions cancel the componentwise curvature. When solutions exist over admissible coefficients, they describe unobstructed deformations of this reference immersed Lagrangian.

For plumbings of cotangent bundles of spheres without bulk deformation, the localized mirror can be compared with a derived completion, as explained in Section~\ref{subsec:globalms}. Under bulk deformation, however, the reference brane itself may become obstructed before nc boundary deformation, and hence is not seen by the wrapped objects. Algebraically, the arrow-zero augmentation of the undeformed preprojective algebra no longer survives when a nonzero constant term is imposed at a vertex. Thus the bulk-deformed construction is not simply a completion at the same undeformed brane. Its purpose is to obtain the curvature equations, their higher-rank and framed versions, and the associated mirror complexes directly from compact-brane Floer theory.


We are particularly interested in the case of affine $ADE$ quivers, with their distinguished primitive dimension
vectors associated to torus-type branes. For these quivers, the framed construction gives a uniform Floer-theoretic realization of the ADHM and Nakajima monadic complexes. 
\begin{thm}[See Theorem~\ref{thm: monad}]
	For an affine $ADE$ plumbing, the bulk-deformed localized mirror functor sends stable framed Lagrangian branes to monadic complexes representing framed torsion-free sheaves on the noncommutative projective surface $\mathbb P^2_{\tau,\Gamma}$, where
	$\Gamma\subset \mathrm{SL}_2(\mathbb C)$ corresponds to the affine
	diagram and $\tau$ is determined by the bulk parameters.
\end{thm}


Using the Morita equivalence between the deformed preprojective algebra and the corresponding deformed skew group algebra, we identify both the terms and the differentials of the complexes. The same Floer-theoretic construction applies at zero and nonzero bulk parameter, with the deformation entering through the coefficient algebra and the Maurer-Cartan equations. It therefore recovers the zero level construction of \cite{HLT24}, and realizes, from the Fukaya category, the noncommutative ADHM and equivariant constructions of Kapustin-Kuznetsov-Orlov~\cite{KKO01} and Baranovsky-Ginzburg-Kuznetsov~\cite{BGK02}.

In the affine $A_n$ case, we also obtain an explicit description of the local charts and their gluing. The bulk-deformed Maurer-Cartan algebras of nodal Lagrangian spheres give the noncommutative affine charts of the deformed $A_n$-resolution (Corollary~\ref{cor: An}). Their comparison maps with the noncommutative deformation space of the core Lagrangian arise by solving Floer-theoretic isomorphism equations between the corresponding branes. These maps determine the transition maps between the charts and recover the noncommutative deformations of the $A_n$-resolution constructed by Kawamata~\cite{Kaw24A} (Proposition~\ref{prop:A_n}). Moreover, the localized mirror functor sends the direct sum of the cocores to the tilting bundle in Kawamata's construction (Lemma \ref{lem:tilting}), giving a Floer-theoretic realization of both the chart gluing and the tilting bundle.

Beyond type $A$, Example \ref{eg: D4} gives an explicit noncommutative local chart for the affine $D_4$ quiver with dimension vector $(1,1,1,1,2)$. The computation uses rank-two data at the central vertex and yields a noncommutative $\C^2$ chart with relation $$XT-TX+\delta_1+\delta_2+\delta_3+\delta_4+2\delta_5=0.$$
This provides a concrete higher-rank example in which the bulk parameters determine the noncommutative deformation of a local chart.

We also investigate an analogous construction in dimension three. Suitable bulk cycles in the smoothing of the conifold deform the localized mirror into a noncommutative crepant resolution of the conifold and give rise to local noncommutative charts organized into a quiver algebroid stack (Proposition \ref{prop:coni} and Lemma \ref{lem:conichart}). 

The first part of the paper (Sections~\ref{sec:extmirkos}, \ref{sec:locfunc}, and~\ref{sec:higher}) establishes the categorical framework for an unobstructed compact reference brane. Let $\bL$ be such an immersed Lagrangian in a symplectic manifold $X$, possibly equipped with a higher-rank flat bundle. Its extended localized mirror $\widetilde{\cA}_{\bL}$ is a differential graded quiver algebra.  Algebraically, $\widetilde{\cA}_{\bL}$ is the Koszul dual of the $A_\infty$ algebra $\CF(\bL,\bL)$. We describe it explicitly using the completed reduced cobar construction in Section~\ref{sec:extmirkos}, working over the Novikov field $\Lambda$ when required for convergence.

The higher-rank setting naturally gives rise to a gauge group $\cG$, so the mirror should be regarded as the stack $(\widetilde{\cA}_{\bL},\cG)$. Multiplicities already arise in  basic examples, including the Lagrangian immersions associated with affine Dynkin diagrams of type $D_n$ and $E_n$ for $n=6,7,8$. As in \cite{HLT24}, framed branes (denoted by $\bL^{\fr}$) are also essential for constructing ADHM spaces and general Nakajima quiver varieties as shown in Theorem \ref{thm: intro_dpre}.

In the rank-one unframed case, \cite{CHL21} constructed an extended mirror functor $\cF^{\bL}$ from the Fukaya category of $X$ to the dg category of dg modules over $\tilde{\cA}_{\bL}$. We extend this functor to the framed and higher-rank setting. Along the way, the Koszul-dual description yields, under suitable finiteness hypotheses, an explicit free bimodule resolution of $\tilde{\cA}_{\bL}$. This resolution identifies the Hochschild cohomology of the localized mirror with an appropriate version of the Floer cohomology of $\bL$ and equips $\tilde{\cA}_{\bL}$ with a canonical Calabi-Yau structure.

Moreover, using Koszul duality, we prove the following quasi-equivalence:
\begin{thm}[Theorem \ref{thm:extfuctorequiv} and Lemma \ref{lem:fd=nil}]
	Let $\bL$ be a compact $\mathbb{Z}$-graded unobstructed immersed Lagrangian without nonpositive degree immersed generators.
	Then the extended mirror functor $\cF^\bL$ induces a quasi-equivalence
	$$\mathcal{D} \Fuk_{\mathbb{L}} (X) \longrightarrow \mathcal{D}_{\mathrm{fd}}(\tilde{\cA}_\bL),$$ where $\Fuk_{\mathbb{L}} (X) \subset \Fuk(X)$ is the full subcategory generated by $\mathbb{L}$ and $\mathcal{D}_{\mathrm{fd}}(\tilde{\cA}_\bL)$ denotes the derived category of $\tilde{\cA}_\bL$-modules with finite-dimensional total cohomology.
\end{thm}



The unobstructedness assumption in this quasi-equivalence is essential: the statement is not being applied to a bulk-deformed reference immersion with uncancelled curvature. As a first application in the unobstructed setting, we recover known split-generation results for the core Lagrangians of several plumbings of $T^*\bS^n$ in Section~\ref{sec:fukplumb} (see Lemma~\ref{lem:gen}).


Suppose that $\tilde{\cA}_{\bL}$ is formal and that its cohomology is concentrated in degree zero. Then it is quasi-isomorphic to the ordinary Maurer-Cartan algebra, $ \cA_{\bL}=H^0(\tilde{\cA}_{\bL})$.
In this case, a spectral-sequence argument shows that the extended mirror functor reduces to the localized mirror functor with target the category of $\cA_{\bL}$-modules. Consequently, it induces an equivalence between the derived Fukaya category split-generated by the core spheres and the derived category of finite-dimensional $\cA_{\bL}$-modules.

Applying this result to plumbings of $T^\ast \bS^2$ associated with affine $ADE$ diagrams, we obtain a fully faithful embedding of the derived Fukaya category split-generated by the core spheres into the derived category of coherent sheaves on the crepant resolution of $\C^2/\Gamma$. In particular, the non-affine core Lagrangian spheres are sent to coherent sheaves supported on the exceptional locus. See Corollary \ref{cor:localHMS} for the precise correspondence.

A cotangent plumbing provides a standard exact ambient realization of a neighborhood of the compact reference immersion, rather than the only possible ambient manifold. For instance, the same neighborhood may lie in a non-exact symplectic manifold: the constant-disc contributions depend on the local brane and bulk data, while nonconstant ambient polygons may add further terms. Under the positivity and compatibility hypotheses of Proposition~\ref{prop:zero-intersection} and Corollary~\ref{cor:cpt}, anticanonical compactifications preserve the non-unit weak Maurer--Cartan relations while introducing a superpotential $W$ over $\cA_\bL$. The $\mathbb{P}^1\times\mathbb{P}^1$ example in Section~\ref{sec:fukplumb} gives a concrete non-exact ambient realization.


The paper is organized as follows. Section \ref{sec:prelim} reviews the extended localized mirror and the mirror functor. Section \ref{sec:extmirkos} develops their Koszul-dual description, and Section \ref{sec:locfunc} establishes the categorical properties of the mirror functor. Section \ref{sec:higher} treats higher-rank flat bundles. Section \ref{sec:fukplumb} applies the construction to plumbings of cotangent bundles of spheres and discusses anticanonical compactifications. Finally, Section \ref{bulk_deform} studies bulk deformations, including deformed preprojective algebras, noncommutative local charts, and framed ADHM and Nakajima monadic complexes.

\subsubsection*{Notations}
Throughout, we will use the following to denote a general element of a tensor algebra:
$$\vec{X} = X_1 \otimes \cdots \otimes X_l$$
for some $l \geq 0$. For $v=(v_1,\cdots, v_k)$ and $w=(w_1,\cdots,w_k)$,
$$x_{\vec{v}}=x_{v_1} \otimes \cdots \otimes x_{v_k}, \qquad x_{\vec{v}^{op}}=x_{v_k} \otimes \cdots \otimes x_{v_1},$$ a notation frequently used in the localized mirror formalism is
$$x_{\vec{v}} X_{\vec{w}^{op}} = x_{v_l} X_{w_l} \otimes x_{v_{l-1}} X_{w_{l-1}} \otimes \cdots \otimes x_{v_1} X_{w_1}.$$

For a graded module $M = \oplus_{d \in \mathbb{Z}} M_d$ over a coefficient ring $\Bbbk$, $\hom_\Bbbk (M, \Bbbk)$ means the full linear dual
$$ \hom_\Bbbk (M, \Bbbk) = \prod_d \hom_\Bbbk(M_d, \Bbbk)$$
and $\Hom_\Bbbk (M, \Bbbk)$ means the graded dual,
$$\Hom_\Bbbk (M,\Bbbk) := \oplus_d \hom_\Bbbk(M_d,\Bbbk).$$

\begin{center}
{\bf Acknowledgement}
\end{center}
The authors thank Sangjin Lee for helpful discussions on related work concerning Floer theory for plumbings. We are grateful to Matt Booth for valuable discussions on Koszul duality. The third-named author thanks Yujiro Kawamata for helpful conversations and for pointing out related work of Toda.

The work of the first-named author is supported by the National Research Foundation of Korea (NRF) grant funded by the Korea government (MSIT) (RS-2025-00517727, 2020R1A5A1016126). The work of the third-named author is supported by the Institute for Basic Science (IBS-R003-D1).

The authors used AI-assisted tools solely to improve the language and exposition of this
manuscript. These tools were not used to develop the mathematical ideas, results, or proofs.
The authors take full responsibility for the content.

\section{Review on the extended localized mirror and the mirror functor}\label{sec:prelim}

In \cite{CHL21}, deformation theory of the Lagrangian Floer complex of a compact Lagrangian was used to construct a noncommutative mirror, referred to as the \emph{localized mirror} in the sense that it captures the mirror geometry locally around the object corresponding to the chosen compact Lagrangian. In this section, we review the mirror construction for higher-rank trivial vector bundles. In particular, this is the case when each component of the domain of the immersion is simply-connected. The Floer theory involving higher-rank flat vector bundles with nontrivial holonomy will be explained in Section \ref{sec:higher}. See also \cite[Section 3.1]{Kon17} and \cite[Section 3]{LT26}.

\subsection{Noncommutative local mirror arising from a Lagrangian immersion}\label{subsec:reviewlocmir}

Let $M$ be a symplectic manifold and let $\bL \subset M$ be a spin oriented Lagrangian immersion. Let $(\bL, \mathcal{E})$ be a Lagrangian brane, where $ \mathcal{E}$ is a trivial vector bundle over the domain of the immersion $\bL$.
We define its formal deformation space as follows, which is a natural generalization of \cite{CHL21} to include higher-rank bundles. Including higher-rank bundles is important for the construction of creation and annihilation operators in representation theory as explained in \cite{LT26}.
\begin{enumerate}
	\item 	
	Associate a quiver $Q$ to $\CF^1(\bL,\bL)$. Each connected component $L_v$ of (the domain of the immersion) $\bL=\oplus_v L_v$ is associated with a vertex $v$. For each generator $X \in \CF^1(\bL,\bL)$, let $t(X)$ and $h(X)$ denote its tail and head, respectively, i.e., $X \in \CF^1 (\bL_{t(X)},\bL_{h(X)})$; we associate to $X$ an arrow $x$ from $t(X)$ to $h(X)$.
	\item Let $\C Q$ be the free path algebra, equipped with the path-length filtration. We take the completion of $\C Q$ with respect to this filtration, and denote it by $\widehat{\C Q}$.\footnote{One may instead work with convergent power series over the Novikov field $\Lambda$ as in \cite{CHL21}, rather than formal power series. However, for our purposes, it suffices to work over $\C$.} 
	\item 
	Extend the Fukaya algebra $\CF((\mathbb{L}, \mathcal{E}),(\mathbb{L}, \mathcal{E}))$ over the path algebra $\widehat{\C Q}$ to obtain a noncommutative $A_\infty$-algebra  $\widehat{\C Q}\hat{\otimes}_{\C^{\oplus}} \CF((\mathbb{L}, \mathcal{E}),(\mathbb{L}, \mathcal{E})).$ \footnote{If we fix the finite-dimensional Floer model of $\CF((\mathbb{L}, \mathcal{E}),(\mathbb{L}, \mathcal{E}))$, we can use the ordinary tensor product.}
	Here, the unit of $\CF((\mathbb{L}, \mathcal{E}),(\mathbb{L}, \mathcal{E}))$ is $1_{\mathbb{L}} = \sum 1_{L_v}$, and $\C^{\oplus} \subset \widehat{\C Q}$ denotes $\bigoplus_v \C \cdot e_v$, where $e_v$ are the trivial paths (idempotents) at corresponding vertices of $Q$.  The fibered tensor product forces an element $a \otimes X$ to be non-zero only when the tail of $a$ corresponds to the source of $X$. Using flat trivializations of $\mathcal{E}|_{L_i}$, the $A_\infty$-operations are defined by
	\begin{equation} \label{eq:mk}
		m_k (f_1 X_1,\ldots,f_k X_k) := f_k \ldots f_1 \, m_k (X_1,\ldots,X_k) \footnote{For a nontrivial flat vector bundle, this formula must be modified by inserting the ordered parallel-transport maps along the boundary segments of the corresponding holomorphic discs; see Section \ref{sec:higher}. }
	\end{equation} 
 	where $X_l \in \CF(\bL)$  and $f_l \in \widehat{\C Q} \hat{\otimes}_{\C^{\oplus}} \Hom( \mathcal{E}|_{\bL_{t(X_l)}},  \mathcal{E}|_{\bL_{h(X_l)}})$. Here and throughout, arrows are composed from right to left, i.e. $t(f_{l+1})=h(f_l)$, so the product $f_k \ldots f_1$ matches the left-to-right ordering of the inputs $X_1,\ldots, X_k.$ Moreover, we use restriction notation to denote the fibers of the flat vector bundles $\mathcal{E}$.

	\item Extend the formalism of bounding cochains of \cite{FOOO09} over $\widehat{\C Q}$, that is, we take 
	\begin{equation} \label{eq:b}
		\mathbf{b} = \sum_l b_l B_l
	\end{equation}
	where $B_l$ are the generators of $\CF^1(\bL)$, and $b_l$ are matrices whose entries are the corresponding arrows in $Q$.  Then define the deformed $A_\infty$-structure $m_k^{\mathbf{b}}$ as in \cite{FOOO09}, 
	\begin{equation}\label{eqn:mkbbb}
	m_k^{\mathbf{b},\cdots,\mathbf{b}} (X_1, \cdots, X_k) := \sum_l m_l (\mathbf{b}, \cdots, \mathbf{b}, X_1 ,\mathbf{b} \cdots,\mathbf{b}, X_k, \mathbf{b}, \cdots, \mathbf{b})
	\end{equation}
	which can be further expanded via Equation \eqref{eq:mk}.
	\item Take the quotient of the quiver algebra by the two-sided ideal $R$ generated by coefficients of the obstruction term $m_0^{b}$, so that $m_0^{\mathbf{b}} = W\cdot 1_\bL$ holds over 
	$$\cA_{(\bL, \mathcal{E})} := \widehat{Rep(Q, \mathcal{E})}/\overline{R},$$ where $\widehat{Rep(Q, \mathcal{E})}$ is the completed tensor algebra $\hat{T}_{Rep^0(Q, \mathcal{E})} Rep^1(Q, \mathcal{E})$ and $\overline{R}$ denotes the closure of $R$ in the tensor-length topology. Here $Rep^0(Q, \mathcal{E}):= \CF^0((\bL,\mathcal{E}),(\bL,\mathcal{E})) = \oplus_i \C  e_i \otimes_{\C} \Hom(\mathcal{E}|_{\bL_i},\mathcal{E}|_{\bL_i})$ and 
	$$Rep^1(Q, \mathcal{E}):= \bigoplus_{X_l} \left( \C x_l \otimes_{\C^{\oplus}} \Hom( \mathcal{E}|_{\bL_{t(X_l)}},  \mathcal{E}|_{\bL_{h(X_l)}}) \right)$$ where $x_l$ is the arrow associated to $X_l$. In other words, $\widehat{Rep(Q, \mathcal{E})}$ is a higher-rank quiver algebra generated by matrices whose entries are paths. The product of two such matrices is zero unless the paths are concatenable. $\cA_{(\bL, \mathcal{E})}$ is the space of noncommutative weakly unobstructed deformations of the Lagrangian brane $(\bL, \mathcal{E})$.
	We call $(\cA_{(\bL, \mathcal{E})},W)$ the noncommutative local mirror of $X$ probed by $(\bL, \mathcal{E})$.
	\item For simplicity, we denote $(\bL, \mathcal{E},\mathbf{b})$ by $(\bL,\mathbf{b}).$ The morphism spaces between $(\bL,\mathbf{b})$ and $L$ are enlarged to be $\CF((\bL,\mathbf{b}),L) := \cA_{(\bL, \mathcal{E})}\hat{\otimes}_{\C^{\oplus}} \CF((\mathbb{L}, \mathcal{E}),L)$ (and similarly for $\CF(L,(\bL,\mathbf{b}))$), where $L$ is an object in $\mathrm{Fuk}(M)$. Moreover, the differential is defined in a similar way as \eqref{eq:mk} in Step 3. More precisely, the differential on $\CF((\bL,\mathbf{b}),L)$ can be written as
	\begin{equation*}
		\begin{array}{rcl}
			m_1^{\mathbf{b},0} (X) &=& \sum_{l \geq 0} m_{l+1} (\underbrace{\mathbf{b}, \cdots, \mathbf{b}}_{l}, X) \\
			&=& \sum_{l \geq 0} \sum_{i_1,\cdots, i_l} m_{l+1} (b_{i_1} B_{i_1}, \cdots, b_{i_l} B_{i_l}, X) \\
			&=& \sum_{l \geq 0} \sum_{i_1,\cdots, i_l} b_{i_l} \cdots b_{i_1} \otimes m_{l+1} ( B_{i_1}, \cdots,  B_{i_l}, X)
		\end{array}
	\end{equation*}
	for a Floer generator $X \in L \cap \mathbb{L}$. Note that the $m_k$-operations on the right-hand side are defined consistently with \eqref{eq:mk}. Moreover, there is a natural bimodule complex associated with $(\bL,\mathbf{b})$, which we denote by
    $$\CF((\bL,\mathbf{b}),(\bL,\mathbf{b})):=  \cA_{(\bL, \mathcal{E})}\hat{\otimes}_{\C^{\oplus}} \CF((\mathbb{L}, \mathcal{E}),(\bL, \mathcal{E})) \hat{\otimes}_{\C^{\oplus}} \cA_{(\bL, \mathcal{E})}, $$ 
    whose differential\footnote{Although one may attempt to define higher operations analogously, these operations do not, in general, endow $\CF((\bL,\mathbf b),(\bL,\mathbf b))$ with an $A_\infty$-structure because of the noncommutativity of the coefficients.} is given by \begin{equation*}
    	\begin{array}{rcl}
    		m_1^{\mathbf{b},\mathbf{b}} (X) &=& \sum_{l,k \geq 0} m_{k+l+1} (\underbrace{\mathbf{b}, \cdots, \mathbf{b}}_{l}, X,\underbrace{\mathbf{b}, \cdots, \mathbf{b}}_{k}) \\
    		&=& \sum_{l,k \geq 0} \sum_{i_1,\cdots, i_l,j_1,\ldots,j_k} m_{k+l+1} (b_{i_1} B_{i_1}, \cdots, b_{i_l} B_{i_l}, X,b_{j_1} B_{j_1}, \cdots, b_{j_k} B_{j_k}) \\
    		&=& \sum_{l,k \geq 0} \sum_{i_1,\cdots, i_l,j_1,\ldots,j_k} b_{i_l} \cdots b_{i_1} \otimes m_{k+l+1} ( B_{i_1}, \cdots,  B_{i_l}, X,B_{j_1}, \ldots, B_{j_k}) \otimes b_{j_k} \cdots b_{j_1}.
    	\end{array}
    \end{equation*}  More details can be found in Section 3.2 of \cite{LNT23}. 
\end{enumerate}

When $E$ is a trivial line bundle, we simply write $\cA_\bL$ for $\cA_{(\bL,E)}$. We have the following localized mirror functor.

\begin{thm}\cite[Thm 4.7]{CHL21}
	Consider the boundary deformation $\mathbf{b}$ of $(\bL, \mathcal{E})$. There exists a well-defined $A_\infty$-functor
	$$\cF^{(\bL,\mathbf{b})}: \Fuk(X) \to \mathrm{MF}(\cA_{(\bL, \mathcal{E})},W) \qquad L \mapsto \left(\CF((\bL,\mathbf{b}), L),m_1^{\mathbf{b},0} \right).$$
\end{thm}

\begin{remark}
$m_0^{\mathbf{b}}=W\cdot 1_\bL$ has degree $2$.  Thus if furthermore $\bL$ is graded and we work in the $\Z$-graded case, $W$ automatically vanishes and we instead obtain a functor, 
$$ \cF^{(\bL,\mathbf{b})}: \Fuk(X) \to \mathrm{Mod} (\cA_{(\bL, \mathcal{E})}) $$
where the right-hand side is the (dg) category of dg modules over $\cA_{(\bL, \mathcal{E})}$.
\end{remark}

Intuitively, the noncommutative local mirror $\cA_\bL$ can be understood via Strominger-Yau-Zaslow Conjecture \cite{SYZ96}, which predicts that the mirror space is constructed as the moduli space of (special) Lagrangians (see also \cite{Aur07}). Roughly, a Lagrangian $\bL$ corresponds to a point of the mirror, while its deformation space $\cA_\bL$ gives rise to a neighborhood of that point. Thus, $\cA_\bL$ is also referred to as the localized mirror or noncommutative deformation space. Moreover, $\cA_{(\bL, \mathcal{E})}$ provides a natural framework for describing the moduli space of sheaves on the localized mirrors.

\subsection{Commutative Maurer-Cartan solutions and their moduli spaces}
Instead of solving the Maurer–Cartan equations formally, one may restrict to commutative solutions. Equivalently, this amounts to evaluating the arrows and solving the resulting matrix equations. For simplicity, throughout this subsection we assume that the Maurer–Cartan expressions are well-defined over $\mathbb C$ and polynomial in the deformation variables. We therefore carry out the construction using the ordinary path algebra, and $\mathcal A_{\mathbb L}$ denotes its quotient by the Maurer–Cartan relations, without completion. This applies, in particular, to some of the exact plumbings considered below. We define the Maurer-Cartan space as follows:

\begin{defn}
	Let $\bL$ be a relatively-spin oriented graded Lagrangian immersion and $ \mathcal{E}$ be a trivial flat vector bundle of rank $\mathbf{v}$ over $\bL$.
	The set of Maurer-Cartan solutions $\mathrm{MC}(\mathbf{v})$ is defined to be
	$$\mathrm{MC}(\mathbf{v}) := \left\{b \in \CF^1((\bL, \mathcal{E}),(\bL, \mathcal{E})): \, m_0^b = 0\right\}$$
	where $\bL_i$ is the $i$-th component of the domain of $\bL$, and $v_i$ is the $i$-th component of $\mathbf{v}$, namely the rank of $ \mathcal{E}|_{\bL_i}$.
	
	We have the gauge group symmetry
	$\cG$ on $\mathrm{MC}(\bv)$, where $\cG = \prod_i GL(v_i)$ is a reductive group.
	The Maurer-Cartan space is defined to be the quotient stack 
	$$[\mathrm{MC}(\bv) / \cG]$$ 
	which parametrizes a family of Lagrangian branes $(E,b)$ supported on $\bL$.
	
	To be more geometric, we can fix a GIT stability condition $\zeta$ and define the GIT quotient 
	$$\mathcal{MC}_\zeta(\bv) := \mathrm{MC}(\bv) \sslash_\zeta \cG.$$ 
\end{defn}

If the trivial-flat assumption is dropped, one needs to take into account the holonomy of flat connections on $ \mathcal{E}$; we refer to \cite[Section 3.1]{LT26} for more details. 

By construction, the Maurer-Cartan equations are precisely the defining relations of the quiver algebra $\cA_{\bL}$ evaluated on matrices. Consequently, the Maurer-Cartan space admits the following representation-theoretic interpretation.

\begin{lemma}\cite[Prop. 3.10]{LT26}
	The Maurer-Cartan space $\mathcal{MC}_\zeta(\mathbf{v})$ is isomorphic to the moduli space of $\zeta$-semistable representations of the quiver algebra $\cA_\bL$ of dimension $\bv$.
\end{lemma}

\subsection{Extended noncommutative local mirror}
To encode the full deformation theory, one needs to incorporate Floer generators in positive degrees. The resulting enhanced deformation space is called the extended localized mirror, introduced in \cite[Chapter 9]{CHL21}. Here, we briefly recall its construction.

\begin{defn}[Extended Quiver]
	Let $\bL$ be a spin oriented graded Lagrangian immersion $\bL \subset M$. Take a Morse function on each $L_i$ and let $\{T_q\}$ be the set of critical points other than the maximum points.  The extended quiver $\tilde{Q}$ is defined to be a directed graph whose vertices are in one-to-one correspondence with the connected components ${L_i}$, and whose arrows, denoted by $\{x_e\}$ and $\{t_q\}$, are in one-to-one correspondence with the immersed sectors $\{X_e\}$ and critical points $\{T_q\}$.
	
	The head and tail of $x_e$ are defined as before from the immersed sector $X_e$. 
	For a critical point $T_q$ on $L_i$, the associated arrow $t_q$ is a loop at the vertex $i$, namely
	$h(t_q)=t(t_q)=i.$
\end{defn}

To incorporate these additional generators into the deformation theory, we treat the immersed sectors and non-maximal Morse critical points uniformly. 
Let $\{Y_\alpha\}$ denote the collection of all such generators, and let $\{y_\alpha\}$ be the corresponding arrows of the extended quiver $\tilde Q$.

Take the extended formal deformation parameter
$$
\tilde b=\sum_\alpha y_\alpha Y_\alpha.
$$ Analogous to Section \ref{subsec:reviewlocmir}, one extends the Fukaya algebra $\CF(\bL,\bL)$ over the completed path algebra $\widehat{\C\tilde Q}$ as in Equation~\eqref{eq:mk} by
	\begin{equation*} 
m_k(f_1 X_1,\ldots,f_k X_k)
=
(-1)^{\sum_{i<j} |f_j|\,(|X_i|-1)}
\, f_k \cdots f_1 \, m_k(X_1,\ldots,X_k),
	\end{equation*}
using the Koszul sign convention. Here, $| \cdot |$ denotes the degree in $\mathbb{Z}$-grading, where for dual variables, we set $|y_\alpha|:= 1- |Y_\alpha|$ which determines $|f_j|$ (see Definition \ref{defn:extlocmir} (1) below).  
Then the Maurer--Cartan equation of $(\bL,\tilde b)$ takes the form
$$
m_0^{\tilde b}=\sum_\alpha P_\alpha Y_\alpha.
$$

\begin{defn}\label{defn:extlocmir}
	The extended localized mirror associated to an immersed Lagrangian $\bL$ is a dg algebra $\tilde{\cA}_{\bL}:=(\widehat{\C\tilde Q},d)$ where
	\begin{enumerate}
		\item For each arrow $y$ corresponding to a Floer generator $Y$, we set
		\[
		\deg y = 1-\deg Y.
		\]
		Writing $J\subset \C\tilde Q$ for the arrow ideal, we take
		the path-length completion degreewise:
		\[
		\widehat{\C\tilde Q}
		:=
		\bigoplus_{j\in\mathbb Z}
		\varprojlim_N
		\left(\C\tilde Q/J^N\right)^j.
		\]
		Thus infinite series are allowed within each fixed degree, while each element has only finitely many nonzero homogeneous components.

		\item The differential $d$ is the continuous degree-one derivation on $\widehat{\C\tilde Q}$
		determined by
		\[
		d(a_\alpha y_\alpha)=a_\alpha P_\alpha, \qquad \forall a_\alpha \in \C
		\]
		and extended to all of $\widehat{\C\tilde Q}$ by the graded Leibniz rule.
	\end{enumerate}
\end{defn}

As before, we have the following extended localized mirror functor:
\begin{thm}\cite[Thm 9.16]{CHL21}
	Consider the extended deformation parameter $\tilde{b}$. There exists a well-defined $A_\infty$-functor $$\cF^{(\bL,\tilde{b})}: \Fuk(X) \to   \mathrm{Mod}_\dg(\tilde{\cA}_{\bL}) \qquad L \mapsto \left(\CF((\bL,\tilde{b}), L),m_1^{\tilde{b},0} \right),$$ 
	where $\mathrm{Mod}_\dg(\tilde{\cA}_{\bL})$ is the dg category of dg $\tilde{\cA}_{\bL}$-modules.
\end{thm}

As in the case of the (unextended) localized mirror, one can naturally associate a bimodule
    $$ CF((\bL,\tilde{b}),(\bL,\tilde{b})):=  \tilde{\cA}_{\bL} \hat{\otimes}_{\C^{\oplus}} \CF(\mathbb{L},\bL) \hat{\otimes}_{\C^{\oplus}} \tilde{\cA}_{\bL}, $$ 
to the dga $\tilde{\cA}{\bL}$. We will see that this serves as a bimodule resolution of the diagonal bimodule $\tilde{\cA}_{\bL}$. In fact, it is a projective resolution in an appropriate topological sense.

\begin{remark}\label{rmk:afincfl}
The dga $\tilde{\cA}_{\bL}$ is inherently local in nature, as it is always defined as a completion. When the $A_\infty$-operations on $\CF(\bL,\bL)$ satisfy a suitable finiteness condition (for instance, when $m_k \equiv 0$ for all sufficiently large $k$; see Definition \ref{def:finmkv}), then one can carry out an analogous construction using $\mathbb{C} \tilde{Q}$, resulting in a dga denoted by $\tilde{\cA}_{\bL,c}$. In this case, one can define
 $$ \CF_c((\bL,\tilde{b}),(\bL,\tilde{b})):=  \tilde{\cA}_{\bL,c} \otimes_{\C^{\oplus}} \CF(\mathbb{L},\bL) \otimes_{\C^{\oplus}} \tilde{\cA}_{\bL,c}, $$
 which is a genuine projective resolution of the diagonal bimodule $\tilde{\cA}_{\bL,c} $ up to degree shift. 
\end{remark}

\section{The extended localized mirror via Koszul duality}\label{sec:extmirkos}

Suppose $\bL$ is graded. The construction in Section \ref{sec:prelim} produces a dga  $\tilde{\mathcal{A}}_\bL$ as the Maurer-Cartan deformation space of $\bL$. Indeed, this local mirror $\tilde{\mathcal{A}}_\bL$ can be understood as a geometric manifestation of Koszul duality for abstract $A_\infty$-algebras. Adopting this perspective, we establish in this section several universal results concerning localized mirrors. Compactness allows us to use a finite-dimensional self-Floer model; the grading hypotheses then give the additional algebraic conditions in Assumption~\ref{assume:onv}. The resulting construction starts from this compact-brane data and does not require a global generator or a dual noncompact Lagrangian. Finite-dimensionality of the model does not imply that the higher $A_\infty$-operations truncate, which is why completions and the separate finite-type hypothesis play distinct roles below.

We shall see that this structure enables us to construct explicitly a free resolution of the localized mirror. As applications, we obtain closed-string invariants of the local mirror directly from the $A_\infty$-formalism of Floer theory of $\bL$ by suitably incorporating its boundary deformations.

\subsection{Koszul dual of an $A_\infty$-algebra}
Let us first introduce a general algebraic framework. Let $V$ be an  $A_\infty$-algebra over a semisimple ring $\Bbbk$. Throughout, we will work with the following assumption.

\begin{assumption}\label{assume:onv}
We assume that $V$ satisfies the following properties: $V$ is
\begin{itemize}
\item[-] a finite rank bimodule over a (semisimple) coefficient ring $\Bbbk$,
\item[-] $\mathbb{Z}$-graded and supported at nonnegative degrees,
\item[-] unital and $V_0 \cong \Bbbk \langle 1_V \rangle$.
\end{itemize}
\end{assumption}

In particular, we have a canonical augmentation $\varepsilon: V \to \Bbbk$, which makes $\Bbbk$ into a module over $V$.  We write $V_{>0}:=\ker \varepsilon$. The algebra $V$ serves as an algebraic model for $\CF(\bL,\bL)$ for a compact spin, graded, unobstructed Lagrangian $\bL$, possibly with multiple irreducible components, and has no nonpositive-degree immersed generators. More precisely, $V$ may be taken to be a minimal model for $\CF(\bL,\bL)$. 
 If $\bL$ consists of more than one component, we take $\Bbbk$ to be the semisimple ring $\Bbbk= \oplus_i \mathbb{C}$ with one copy of $\mathbb{C}$ for each component of $\bL$. 

Let $\bar{B}V := \oplus_{k\geq 0} \left(sV_{>0}\right)^{\otimes_\Bbbk k}$ denote the (reduced) bar construction of $V$ with respect to the augmentation $\varepsilon$, where $(sV_{>0})^\bullet := (V_{>0})^{\bullet+1}$. $\bar{B}V$ is the dg coalgebra that encodes the $A_\infty$-algebra structure of $V$ with the differential 
$$ \vec{Y}   \mapsto  \sum (-1)^{|\vec{Y}_1|'}\left( \vec{Y}_1 \otimes m(\vec{Y}_2) \otimes \vec{Y}_3 \right),$$
and the comultiplication 
$$\Delta : \bar{B} V \to \bar{B} V \otimes \bar{B} V \qquad \vec{Y} \to \sum  \vec{Y}_1 \otimes \vec{Y}_2.$$ 
Here, $|\vec{Y}|'$ denotes the (sum of) shifted degrees, and hence if $\vec{Y} = Y_1 \otimes \cdots \otimes Y_l$, then
$$|\vec{Y}|' =\sum_{i=1}^l (\deg Y_i -1).$$
Throughout Sections \ref{sec:extmirkos} and \ref{sec:locfunc}, tensor products of homogeneous maps use the Koszul convention
\[
(F\otimes G)(u\otimes v)=(-1)^{|G||u|}F(u)\otimes G(v),
\]
and the bar signs are computed using the shifted degree $|\cdot|'$. Thus, for example, the operator $1\otimes d$ on a tensor product already includes the sign obtained by moving the degree-one map $d$ past the preceding tensor factors.

The Koszul dual $V^!$ of  $V$ is defined to be the dga
\begin{equation}\label{eqn:defkosv}
V^!:=\mathrm{RHom}_V (\Bbbk,\Bbbk) \cong 
(\bar{B} V)^\sharp,
\end{equation}
where $\mathrm{RHom}_V$ is taken in the dg category  $\mathrm{Mod}_{A_\infty} (V)$ of $A_\infty$-modules over $V$ with pre-$A_\infty$ homomorphisms, and $(-)^\sharp$ is taking the linear graded dual. i.e.,
\begin{equation*}
 \left( \bar{B}V \right)^\sharp= \oplus_{d}  \hom_{\Bbbk} ((\bar{B}V)_d,\Bbbk).
\end{equation*} 
Here, $(\bar{B}V)_d$ is the degree $d$ homogeneous component of $ \bar{B}V$. The second equality in \eqref{eqn:defkosv} may be interpreted as replacing $\Bbbk$ in the first slot with its free resolution as a right $V$-module
\begin{equation}\label{eqn:barresol}
 \bar{B}V \otimes_\Bbbk V \to \Bbbk 
\end{equation}
which extends $\varepsilon : V \to \Bbbk$. The differential on $\bar{B} V \otimes_\Bbbk V$ is given by
$$ \vec{Y} \otimes X \mapsto  \sum  (-1)^{|\vec{Y}_1|'} \left( \vec{Y}_1 \otimes m(\vec{Y}_2) \otimes \vec{Y}_3 \right)  \otimes X + \sum  (-1)^{|\vec{Y}_1|'} \vec{Y}_1 \otimes  m(\vec{Y}_2,  X).$$
The fact that \eqref{eqn:barresol} is a quasi-isomorphism can be proved by constructing a contracting homotopy, namely by inserting the unit class into the (V)-factor. We prove this for the analogously defined bimodule resolution in Lemma \ref{lem:conthtpy}, using essentially the same homotopy.

Observe that $V^!$ has a canonical augmentation induced from the coaugmentation
$ \Bbbk \hookrightarrow \bar{B}V$.

\begin{remark}
When $V= \CF(\bL,\bL)$, then $(\bar{B} V)^\sharp$ agrees with the extended localized mirror $\tilde{\mathcal{A}}_\mathbb{L}$ from $\mathbb{L}$. (See \cite{Ho}.) $\mathrm{Mod}_{A_\infty} (V)$ is equivalent to the subcategory of $\Fuk(X)$ generated by $\mathbb{L}$ for $V= \CF(\bL,\bL)$.
\end{remark}

On the other hand, taking the graded dual to the homotopy equivalence $V \otimes_\Bbbk \bar{B}V \cong \Bbbk$ (left module version of \eqref{eqn:barresol}) gives
$$ \Hom_\Bbbk (V \otimes_\Bbbk \bar{B}V , \Bbbk) 
= (\bar{B}V)^\sharp \otimes V^\sharp = V^! \otimes V^\sharp$$
since $V$ is finite-dimensional (we would need a completed tensor product in general).  Here, $\Bbbk \to V \otimes_\Bbbk \bar{B}V$ is the inclusion of $\Bbbk$ into the unit component of $V$ (tensored with the length $0$ part of $\bar{B} V$), and it dualizes to the canonical augmentation $V^! (\cong V^! \otimes V_0^\sharp) \to \Bbbk$. 
This gives us a (finite-length) free resolution $V^! \otimes V^\sharp$ of $\Bbbk$, now as a module over the dga $V^!$ (via the above-mentioned canonical augmentation). We record this observation as a separate statement for later use.

\begin{lemma}\label{lem:leftrightresol}
For $V$ satisfying Assumption \ref{assume:onv}, $V^! \otimes V^\sharp$ is quasi-isomorphic to $\Bbbk$ as a left $V^!$ module.
\end{lemma}

One can prove the following involutivity of $(-)^!$ for $V$ from Lemma \ref{lem:leftrightresol}.

\begin{lemma}\label{lem:vdoubledual}
Let $V$ be a finite-dimensional unital $A_\infty$-algebra satisfying Assumption \ref{assume:onv}. Then $(V^!)^! \cong V$.
\end{lemma}

\begin{proof}
By definition, $(V^!)^! $ is given by
\begin{equation*}
\begin{array}{rcl}
(V^!)^!  &=& \mathrm{RHom}_{V^!} (\Bbbk, \Bbbk )  \\
 & =& \mathrm{RHom}_{V^!} (V^! \otimes_\Bbbk  V^\sharp, \Bbbk )  \\
 &\cong&  \Hom_{\Bbbk} ( V^\sharp, \Bbbk ) \cong V.
 \end{array}
 \end{equation*}
We used $V^! \otimes_\Bbbk V^\sharp \cong \Bbbk$ (Lemma \ref{lem:leftrightresol}) in the middle, which is applicable since $\mathrm{RHom}$ (the derived hom) only depends on the quasi-isomorphism type of objects. 
\end{proof}

\subsubsection{Some variants of $V^!$}\label{subsub:compl}
While the original definition $V^! = (\bar{B}V)^\sharp$ uses the graded dual, one can also consider alternatives such as the compactly supported dual which results in
\begin{equation}\label{eqn:v!c}
V^!_c := \{ \sigma : \bar{B} V \to \Bbbk \in V^!_f : \mathrm{supp} (\sigma) \,\,\textnormal{has a finite rank over} \,\, \Bbbk\}.
\end{equation}
Since $V$ is of finite rank over $\Bbbk$, $V^!_c$ can be also identified as
\begin{equation}\label{eqn:vcvc}
V^!_c:= \bigoplus_{k \geq 0}  \left((s V_{>0})^{\otimes k} \right)^\sharp.
\end{equation}
Note however that $V^!_c$ is well-defined as a dga only when the $A_\infty$-operations on $V$ have a certain finiteness property as in Definition \ref{def:finmkv} (which is not functorial).
Clearly, we have inclusions
$$ V^!_c \subset V^! $$
which arises from a completion. Let $I \subset V^!_c$ denote the 2-sided ideal generated by degree $0$ elements in $V^!_c$. Then it is not difficult to see that
$$ V^! = \bigoplus_d \varprojlim_{n} (V^!_c /I^n)^d.$$

In our geometric setting, where $V=\CF(\bL,\bL)$, the dga $V^!_c$ corresponds to $\tilde{\cA}_{\bL,c} $ defined in Remark \ref{rmk:afincfl}.

\subsubsection{Coordinate-description of the Koszul dual}\label{subsec:coordkos}
We give a more explicit description of \eqref{eqn:grdualres} in terms of (noncommutative) coordinates. 
Let us choose a basis $\{X_1,\cdots,X_N\}$ of $V_{>0}$ so that $V \cong \Bbbk \oplus \Bbbk \langle X_1,\cdots,X_N \rangle$. Then, by definition, $V^!$ is generated by noncommutative homogeneous power series in the dual basis $\{x_1,\cdots,x_N\}$.\footnote{Similarly, $V^!_c$ consists of polynomials in these variables whereas $V^!_f$ allows a series which is supported over infinitely many different degrees.} More specifically, its element satisfies $x_i (X_j) = \delta_{ij} \pi_{a_j}$ where $X_j \in V \cdot \pi_{a_j}$, which is natural since we are dealing with right module maps.
The $\Bbbk$-bimodule structure on $V^\sharp$ is given by
$$ (a\cdot f \cdot b) (X) =a \cdot f( b \cdot X)$$
for $f: V \to \Bbbk$. As an element of $V^!$, the degree of $x_i$ is $\deg_{V^!} (x_i)= 1 - \deg_V (X_i)$.

When $V$ is equipped with a Poincar\'e pairing, we further require that the basis $\{1_V, X_1,\cdots,X_N\}$ makes the pairing into the \emph{standard} form. In other words, one can rearrange elements in the basis to make it into the form
$$\{1_V, X_1,\cdots,X_N\}=\{1_V, X_1,\cdots,X_{N'}, \bar{X}_1,\cdots, \bar{X}_{N'}, [\mathrm{pt}_V]\}$$ 
such that $\langle X_i, \bar{X}_j \rangle = \delta_{ij}$, and $ [\mathrm{pt}_V]$ is a generator of the top degree component of $V$ satisfying $\langle 1_V, [\mathrm{pt}_V] \rangle =1$. In particular, one has an isomorphism
$$ PD: V \stackrel{\cong}{\longrightarrow} V^\sharp \qquad X_i \mapsto \langle X_i, - \rangle $$
and analogously defined for $1_V$ and $[\mathrm{pt}_V]$. In particular, we have $x_i = PD (\bar{X}_i)$. Note however that this is not a graded morphism, for e.g., 
\begin{equation}\label{eqn:graddual}
\deg_{V^\sharp} (x_i) := - \deg_V (X_i)
\end{equation}
whereas $\deg_V (\bar{X}_i) = n- \deg_V (X_i)$. We set $1_V^\sharp:=PD([\mathrm{pt}_V])$ so that $\{1_V^\sharp, x_1,\cdots, x_N\}$ gives the basis of $V^\sharp$ dual to $\{1_V, X_1,\cdots,X_N\}$.

\begin{remark}
If the $A_\infty$-structure on $V$ has a cyclic symmetry with respect to the pairing $\langle -,- \rangle$, we have
$$1 =  \langle m_2 (1_V,X_i), \bar{X}_i \rangle = \langle m_2 (X_i, \bar{X}_i ), 1_V \rangle,$$
and hence $m_2 (X_i,\bar{X}_i) = [\mathrm{pt}_V]$. Similarly, $m_2 (\bar{X}_i,X_i) = (-1)^\ast [\mathrm{pt}_V]$.
\end{remark}

On the chain-level, $V^!$ is nothing but the tensor algebra generated by $x_1,\cdots, x_N$. The differential $d$ can be explicitly written by
$$ \sum_k m_k (\tilde{b} ,\cdots, \tilde{b})= \sum_{k} \sum_{i_1,\cdots,i_k} m_k (x_{i_1} X_{i_1}, \cdots , x_{i_k} X_{i_k}) = \sum_{i=1}^N (d x_i ) X_i$$
where $\tilde{b} = \sum x_i X_i$ and the $m_k$-operations can be expanded as in Section \ref{sec:prelim}. Note that under our assumption $V^!$ is nonpositively graded, and its degree zero component is generated by dual elements to degree $1$ generators of $V$. It is not difficult to see that the zero cohomology $H^0 (V^!)$ agrees with the Maurer-Cartan algebra $A_{(\bL,E)}$ when $V = \CF(\bL,\bL)$ and $E$ is the trivial line bundle.

\subsection{The `Koszul' resolution of $V^!$ as a bimodule over itself}

We next look for a finite-length free resolution of $V^!$ as a bimodule over itself which will be particularly useful when studying the ``local version" of closed-string mirror symmetry.
Observe first that the (unital) bimodule structure on the diagonal bimodule $V$ gives rise to a dg bicomodule $\bar{B}V \otimes_\Bbbk V \otimes_\Bbbk \bar{B}V$ over the coalgebra $\bar{B} V$ equipped with the differential
$$ \delta:  \bar{B}V \otimes_\Bbbk V \otimes_\Bbbk \bar{B}V \to  \bar{B}V \otimes_\Bbbk V \otimes_\Bbbk \bar{B}V$$
defined as
\begin{equation}\label{eqn:formuladelta}
\begin{array}{rcl}
\delta(\vec{Y} \otimes X \otimes \vec{Z}) &=& \sum (-1)^{|\vec{Y}_1|'}\left( \vec{Y}_1 \otimes m(\vec{Y}_2) \otimes \vec{Y}_3 \right)  \otimes X \otimes \vec{Z} \\
&&+ \sum  (-1)^{|\vec{Y}_1|'} \vec{Y}_1 \otimes  m(\vec{Y}_2,  X, \vec{Z}_1) \otimes \vec{Z}_2 \\
&& + \sum  (-1)^{|\vec{Y}|' +|X|'+|\vec{Z}_1|'} \vec{Y} \otimes  X \otimes \left( \vec{Z}_1  \otimes m(\vec{Z}_2)  \otimes \vec{Z}_3 \right). 
\end{array}
\end{equation}
Due to shift of grading in $\bar{B}V$, one can easily see that $\delta$ is of degree $1$.
Notice that the $A_\infty$-relation (for bimodules) is equivalent to $\delta^2=0$. 
Since we allow length 0 tensors, $\bar{B}V \otimes_\Bbbk V \otimes_\Bbbk \bar{B}V $ contains $V$, $\bar{B} V \otimes_\Bbbk V$ and $V \otimes_\Bbbk \bar{B} V$.

Let $\left(\bar{B}V \otimes_\Bbbk V \otimes_\Bbbk \bar{B}V\right)^\sharp$ denote its graded dual,
\begin{equation}\label{eqn:grdualres}
 \left( \bar{B}V \otimes_\Bbbk V \otimes_\Bbbk \bar{B}V\right)^\sharp:= \oplus_{d}  \hom_{\Bbbk} (\mathcal{V}_d,\Bbbk)
\end{equation}
where $\mathcal{V}_d$ is the degree $d$ homogeneous component of $ \bar{B}V \otimes_\Bbbk V \otimes_\Bbbk \bar{B}V$ given as
$$ \mathcal{V}_d = \oplus_{d_1+d_2+d_3=d} (\bar{B}V)_{d_1} \otimes_\Bbbk V_{d_2} \otimes_\Bbbk (\bar{B}V)_{d_3}.$$ 
$\left(\bar{B}V \otimes_\Bbbk V \otimes_\Bbbk \bar{B}V\right)^\sharp$ can be viewed as the dg bimodule over the dg algebra $V^! =  (\bar{B} V)^\sharp$, Koszul dual to the diagonal bimodule over $V$. 


Now we describe the dual of the tensors in terms of the basis chosen above. 
By definition, the degree $-d$ component of \eqref{eqn:grdualres} takes the form of
\begin{equation}\label{eqn:dualprod}
\hom_\Bbbk( \mathcal{V}_d, \Bbbk) =  \prod_{d_1+d_2+d_3=d} \hom_\Bbbk \left((\bar{B}V)_{d_1} \otimes_\Bbbk V_{d_2} \otimes_\Bbbk (\bar{B}V)_{d_3},\Bbbk \right).
\end{equation}
Therefore we have
\begin{equation}\label{kdudiag}
\left( \bar{B}V \otimes_\Bbbk V \otimes_\Bbbk \bar{B}V\right)^\sharp \cong \widehat{V^! \otimes_\Bbbk V^\sharp \otimes_\Bbbk V^!},
\end{equation}
where the right-hand side denotes the completion of $V^! \otimes_\Bbbk V^\sharp \otimes_\Bbbk V^!$ by the ideal $I \otimes_\Bbbk V^\sharp \otimes_\Bbbk I$. Alternatively, it can be obtained as the completion of $V^!_f \otimes_\Bbbk V^\sharp \otimes_\Bbbk V^!_f$ by the ideal 
$$I \otimes_\Bbbk V^\sharp \otimes_\Bbbk  V^!_f + V^!_f \otimes_\Bbbk V^\sharp \otimes_\Bbbk  I.$$
The completion accounts for the fact that a functional on $\mathcal{V}_d$ can be supported over infinitely many different components in \eqref{eqn:dualprod}. It is generated by elements of the form $f (x) \otimes x_i \otimes g (x)$ or $f (x) \otimes 1_V^\sharp \otimes g (x)$, where $f$ and $g$ are power series, allowing for infinite sums in each homogeneous piece. For example, one can have
$ \sum_{i=0}^{\infty} x^i \otimes y \otimes x^i$ for $\deg(x)=0$.

With respect to these bases, the differential $d$ on \eqref{kdudiag} can be described as follows. This still has degree $1$, in our convention on the grading of the dual \eqref{eqn:graddual}.  Take elements $E:=\vec{Y} \otimes X \otimes \vec{Z} \in \bar{B}V \otimes_\Bbbk V \otimes_\Bbbk \bar{B}V $ and $x_{\vec{a}} \otimes x_b \otimes x_{\vec{c}} = x_{a_1} \cdots x_{a_k} \otimes x_b \otimes x_{c_1} \cdots x_{c_l} \in  \widehat{V^! \otimes_\Bbbk V^\sharp \otimes_\Bbbk V^!}$,
where $X \in \{X_1,\cdots,X_N\}$ and $\vec{Y}, \vec{Z}$ are tensors in $X_i$'s. From \eqref{eqn:formuladelta},
\begin{equation}\label{eqn:dbefpd}
\begin{array}{rl}
& d \left(x_{a_1} \cdots x_{a_k} \otimes x_b \otimes x_{c_1} \cdots x_{c_l} \right) (E)\\
= & \left(x_{a_1} \cdots x_{a_k} \otimes x_b \otimes x_{c_1} \cdots x_{c_l} \right) \left(\delta( \vec{Y} \otimes X \otimes \vec{Z})\right) \\
 =& \sum (-1)^{|\vec{Y}_1|'}(x_{c_1} \cdots x_{c_l} ) \left( \vec{Y}_1 \otimes m(\vec{Y}_2) \otimes \vec{Y}_3 \right)  \cdot x_b (X) \cdot ( x_{a_1} \cdots x_{a_k}) ( \vec{Z}) \\
 &+\sum (-1)^{|\vec{Y}|'+|X|'+|\vec{Z}_1|'}(x_{c_1} \cdots x_{c_l} ) (\vec{Y})  \cdot x_b (X) \cdot ( x_{a_1} \cdots x_{a_k}) \left( \vec{Z}_1 \otimes m(\vec{Z}_2) \otimes \vec{Z}_3 \right)\\
 &+\sum (-1)^{|\vec{Y}_1|'}(x_{c_1} \cdots x_{c_l} ) (\vec{Y}_1)  \cdot x_b \left(m(\vec{Y}_2,X,\vec{Z}_1)\right) \cdot ( x_{a_1} \cdots x_{a_k})(\vec{Z}_2).
 \end{array}
 \end{equation}

Using the identification $V^\sharp \cong V[n]$ via the Poincar\'e pairing on $V$, \eqref{kdudiag} becomes
\begin{equation}\label{eqn:cflblbkoszul}
\left( \bar{B}V \otimes_\Bbbk V \otimes_\Bbbk \bar{B}V\right)^\sharp \cong  \widehat{V^! \otimes_\Bbbk V \otimes_\Bbbk V^!} [n]
\end{equation}
where the right-hand side is defined analogously to \eqref{kdudiag}, except with $V^\sharp$ replaced by $V$.
Note that this is precisely $CF^\ast ((\mathbb{L}, \tilde{b}), (\mathbb{L}, \tilde{b}))$ in Section \ref{sec:prelim} up to degree shift, when $V= CF^\ast (\mathbb{L},\mathbb{L})$. Moreover,

\begin{lemma}
If the $A_\infty$-structure on $V$ is cyclic with respect to $\langle -,- \rangle$, then the differential on $V^! \hat{\otimes}_\Bbbk V^\sharp \hat{\otimes} V^! [n]$ induced from $d$ \eqref{eqn:dbefpd} can be written as
 $$ d= d_{V^!} \otimes 1 \otimes 1  + 1 \otimes 1 \otimes d_{V^!} +  1 \otimes m_1^{\tilde{b},\tilde{b}} \otimes 1 .$$
 where $m_1^{\tilde{b},\tilde{b}}$ is defined by the same formula as in \eqref{eqn:mkbbb} (applying to $\tilde{b}$).
\end{lemma}
 \begin{proof}
 It is straightforward to see that the first two terms of \eqref{eqn:dbefpd} precisely match  $d_{V^!} \otimes 1 \otimes 1+ 1 \otimes 1 \otimes d_{V^!} $. (This part is indeed irrelevant to the identification $V^\sharp \cong V[n]$.)
We need to compare the remaining term, that is, the map 
 $$\vec{Y} \otimes X \otimes \vec{Z}  \mapsto \sum (-1)^{|\vec{Y}_1|'}(x_{c_1} \cdots x_{c_l} ) (\vec{Y}_1)  \cdot x_b \left(m(\vec{Y}_2,X,\vec{Z}_1)\right) \cdot ( x_{a_1} \cdots x_{a_k})(\vec{Z}_2)$$
belonging to $ V^! \hat{\otimes}_\Bbbk V^\sharp \hat{\otimes} V^!$  with an element of  $V^! \hat{\otimes}_\Bbbk V \hat{\otimes} V^![n]$
 $$  x_{a_1} \cdots x_{a_k} \left( \sum m(\tilde{b},\cdots,\tilde{b}, X_b, \tilde{b},\cdots,\tilde{b}) \right)  x_{c_1} \cdots x_{c_l} $$
where $X_b$ is defined by $x_b = \langle X_b,- \rangle$. The latter can be more explicitly written as
\begin{equation}
\begin{array}{l}
\displaystyle\sum_{i,j} x_{a_1} \cdots x_{a_k} m \left(x_{i_1} X_{i_1}, \ldots, x_{i_q} X_{i_q}, X_b, x_{j_1} X_{j_1}, \ldots, x_{j_r} X_{j_r}  \right)  x_{c_1} \cdots x_{c_l}\\
=\displaystyle\sum_{i,j} (-1)^{\epsilon(\vec{i},X_b,\vec{j})}x_{a_1} \cdots x_{a_k} x_{j_r} \cdots x_{j_1} \otimes   m(X_{i_1} ,\cdots, X_{i_q}, X_b, X_{j_1}, \cdots, X_{j_r}) \otimes  x_{i_q} \cdots x_{i_1} x_{c_1} \cdots x_{c_l},
\end{array}
\end{equation}
where $\epsilon(\vec{i},X_b,\vec{j})$ is the Koszul sign from Equation~\eqref{eq:mk}, namely the sign obtained by moving each coefficient variable past the earlier shifted inputs.
Hence its evaluation at $\vec{Y} \otimes X \otimes \vec{Z}$ gives
\begin{equation}
\begin{array}{l}
\displaystyle\sum_{i,j, Y_1, Z_1} (-1)^{\epsilon(\vec{i},X_b,\vec{j})}x_{\vec{c}} (\vec{Y}_1) x_{\vec{i}} ( \vec{Y}_2) \otimes \langle m(X_{i_1} ,\cdots, X_{i_q}, X_b, X_{j_1}, \cdots, X_{j_r}), X \rangle \otimes x_{\vec{j}} (\vec{Z}_1) x_{\vec{a}} (\vec{Z}_2) \\
=\displaystyle\sum_{i,j, Y_1, Z_1} (-1)^{\epsilon(\vec{i},X_b,\vec{j})+\chi(\vec{i},X_b,\vec{j},X)}x_{\vec{c}} (\vec{Y}_1) x_{\vec{i}} ( \vec{Y}_2) \otimes \langle m(X_{j_1}, \cdots, X_{j_r}, X, X_{i_1} ,\cdots, X_{i_q}), X_b \rangle \otimes  x_{\vec{j}} (\vec{Z}_1) x_{\vec{a}} (\vec{Z}_2) \\
=\displaystyle\sum_{\substack{Y_1, Z_1 \\ \vec{i}  = \vec{y}_2, \vec{j} = \vec{z}_1}} (-1)^{\epsilon(\vec{i},X_b,\vec{j})+\chi(\vec{i},X_b,\vec{j},X)}x_{\vec{c}} (\vec{Y}_1) x_{\vec{i}} ( \vec{Y}_2) \otimes \langle m(X_{j_1}, \cdots, X_{j_r}, X, X_{i_1} ,\cdots, X_{i_q}), X_b \rangle \otimes x_{\vec{j}} (\vec{Z}_1) x_{\vec{a}} (\vec{Z}_2) \\
=\displaystyle\sum_{Y_1, Z_1}  (-1)^{|\vec{Y}_1|'}x_{\vec{c}} (\vec{Y}_1) \otimes \langle m(\vec{Y}_2, X, \vec{Z}_1), X_b \rangle \otimes   x_{\vec{a}}(\vec{Z}_2) =  \displaystyle\sum_{Y_1, Z_1}  (-1)^{|\vec{Y}_1|'}x_{\vec{c}} (\vec{Y}_1) \otimes x_b \left( m(\vec{Y}_2, X, \vec{Z}_1) \right) \otimes   x_{\vec{a}} (\vec{Z}_2) 
\end{array}
\end{equation}
Here $\chi(\vec{i},X_b,\vec{j},X)$ is the cyclic Koszul sign for rotating the cyclic word $(\vec{i},X_b,\vec{j},X)$ to $(\vec{j},X,\vec{i},X_b)$. With the shifted-degree cyclicity convention, the product $(-1)^{\epsilon+\chi}$ becomes the middle sign $(-1)^{|\vec{Y}_1|'}$ after imposing $\vec{i}=\vec{y}_2$ and $\vec{j}=\vec{z}_1$.
The first equality uses the cyclic symmetry of the Poincar\'e pairing with this sign. The second and third follow since $\vec{x}_j (\vec{Y}_2)$ and $\vec{x}_i (\vec{Z}_1)$ are nonzero if and only if $\vec{x}_j$ and $\vec{x}_i$ are precisely dual to $\vec{Y}_2$ and $\vec{Z}_1$ (in which case $ \vec{X}_j = \vec{Y}_2$ and $\vec{X}_i = \vec{Z}_1$). 
Thus, after the identification
 $$ \widehat{ V^!  \otimes_\Bbbk V^\sharp  \otimes_\Bbbk  V^!} \cong \widehat{ V^! \otimes_\Bbbk V[n] \otimes_\Bbbk V^!}$$
 using the Poincar\'e pairing, the differential $d$ reads
 $$ d= d_{V^!} \otimes 1 \otimes 1 +  1 \otimes 1 \otimes d_{V^!} + 1 \otimes m_1^{\tilde{b},\tilde{b}} \otimes 1.$$
\end{proof}

In particular, we see that the completed tensor product in the graded dual is crucial for well-definedness of $d$, as $m_1^{\tilde{b}, \tilde{b}}$ is in general an infinite sum.

%

We claim that  $V^! \otimes_\Bbbk V^\sharp \otimes_\Bbbk V^!$ gives a resolution of $V^!$ as a bimodule over itself as long as it is well-defined (i.e., when it becomes a subcomplex). Note that $V^! \otimes_\Bbbk V^\sharp \otimes_\Bbbk V^!$ is automatically a free $V^!$-bimodule. 
We first find a homotopy equivalence between $\bar{B}V \otimes_\Bbbk V \otimes_\Bbbk \bar{B}V$ and $\bar{B} V$ (and hence the same holds for their duals). A homotopy equivalence can be chosen as follows. We define
$$s:  \bar{B} V \to  \bar{B}V \otimes_\Bbbk V \otimes_\Bbbk \bar{B}V \qquad \vec{X} \to \sum \vec{X}_1 \otimes 1_V \otimes \vec{X}_2,$$
where we use Sweedler notation $ \Delta \vec{X} = \sum \vec{X}_1 \otimes \vec{X}_2$
for the comultiplication on the coalgebra $\bar{V}$. (This includes the case $\vec{X}_1 = 1_\Bbbk$ or $\vec{X}_2=1_\Bbbk$.) It is elementary to check that $s$ is a chain map.
For the opposite direction, we take the projection 
$$\pi : \bar{B}V \otimes_\Bbbk V \otimes_\Bbbk \bar{B}V \to 1_\Bbbk \otimes 1_V \otimes \bar{B} V \cong \bar{B} V$$
which is obviously a chain map.

\begin{lemma}\label{lem:conthtpy} $s$ and $\pi$ give homotopy equivalences between $\bar{B} V$ and $\bar{B}V \otimes_\Bbbk V \otimes_\Bbbk \bar{B}V$. 
\end{lemma}

\begin{proof}
$\pi \circ s$ is the identity map on $\bar{B}V$. We need to find a homotopy between $s \circ \pi$ and the identity map on $\bar{B}V \otimes_\Bbbk V \otimes_\Bbbk \bar{B}V$. By definition
$$(s \circ \pi) (1_\Bbbk \otimes 1_V \otimes \vec{Z}) = s(\vec{Z})=\sum \vec{Z}_1 \otimes 1_V \otimes \vec{Z}_2,$$
and it is trivial on the other types of elements in $\bar{B}V \otimes_\Bbbk V \otimes_\Bbbk \bar{B}V$.
Consider the degree $-1$ map $H: \bar{B}V \otimes_\Bbbk V \otimes_\Bbbk \bar{B}V \to \bar{B}V \otimes_\Bbbk V \otimes_\Bbbk \bar{B}V$ given by
$$H :\vec{Y} \otimes X \otimes \vec{Z} \mapsto \sum (\vec{Y} \otimes X \otimes \vec{Z}_1) \otimes 1 \otimes \vec{Z}_2 $$
for $X \in V^{>0}$. (The sum includes the case when $\vec{Z}_1 = 1_\Bbbk$ or $\vec{Z}_2 = 1_\Bbbk$.)
When $X$ is itself the unit, $H$ is defined to be zero, i.e., it vanishes on $\bar{B} V \otimes_\Bbbk 1_V \otimes_\Bbbk  \bar{B} V$. 

It is straightforward to check that $(d H + H d)(\vec{Y} \otimes X \otimes \vec{Z})$ equals $\vec{Y} \otimes X \otimes \vec{Z}$ when $X$ is not a multiple of $1_V$ or $\vec{Y} = Y_1 \otimes \cdots \otimes Y_l$ with $l \geq 1$:

\noindent(i) for $X \notin \Bbbk \langle 1_V \rangle$,
\begin{equation*}
\begin{array}{rl}
 &(dH + Hd) (\vec{Y} \otimes X \otimes \vec{Z}) \\
  =& \sum d \left( (\vec{Y} \otimes X \otimes \vec{Z}_1) \otimes 1_V \otimes \vec{Z}_2 \right) + \sum H \left(  ( \vec{Y}_1 \otimes m(\vec{Y}_2) \otimes \vec{Y}_3  )  \otimes X \otimes \vec{Z} \right) \\
  &+ \sum  H\left( \vec{Y}_1 \otimes  m(\vec{Y}_2,  X, \vec{Z}_1) \otimes \vec{Z}_2 ) \right) + \sum H \left(   \vec{Y} \otimes  X \otimes ( \vec{Z}_1  \otimes m(\vec{Z}_2)  \otimes \vec{Z}_3  ) \right)
 \end{array}
 \end{equation*}
and the last three terms in the above sum can be classified into the following four types, which all appear in the first term:
\begin{equation*}
\begin{array}{l}
  \left(\vec{Y}_1 \otimes m(\vec{Y}_2) \otimes \vec{Y}_3  \otimes X \otimes \vec{Z}_1\right) \otimes 1_V \otimes  \vec{Z}_2 \\
  \left( \vec{Y}_1 \otimes m(\vec{Y}_2, X , \vec{Z}_1) \otimes \vec{Z}_2 \right)\otimes 1_V \otimes  \vec{Z}_3 \\
   \left( \vec{Y}  \otimes X \otimes \vec{Z}_1 \right)\otimes 1_V \otimes \left( \vec{Z}_2 \otimes m(\vec{Z}_3)  \otimes \vec{Z}_4 \right) \\
  \left( \vec{Y}  \otimes X \otimes \vec{Z}_1 \otimes m(\vec{Z}_2) \otimes \vec{Z}_3 \right)\otimes 1_V  \otimes \vec{Z}_4.
  \end{array}
  \end{equation*}
After cancellations of these terms, we are left with
\begin{equation*}
\begin{array}{rl}
 &(dH + Hd) (\vec{Y} \otimes X \otimes \vec{Z}) \\
  =& 
\vec{Y} \otimes m_2 (X, 1_V) \otimes \vec{Z} + \\
&+\sum_{a=1}^k    ( \vec{Y} \otimes X \otimes Z_1 \cdots Z_{a-1} ) \otimes m_2(1_V, Z_a) \otimes  (Z_{a+1} \otimes \cdots \otimes Z_k)  \\
&  - \sum_{a=1}^k ( \vec{Y} \otimes X \otimes Z_1 \cdots Z_{a-1} ) \otimes m_2(Z_a ,1_V) \otimes  (Z_{a+1} \otimes \cdots \otimes Z_k) 
\\
=&\vec{Y} \otimes X \otimes \vec{Z}.
\end{array}
 \end{equation*}
where $\vec{Z} = Z_1 \otimes \cdots \otimes Z_k$

\noindent (ii) If $\vec{Y} = Y_1 \otimes \cdots \otimes Y_l$ with $l \geq 1$,
\begin{equation*}
\begin{array}{rcl}
 (dH + Hd) (\vec{Y} \otimes 1_V \otimes \vec{Z}) &=&  H ( (Y_1 \otimes \cdots \otimes Y_{l-1}) \otimes Y_l \otimes \vec{Z} + \vec{Y} \otimes Z_1 \otimes  (Z_2 \otimes \cdots \otimes Z_k)
)\\ 
 &=& 
 \sum_{a=0}^{k} (\vec{Y} \otimes Z_1 \otimes \cdots \otimes Z_a) \otimes 1_V \otimes (Z_{a+1} \otimes \cdots \otimes Z_k) \\
 && -   \sum_{a=1}^{k} (\vec{Y} \otimes Z_1 \otimes \cdots \otimes Z_a) \otimes 1_V \otimes (Z_{a+1} \otimes \cdots \otimes Z_k) \\
&=&\vec{Y} \otimes 1_V \otimes \vec{Z}.
  \end{array}
 \end{equation*}

Lastly, $dH + Hd$ on the component $1_\Bbbk \otimes 1_V \otimes \bar{B} V$ can be computed as
\begin{equation*}
\begin{array}{rcl}
(dH + Hd )(1_\Bbbk \otimes  1_V \otimes \vec{Z}) &=& H(1_\Bbbk \otimes Z_1 \otimes (Z_2 \otimes \cdots \otimes Z_k))  \\
&=&  Z_1 \otimes 1_V \otimes (Z_2 \otimes \cdots \otimes Z_k) \\
&&+ (Z_1 \otimes  Z_2) \otimes 1_V  \otimes (Z_3 \otimes \cdots \otimes Z_k) \\
&&+ (Z_1 \otimes  Z_2 \otimes Z_3) \otimes 1_V  \otimes (Z_4 \otimes \cdots \otimes Z_k) \\
&&+ \cdots + (Z_1 \otimes \cdots Z_k) \otimes 1_V \otimes 1_\Bbbk \\
&=& \sum \vec{Z}' \otimes 1_V \otimes \vec{Z}'' - 1_\Bbbk \otimes  1_V \otimes \vec{Z}\\
&=& (s \circ \pi) (1_\Bbbk \otimes 1_V \otimes \vec{Z}) -1_\Bbbk \otimes  1_V \otimes \vec{Z}
\end{array} 
\end{equation*}
Therefore $H$ is a contracting homotopy from the identity $s \circ \pi$ as desired.
\end{proof}

In particular, the homology of $\delta$ equals that of $\bar{B} V$. 
Dually,
$\left(\bar{B}V \otimes_\Bbbk V \otimes_\Bbbk \bar{B}V \right)^\sharp $ is homotopic to $V^! = (\bar{B} V)^\sharp$. 

\begin{cor}\label{cor:resolalg1} The map $ \left(\bar{B}V \otimes_\Bbbk V \otimes_\Bbbk \bar{B}V \right)^\sharp=\widehat{V^! \otimes_\Bbbk V^\sharp \otimes_\Bbbk V^!} \stackrel{s^\sharp}{\to} V^! $ dual to $s$ is a homotopy equivalence. 
\end{cor} 

For later use, we give a more explicit expression of homotopy equivalences in terms of coordinates chosen in \ref{subsec:coordkos}. For $f(x), g(x) \in V^!$, we have
\begin{equation}\label{eqn:homotopyequiv}
\begin{array}{rcl}
s^\sharp (f (x) \otimes 1_V^\sharp \otimes g (x)) &=& f(x)g(x), \\
s^\sharp (f (x) \otimes x_i \otimes g (x)) &=& 0, \\
 \pi^\sharp (f(x)) &=& f(x) \otimes 1_V^\sharp \otimes 1_\Bbbk .
 \end{array}
\end{equation}
Clearly, $ s^\sharp \circ \pi^\sharp = id$. 

On the other hand, $\pi^\sharp \circ s^\sharp$ can be homotoped to identity via $H^\sharp$ dual to $H$. $H^\sharp$ can be seen more easily through the identification $\widehat{V^! \otimes_\Bbbk V^\sharp \otimes_\Bbbk V^!} \cong \widehat{V^! \otimes_\Bbbk V[n] \otimes_\Bbbk V^!}$. 
With this identification,
$s^\sharp: \widehat{V^! \otimes_\Bbbk V[n]  \otimes_\Bbbk V^!} \to V^!$ can be written as  
$$s^\sharp (f (x) [\mathrm{pt}_V]  g (x)) = f(x)g(x), \qquad s^\sharp (f (x)  X_i   g (x)) =0$$ 
and $\pi^\sharp :  \widehat{V^! \to V^!  \otimes_\Bbbk V[n] \otimes_\Bbbk V^! }$,
$$ \pi^\sharp (f(x)) = f(x) [\mathrm{pt}_V].$$ 
Recall that $\deg [\mathrm{pt}_V]=0$ in $V[n]$, and hence both $s^\sharp$ and $\pi^\sharp$ are degree $0$. The following is by direct computation.

\begin{lemma}
If $V$ is a cyclic $A_\infty$-algebra with respect to $\langle -,- \rangle$, $H^\sharp$ can be written as
$$ H^\sharp ( x_1 \cdots x_k \otimes [\mathrm{pt}_V] \otimes y_1 \cdots y_l ) = \sum_i x_1 \cdots x_{i-1} \otimes \bar{X}_i \otimes x_{i+1} \cdots x_k y_1 \cdots y_l.$$
under the identification $\widehat{V^! \otimes_\Bbbk V^\sharp \otimes_\Bbbk V^!} \cong \widehat{V^! \otimes_\Bbbk V[n] \otimes_\Bbbk V^!}$. 
\end{lemma}

It is not surprising to have $[\mathrm{pt}_V]$ in many places, as it is  the dual of $1_V$ and $d (X_i)$ always involves $[\mathrm{pt}_V]$ (due to the cyclic symmetry $m_2 (X_i , \bar{X}_i) = m_2 (\bar{X}_i,X_i) = [\mathrm{pt}_V]$).



Unfortunately, $\widehat{V^! \otimes_\Bbbk V \otimes_\Bbbk V^!}[n]$ is not itself a free $V^!$-bimodule. Moreover, its structure does not, in general, descend to $V^! \otimes_\Bbbk V^\sharp \otimes_\Bbbk V^!$, since the differential need not take values in this uncompleted tensor product. To circumvent this difficulty, one may impose a strong finiteness condition on the $A_\infty$-operations on $V$ and work with $V^!_c$ \eqref{eqn:v!c} instead. For instance, one may impose the following condition:


\begin{defn}\label{def:finmkv}
We call the $A_\infty$-operations $\{m_k\}$ on $V$ finite-type if there exists $N$ such that $m_k \equiv 0$ for $k \geq N$.\footnote{Such an algebra is also referred to an $A_{\mathrm{fin}}$-algebra in the literature.} 
\end{defn}

This is not an intrinsic property of an $A_\infty$-algebra, but rather a property of a particular model $V$, since it is not preserved by $A_\infty$-homotopies. Note that $V$ is always $A_\infty$-isomorphic to some dga $V'$, but the rank of $V'$ does not need to be finite anymore, hence may not satisfy Assumption \ref{assume:onv} in general.

When $(V,\{m_k\})$ is finite-type, $V^!_c$ \eqref{eqn:vcvc} is a well-defined dga and it is not difficult to check that the differential on $V^!_c  \otimes_\Bbbk V^\sharp \otimes_\Bbbk V^!_c$ is a finite sum. Furthermore, the previous discussion on homotopy equivalence is still valid for $V^!_c \otimes_\Bbbk V^\sharp \otimes_\Bbbk V^!_c$, and hence $V^!_c \otimes_\Bbbk V^\sharp \otimes_\Bbbk V^!_c$ can serve as its resolution.
In summary,

\begin{cor}\label{cor:freeresolfin}
Suppose $V$ is finite-type. Then $V^!_c \otimes_\Bbbk V^\sharp \otimes_\Bbbk V^!_c$ is a well-defined dg bimodule over $V^!_c$, and is homotopy equivalent to $V^!_c$ itself, or equivalently $s^\sharp: V^!_c \otimes_\Bbbk V^\sharp \otimes_\Bbbk V^!_c \to V^!_c$ gives a free resolution of the diagonal bimodule $V^!_c$. 
\end{cor}

When $V=\CF(\bL,\bL)$, the above can be rephrased in geometric terms as follows. (Recall that $V_c^! = \tilde{\cA}_{\bL,c}$ in this case.)

\begin{prop}\label{prop:fresolofdga}
Suppose $\bL$ is an n-dimensional compact $\mathbb{Z}$-graded unobstructed immersed Lagrangian without nonpositive degree immersed generators. If one of the minimal models of $\CF (\bL ,\bL)$ is finite-type, then $\CF_c ((\bL,\tilde{b}),(\bL,\tilde{b}))[n]$ (defined on the same model; see Remark \ref{rmk:afincfl}) 
 gives a projective bimodule resolution $\tilde{\cA}_{\bL,c}$ viewed as the diagonal bimodule over the dga $\tilde{\cA}_{\bL,c}$.
\end{prop}

Any spherical object obviously satisfies the condition in the statement. Namely, if a unital $\mathbb{Z}$-graded $A_\infty$ algebra $V$ takes the form of $V \cong \mathbb{C} \langle 1_V \rangle \oplus \mathbb{C} \langle [\mathrm{pt}_V] \rangle$, then by unitality and degree considerations, all higher $m_k$ with $k \geq 3$ vanish. More generally, if $V$ is formal, then it is obviously finite-type.


\subsection{Koszul duality and Hochschild invariants of local mirrors}

Recall that, for a dga \(A\) over \(\Bbbk\), its Hochschild cohomology with
coefficients in itself is defined by
\[
HH^\ast(A,A)
=
H^\ast\!\left(\operatorname{RHom}_{A^e}(A,A)\right)
=
\operatorname{Ext}_{A^e}^\ast(A,A),
\qquad
A^e:=A\otimes_\C A^{op},
\]
where \(A\) is regarded as a bimodule over itself, or equivalently as a left
\(A^e\)-module.
Suppose that the \(A_\infty\)-structure on \(V\) is of finite type, so that
\(V_c^!\) is well-defined. Using the bimodule resolution
\[
P_c:=V_c^!\otimes_\Bbbk V^\sharp\otimes_\Bbbk V_c^!
\longrightarrow V_c^!,
\]
we obtain
\begin{equation}\label{eq:hh}
\begin{aligned}
HH^\ast(V_c^!,V_c^!)
&=
H^\ast\!\left(
\operatorname{Hom}_{V_c^!-V_c^!}(P_c,V_c^!)
\right) \\
&\cong
H^\ast\!\left(
\operatorname{Hom}_{\Bbbk-\Bbbk}(V^\sharp,V_c^!)
\right) \\
&\cong
H^\ast\!\left(
(V_c^!\otimes_\Bbbk V)_{\mathrm{cyc}}
\right).
\end{aligned}
\end{equation}
Here \((V_c^!\otimes_\Bbbk V)_{\mathrm{cyc}}\) denotes the subspace of cyclic
elements with respect to the \(\Bbbk\)-bimodule structure. The appearance of
this cyclic subspace reflects the fact that a \(V_c^!\)-bimodule homomorphism
is, in particular, a \(\Bbbk\)-bimodule homomorphism.

For the completed Koszul dual $V^!$, it is more natural to retain its natural pseudocompact topology and consider the corresponding continuous Hochschild theory. More precisely, following \cite[2.2, 4.3]{GHL23}, we regard $V^!$ as a pseudocompact dga and compute derived bimodule morphisms in the contraderived category of pseudocompact $V^!$-bimodules.\footnote{In the terminology of \cite[\S4.3]{GHL23}, this is denoted by $HH^\ast_{\mathrm{ctr}}(V^!)$, and is also referred to as the coHochschild cohomology of the dual coalgebra. We use $HH^\ast_{\mathrm{cont}}(V^!)$ below to emphasize the underlying pseudocompact topology. Morphisms of pseudocompact modules are continuous, and tensor products of pseudocompact algebras or modules are understood to be completed; see \cite[2.2]{GHL23}.}
Thus the relevant enveloping algebra is
\[
(V^!)^e_{\mathrm{cont}}
:=
V^!\widehat{\otimes}_{\C}(V^!)^{op}.
\]
In this setting, the dual of the bar resolution,
\[
\left(\bar BV\otimes_\Bbbk V\otimes_\Bbbk\bar BV\right)^\sharp
\cong
\underbrace{
V^!\widehat{\otimes}_\Bbbk V^\sharp
\widehat{\otimes}_\Bbbk V^!
}_{=:P_{\mathrm{cont}}}
\stackrel{s^\sharp}{\longrightarrow}
V^!,
\]
This is a cofibrant resolution of the diagonal $V^!$-bimodule in the (contraderived) category of pseudocompact bimodules. Consequently,
\begin{equation}\label{eqn:hhloop}
\begin{aligned}
HH^\ast_{\mathrm{cont}}(V^!,V^!)
&:=
H^\ast\!\left(
\operatorname{Hom}^{\mathrm{cont}}_{(V^!)^e_{\mathrm{cont}}}
(P_{\mathrm{cont}},V^!)
\right) \\
&\cong
H^\ast\!\left(
\operatorname{Hom}_{\Bbbk-\Bbbk}(V^\sharp,V^!)
\right) \\
&\cong
H^\ast\!\left(
(V^!\otimes_\Bbbk V)_{\mathrm{cyc}}
\right),
\end{aligned}
\end{equation}
in complete analogy with \eqref{eq:hh}.

%

Let us examine the structure of $(V^! \otimes_\Bbbk V)_{cyc}$ more closely using coordinates in \ref{subsec:coordkos}. An element
$\vec{x} \otimes X_k$ in $(V^! \otimes_\Bbbk V)_{cyc}$
 represents a bimodule map
$$ \vec{x} \otimes X_k: V^! \otimes_\Bbbk V^\sharp \otimes_\Bbbk V^!  \to V^! \qquad 
x_j \in V^\sharp \mapsto x_j (X_k) \, \vec{x} 
$$
Note that this is not trivial only when $t(X_k) = h(\vec{x})$ (and $t(\vec{x}) = h(X_k)$ from the definition of $\otimes_\Bbbk$), and hence $\vec{x}X_k$ corresponds to a cycle in the quiver.
Let us compute the differential on \eqref{eqn:hhloop} in these coordinates. We set $a_{v_i,w_i,j} \in \mathbb{C}$ to be the structure coefficients appearing in
$$dx_j = \sum_{ i, v_i, w_i} a_{v_i,w_i,j} (x_{\vec{v}_i} \otimes x_i \otimes x_{\vec{w}_i} ),$$
and hence $a_{v_i,w_i,j} = dx_j ( X_{\vec{w}_i} \otimes X_i \otimes X_{\vec{v}_i})$, or equivalently 
$ m(X_{\vec{w}_i^{op}} \otimes X_i \otimes X_{\vec{v}_i^{op}}) = \sum_j a_{v_i,w_i,j} X_j.$
Then we have
\begin{equation*}
\begin{array}{rcl}
(\vec{x} \otimes X_k) (d x_j) &=& ( \vec{x} \otimes X_k ) \left( \sum_{ i, v_i, w_i} a_{v_i,w_i,j} x_{\vec{v}_i} \otimes x_i \otimes x_{\vec{w}_i} \right) \\
&=&  \sum_{ i, v_i, w_i} a_{v_i,w_i,j} \, (x_{\vec{v}_i} \otimes x_i (X_k) \vec{x} \otimes x_{\vec{w}_i} ) \\
&=&  \sum_{ v_k, w_k} a_{v_i,w_i,j} \, (x_{\vec{v}_k} \otimes  \vec{x} \otimes x_{\vec{w}_k} ) 
\end{array}
\end{equation*}
and
\begin{equation*}
\begin{array}{rcl}
d ((\vec{x} \otimes X_k ) ( x_j)) &=& x_j (X_k)  d( \vec{x}) \\
&=& (d\vec{x} \otimes X_k) (x_j)\end{array}
\end{equation*}
Therefore
\begin{equation}\label{eqn:dcpxhh}
\begin{array}{rcl}
d (\vec{x} \otimes X_k) &=&d\vec{x} \otimes X_k + (-1)^{|\vec{x}|+|X_k|+1}\sum_{j, \vec{v}_k, \vec{w}_k}  \left( x_{\vec{v}_k} \otimes  \vec{x} \otimes x_{\vec{w}_k} \right) \otimes (a_{v_k,w_k,j} X_j) \\
&=&d\vec{x} \otimes X_k + (-1)^{|\vec{x}|+|X_k|+1} \sum_{\vec{v}_k, \vec{w}_k}  \left( x_{\vec{v}_k} \otimes  \vec{x} \otimes x_{\vec{w}_k} \right) \otimes m(X_{\vec{w}_k^{op}} \otimes X_k \otimes X_{\vec{v}_k^{op}}) \\
&=&d\vec{x} \otimes X_k + (-1)^{|\vec{x}|+|X_k|+1} \sum_{\vec{v}_k, \vec{w}_k} m(x_{\vec{v}_k} X_{\vec{v}_k^{op}} \otimes \vec{x} X_k \otimes x_{\vec{w}_k}  X_{\vec{w}_k^{op}}) \\
&=& d\vec{x} \otimes X_k + (-1)^{|\vec{x}|+|X_k|+1} m_1^{\tilde{b},\tilde{b}} (\vec{x} X_k)
\end{array}
\end{equation}
Thus we obtain 
\begin{cor}\label{cor:hhextmir} For a compact $\mathbb{Z}$-graded unobstructed immersed Lagrangian $\bL$ without nonpositive degree immersed generators, 
$$ HF_{cyc} ((\bL,\tilde{b}), (\bL,\tilde{b})) \cong HH^\ast_{cont} (\tilde{\mathcal{A}}_\bL,\tilde{\mathcal{A}}_\bL ).$$
When $\CF(\bL ,\bL)$ admits a finite-type minimal model (Definition \ref{def:finmkv}), we have
$$ HF_{cyc,c} ((\bL,\tilde{b}), (\bL,\tilde{b})) \cong HH^\ast (\tilde{\mathcal{A}}_{\bL,c},\tilde{\mathcal{A}}_{\bL,c} )$$
where the left hand side is the cohomology of the cyclic Floer complex $\CF_{cyc} ((\bL,\tilde{b}), (\bL,\tilde{b}))$ consisting of cyclic elements.
\end{cor}

Note that this does not use any cyclic symmetry properties. 

\subsubsection{`Classical mirror symmetry' and Koszul duality}

On the other hand, the usual bar resolution $V  \otimes_\Bbbk \bar{B}V \otimes_\Bbbk V$ of $V$ can be used to compute the Hochschild cohomology of the $A_\infty$-algebra $V$ itself. By the same argument above, we have
$$HH^\ast (V, V) =H^\ast (\mathrm{Hom}_{V-V} (V \otimes_\Bbbk \bar{B}V \otimes_\Bbbk V, V))=H^\ast (\mathrm{Hom}_{\Bbbk-\Bbbk} ( \bar{B}V , V)) \cong H^\ast (V^! \otimes V)_{cyc},$$
and hence, $HH^\ast(V,V)\cong HH^\ast_{\mathrm{cont}}(V^!,V^!)$. The corresponding identification also appears in the first assertion of
\cite[Corollary 4.12]{GHL23}.
Observe that this does not preserve the bidegree of Hochschild cocycles since the grading on $V^!$ defines the internal grading in the case of $HH^\ast (V^!, V^!)$, but the homological grading in the case of $HH^\ast (V, V)$. Therefore we have
\begin{equation}\label{eqn:hodgedia}
HH^{p,q} (V,V) = HH^{q,p}_{cont} (V^!,V^!).
\end{equation}

\subsubsection{Spectral sequence}\label{subsubsec:specseq}
One can equip $(V^! \otimes V)_{cyc}$ with an obvious bigrading, given by
$$ \deg(\vec{x} \otimes Y) := (\deg (\vec{x}), \deg(Y)).$$
With respect to this grading, the differential \eqref{eqn:dcpxhh} decomposes into two operators of degrees $(1,0)$ and $(0,1)$, respectively; that is, $(V^! \otimes V)_{cyc}$ is a double complex. 
Consider the associated spectral sequence $E_{p,q}$. It lies in the third quadrant, and since $V$ is finite-dimensional, the sequence must collapse at a finite stage. 
\begin{lemma}
The above spectral sequence $E_{p,q}$ associated with the double complex $(V^! \otimes V)_{cyc}$ degenerates. If the cohomology of $V^!$ is concentrated at degree $0$, then the $E_2=E_\infty$, which is isomorphic to $HH_{cont} (V^!,V^!)$. 

For $V=\CF(\mathbb{L},\mathbb{L})$, its $E_2$-page agrees with
$$HF_{cyc} ((\mathbb{L},b),(\mathbb{L},b)) $$
where $b$ is a linear combination of degree 1 generators. This computes $HH^\ast_{cont} (A_\mathbb{L},A_\mathbb{L})$ in the setting of Corollary \ref{cor:hhextmir}.
\end{lemma}

\begin{proof}
From \eqref{eqn:dcpxhh}, we see that the $E_1$-page, which is by definition the cohomology with respect to $(1,0)$-differential, is precisely $(H^0(V^!) \otimes_\Bbbk V)_{cyc} = (A_\mathbb{L} \otimes_\Bbbk \CF_{cyc} (\mathbb{L},\mathbb{L}))$, and the induced $(0,1)$-differential on $E_1$ is $m_1^{b,b}$. Thus the $E_2$-page coincides with $HF_{cyc} ((\mathbb{L},b),(\mathbb{L},b))$.
\end{proof}

\subsection{Calabi-Yau structure on $V^!$}
In our setup, especially the Poincar\'{e} pairing or nondegenerate cyclic pairing on $V$ makes $V^!$ into a Calabi-Yau dga of dimension $n$. Recall
\begin{defn}
A dga $A$ is called Calabi-Yau of dimension $n$ if
$$ \mathrm{RHom}_{A\otimes_\mathbb{C} A^{op}} (A,A \otimes_\mathbb{C} A^{op}) \cong A[n].$$
(The left-hand side can be identified with the space of $A-A$ bimodule maps.)
\end{defn}

Observe that if $V^!_c$ and $V^! \otimes_\Bbbk V^\sharp \otimes_\Bbbk V^!$ are well-defined (hence the latter serves as a free resolution of the former), then
\begin{equation*}
\begin{array}{rcl}
 \mathrm{RHom}_{V^!_c \otimes_\mathbb{C} (V^!_c)^{op}} (V^!_c,V^!_c \otimes_\Bbbk (V^!_c)^{op}) &=&\mathrm{Hom}_{V^!_c - V^!_c} (V^!_c \otimes_\Bbbk V^\sharp \otimes_\Bbbk V^!_c,V^!_c \otimes_\mathbb{C} V^!_c) \\
 &=& \mathrm{Hom}_{\Bbbk-\Bbbk} ( V^\sharp ,V^!_c \otimes_\mathbb{C} V^!_c) 
  \end{array}
 \end{equation*}
where $\mathrm{Hom}_{V^!_c - V^!_c}$ means $V^!_c$-bimodule homomorphisms and $V^!_c \otimes_\mathbb{C} V^!_c$ is given the outer bimodule structure:
$$ \vec{a} \cdot (\vec{x} \otimes \vec{y}) \cdot \vec{b} = (\vec{a} \vec{x}) \otimes (\vec{y} \vec{b}).$$
Then $ \mathrm{Hom}_{\Bbbk} ( V^\sharp ,V^!_c \otimes_\mathbb{C} V^!_c)$ itself becomes a bimodule via the inner bimodule structure on $V^!_c \otimes_\mathbb{C} V^!_c$:
$$ \vec{a} \cdot (\vec{x} \otimes \vec{y}) \cdot \vec{b} = (\vec{x} \vec{b} ) \otimes (\vec{a} \vec{y}).$$

Finally, by adjunction (since $V$ is finite-dimensional),
\begin{equation}\label{eqn::bimodtensor}
\begin{array}{rcl}
\mathrm{Hom}_{\Bbbk} ( V^\sharp ,V^!_c \otimes_\mathbb{C} V^!_c) &=&  V \,\, ``\otimes_{\Bbbk-\Bbbk}" \,\, (V^!_c \otimes_\mathbb{C} V^!_c ) \vspace{0.2cm} \\
&:=& \vspace{0.2cm}
\left(\begin{array}{ccc} & _{\color{red} \Bbbk} V_{\color{red}\Bbbk} & \vspace{-0.1cm}  \\  &{\color{red}\otimes_\Bbbk}   \hspace{0.8cm}  {\color{red}\otimes_\Bbbk}  & \vspace{0.1cm} \\ & {\color{blue}V^!_c} \otimes_\mathbb{C} {\color{olive}V^!_c}  & \end{array}\right)
 \\
&=& {\color{blue}V^!_c} \otimes_\Bbbk V \otimes_\Bbbk {\color{olive}V^!_c} \\
 &\cong& V^!_c \otimes_\Bbbk V^\sharp \otimes_\Bbbk V^!_c [n] \cong V^!_c [n],
\end{array}
\end{equation}

we obtain:
\begin{lemma}
When $V$ is finite-type (Definition \ref{def:finmkv}), 
 $V^!_c$ is Calabi-Yau of dimension $n$. 
\end{lemma}

A similar duality can be formulated and proved for the completed Koszul dual $V^!$.  We do not know how widely this is used in the literature, and omit the details here.

\section{The localized mirror functor}\label{sec:locfunc}

We have mostly discussed the closed-string mirror symmetry (adapted to the localized mirror setup) in Section \ref{sec:extmirkos}. In fact,  one can construct a natural ring homomorphism from $SH^\ast(X)$ or $QH^\ast(X)$ using Corollary \ref{cor:hhextmir}, which can be understood as a generalization of the Kodaira-Spencer map in \cite{FOOOtoric} (see \cite{HJL25, CHJL25} also for related approaches).

We now investigate aspects of homological mirror symmetry for the localized mirror 
$\tilde{\mathcal{A}}_{\bL}$ associated to a compact Lagrangian $\bL$. 
More specifically, we compare the subcategory of $\Fuk(X)$ generated by $\bL$ with 
the dg-module category over the associated local mirror $\tilde{\mathcal{A}}_{\bL}$ in this section. 
From Section~\ref{sec:prelim}, we have a natural $A_\infty$-functor
\[
\cF^{(\bL,\tilde{b})} : \Fuk(X) \to \mathrm{Mod}_{\dg}(\tilde{\mathcal{A}}_{\bL}).
\]
We show that, after restricting to appropriate subcategories on both sides, 
this functor becomes an equivalence.


\subsection{The localized mirror functor $\cF^{(\bL,\tilde{b})}$ as the Koszul dual functor}
The localized mirror functor can be interpreted as a geometric manifestation of the Koszul dual functor, which is defined as follows. In an abstract algebraic setup, there exists a canonical functor
$$\mathcal{F} :  \mathrm{Mod}_{A_\infty} (V) \to  \mathrm{Mod}_{\dg} (V^!)$$
where $\mathrm{Mod}_{A_\infty} (V)$ is the dg category of (right) $A_\infty$-bimodules over $V$ (with pre-homomorphisms), and $\mathrm{Mod}_{\dg} (V^!)$ is analogously defined.
It sends a right $A_\infty$-module $M$ over $V$ to
\begin{equation}\label{eqn:HAmirrorfunc}
  \mathcal{F} (M) =\Hom_{\mathrm{Mod}_{A_\infty} (V) } (M,\Bbbk) \cong   \Hom_\Bbbk (M \otimes_\Bbbk \bar{B}V, \Bbbk).
\end{equation}
Notice that the right-hand side is canonically a right module over $\mathrm{RHom}_{V}  (\Bbbk,\Bbbk) =V^!$. 
When $M$ is finite-dimensional, we can further simplify
\begin{equation}\label{eqn:functorforfdm}
\mathcal{F} (M)= \Hom_\Bbbk (M \otimes_\Bbbk \bar{B}V, \Bbbk) \cong V^! \otimes M^\sharp.
\end{equation}
We assume this is the case from now on.

On the morphism level, $\mathcal{F}$ is given by the obvious duality map (so it is contravariant)
\begin{equation}\label{eqn:kosfunc}
 \mathrm{RHom}_{V}  (M,M') \to \mathrm{Hom}_{\mathrm{Mod}_{\dg} (V^!)}  ( \mathrm{RHom}_{V}  (M', \Bbbk),\mathrm{RHom}_{V}  (M, \Bbbk)),
\end{equation}
which sends $f \in \mathrm{RHom}_{V} (M,M')$ to the map on the right-hand side of \eqref{eqn:kosfunc} 
\begin{equation}\label{eqn:defmorfun}
 g\in \underbrace{\mathrm{RHom}_{V} (M', \Bbbk)}_{ = \Hom_\Bbbk (M' \otimes_\Bbbk \bar{B}V, \Bbbk)} \mapsto \left(M \otimes_\Bbbk \bar{B}V   \stackrel{\hat{f}}{\to} M' \otimes_\Bbbk \bar{B}V  \stackrel{g}{\to} \Bbbk\right) \in \underbrace{ \mathrm{RHom}_{V}  (M, \Bbbk)}_{ =\Hom_\Bbbk (M \otimes_\Bbbk \bar{B}V, \Bbbk)}.
\end{equation}
Here, $\mathrm{Hom}_{\mathrm{Mod}_{\dg} (V^!)}$ are simply $V^!$-linear maps between dg modules over $V^!$.

In order to compare $\mathcal{F}$ and the localized mirror functor in \ref{subsec:reviewlocmir}, we compute the differential on $ \mathcal{F} (M)$ in terms of coordinates. For this purpose, let us fix a basis $\{P_1, \cdots, P_{N'}\}$ of $M$, and its dual basis  $\{Q_1,\cdots, Q_{N'}\} $ of $M^\sharp$, which itself is a module over $V$ via
$$ (P_i, m_k (X_{i_1},\cdots, X_{i_l}, Q_j)):= (m_k (P_i, X_{i_l},\cdots, X_{i_1}), Q_j).$$
As before, we write $\{1,X_1,\cdots, X_N\}$ for a chosen basis of $V$.
We identify $x_{\vec{v}} \otimes Q_i \in V^! \otimes M^\sharp$ with  the element of $\hom_\Bbbk (\bar{B}V \otimes M, \Bbbk)$ \eqref{eqn:functorforfdm} via
$$ x_{\vec{v}} \otimes Q_i :   P_j \otimes X_{\vec{w}} (\in M \otimes_\Bbbk \bar{B}V ) \mapsto Q_i (P_j) \cdot x_{\vec{v}} (X_{\vec{w}})  (\in \Bbbk).$$
By definition, the differential of $x_{\vec{v}} \otimes Q_i$ is given by
\begin{equation*}
\begin{array}{rcl}
 d ( x_{\vec{v}} \otimes Q_i) ( P_j \otimes X_{\vec{w}} ) &=&   Q_i (P_j) \cdot x_{\vec{v}}  ( \sum (-1)^{|\vec{w}_1|'}X_{\vec{w}_1} \otimes m (X_{\vec{w}_2}) \otimes X_{\vec{w}_3})  \\
 &&+ \sum (-1)^{\eta(P_j,\vec{w}_1)} Q_i (m ( P_j, X_{\vec{w}_1} ) ) \cdot x_{\vec{v}} (X_{\vec{w}_2} )  \\
 &=& (d \vec{x} \otimes Q_i) ( X_{\vec{w}} \otimes P_j) + 
\sum_{\vec{w} =  \vec{w}' \otimes \vec{v}^{op}}  (-1)^{\eta(P_j,\vec{w}')}m ( X_{\vec{w}'^{op}} ,Q_i) (  P_j ).
 \end{array}
 \end{equation*}
 Here $\eta$ denotes the Koszul sign in the right-module bar differential, computed by moving the corresponding module operation past the preceding shifted inputs.
 Therefore
 \begin{equation}\label{eqn:funcdouble}
\begin{array}{rcl}
d ( x_{\vec{v}} \otimes Q_i) &=& d \vec{x} \otimes Q_i + \sum_{\vec{w}'} (-1)^{\eta(Q_i,\vec{w}')} (x_{\vec{w}'} \otimes x_{\vec{v} }  ) \otimes \, m (X_{\vec{w}'^{op}}, Q_i)\\
&=& d \vec{x} \otimes Q_i + \sum_{\vec{w}'} (-1)^{\eta(Q_i,\vec{w}')} x_{\vec{v} }   \otimes \, m ( x_{\vec{w}'} X_{\vec{w}'^{op}}, Q_i) = d(\vec{x}) \otimes Q_i +  \vec{x} m_1^{\tilde{b},0} (Q_i).
  \end{array}
 \end{equation}

We see that $\mathcal{F}$ is precisely the (extended) localized mirror functor $\cF^{(\bL,\tilde{b})}$ when $V= \CF(\mathbb{L},\mathbb{L})$.
In the geometric situation, $M$ corresponds to $CF(L,\mathbb{L})$, and due to the Poincar\'e pairing one has a canonical identification $M^{\sharp} \cong \CF(\mathbb{L},L)$. (Note that $M$ is finite-dimensional since $\mathbb{L}$ is compact.) Thus we have obtained:
\begin{prop} \label{prop:kos}
For $L \in \Fuk (X)$, $\mathcal{F}^{(\bL,\tilde{b})} (L) = \mathcal{F} (M_L)$ where $M_L$ is the module over $\CF(\bL,\bL)$ given as $M_L = \CF(L,\bL)[n]$. 
\end{prop}
In other words, $\mathcal{F}^{(\bL,\tilde{b})}$ can be understood as the composition of the Yoneda functor
$$ \mathcal{Y}_\bL : \Fuk (X) \to \mathrm{Mod} (V)$$ 
with respect to $\bL$, and the (categorical) Koszul duality
$$ \mathcal{K}: \mathrm{Mod}(V) \to \mathrm{Mod}(V^!).$$
More precisely, $\mathcal{F}^{(\bL,\tilde{b})} = \mathcal{K} \circ \mathcal{Y}_\bL [n]$, but we will omit the degree shift, and simply identify $\mathcal{F}^{(\bL,\tilde{b})} = \mathcal{K} \circ \mathcal{Y}_\bL$ from now on.

%

\subsection{Fully-faithfulness of the localized mirror functor $\cF^{(\bL,\tilde{b})}$}
Using the above Koszul formulation, we can easily deduce that all the simple objects in the mirror are the image of the mirror functor.

\begin{lemma}\label{lem:hitbbk}
The image of $\mathbb{L}$ itself under the (extended) localized mirror functor 
 $$\cF^{(\bL,\tilde{b})} : \Fuk (X) \to \mathrm{Mod}_{\dg} (\tilde{\mathcal{A}}_{\mathbb{L}})$$
is quasi-isomorphic to the simple object $\Bbbk$. Moreover, the cohomology of $\cF^{(\bL,\tilde{b})}(\mathbb{L})$ is generated by $[pt_{\mathbb{L}}] \in HF((\mathbb{L},b),\mathbb{L})) \cong \Bbbk$.
\end{lemma}

\begin{proof}
Algebraically, this is the case when we plug in $M=V$ in \eqref{eqn:HAmirrorfunc}.
Thus
$$ \cF^{(\bL,\tilde{b})}(\mathbb{L}) = \mathrm{Hom}_{\mathrm{Mod}_{A_\infty} (V)} (V,\Bbbk) \cong  \mathrm{Hom}_{\Bbbk} (V \otimes_\Bbbk \bar{B} V ,\Bbbk) \cong V^! \otimes_\Bbbk V^\sharp \stackrel{qis}{\cong} \Bbbk$$
by Lemma \ref{lem:leftrightresol}, where the quasi-isomorphism takes $(f:V\to\Bbbk) \in \mathrm{Hom}_{\mathrm{Mod}_{A_\infty} (V)} (V,\Bbbk)$ to $ f(1_{\mathbb{L}}) \in  \Bbbk$. If we view $f$ as a $\Bbbk$-linear map,  then by nondegeneracy of the Poincar\'{e} pairing, one can find an element $Q$ of $V$ such that $f= (Q, -)$, and hence $f(1_\mathbb{L}) = (Q, 1_\mathbb{L})$. For this to be nontrivial, $Q$ must be a multiple of $[pt_{\mathbb{L}}]$. In other words, the cohomology of $\CF((\mathbb{L},b),\mathbb{L}))$ is generated by $[pt_{\mathbb{L}}]$. 
\end{proof}

Let $\mathcal{D}(\tilde{\mathcal{A}}_{\mathbb{L}})$ be the derived category of dg modules over $\tilde{\mathcal{A}}_{\mathbb{L}}$ which is obtained from $\mathrm{Mod}_{\dg} (\tilde{\mathcal{A}}_{\mathbb{L}})$ by inverting quasi-isomorphisms. Lemma \ref{lem:hitbbk} can be reformulated as the statement that the derived functor
$$\mathcal{D} \cF^{(\bL,\tilde{b})} : \mathcal{D}  \Fuk (X) \to \mathcal{D}(\tilde{\mathcal{A}}_{\mathbb{L}})$$
sends $\mathbb{L}$ to $\Bbbk \langle [pt_{\mathbb{L}}] \rangle \cong \Bbbk$.


Let \(\Bbbk=\bigoplus_i \Bbbk_i\) be the direct sum of the simple
\(\tilde{\mathcal A}_{\mathbb L}\)-modules associated to the vertices. We denote by
\[
\operatorname{thick}_{\mathcal D(\tilde{\mathcal A}_{\mathbb L})}(\Bbbk)
\]
the smallest thick triangulated subcategory of
\(\mathcal D(\tilde{\mathcal A}_{\mathbb L})\) containing \(\Bbbk\).\footnote{Intuitively, this consists of simple modules, or finite-dimensional quiver representations.}

Equivalently, this is the full subcategory of
\(\mathcal D(\tilde{\mathcal A}_{\mathbb L})\) consisting of dg
\(\tilde{\mathcal A}_{\mathbb L}\)-modules \(M\) whose cohomology \(H^*(M)\) is
bounded, finite-dimensional, and nilpotent with respect to the ideal $I \subset H^*(\tilde{\mathcal A}_{\mathbb L})$
generated by the cohomology classes of the arrows. Namely, for every $M$
there exists \(N\gg 0\) such that
\[
I^{N} H^*(M)=0.
\]
We denote this subcategory by
\[\mathcal D_{\mathrm{nil}}(\tilde{\mathcal A}_{\mathbb L})
:=
\operatorname{thick}_{\mathcal D(\tilde{\mathcal A}_{\mathbb L})}(\Bbbk).\]

The localized mirror functor behaves particularly well on the full subcategory $\Fuk_{\mathbb{L}} (X) \subset \Fuk(X)$ generated by $\mathbb{L}$. Equivalently, this is a restriction of the Koszul duality functor to the perfect complexes over $V$.

\begin{thm}\label{thm:extfuctorequiv}
Let $\bL$ be a compact $\mathbb{Z}$-graded unobstructed immersed Lagrangian without nonpositive degree immersed generators.
Then the extended mirror functor induces a quasi-equivalence
$$\mathcal{D} \Fuk_{\mathbb{L}} (X) \longrightarrow \mathcal{D}_{\mathrm{nil}}(\tilde{\cA}_\bL).$$
\end{thm}

\begin{proof}
Since we already know $\Bbbk$ lies in the image of the functor, it suffices to compare the endomorphism $V=CF(\mathbb{L},\mathbb{L})$ with its image $\Hom_{\mathrm{Mod}_{\dg} (V^!)} (\Bbbk,\Bbbk)$ on morphism level, passing to the derived category.
Using \eqref{eqn:functorforfdm}, we have
\begin{equation*}
\begin{array}{rcl}
\Hom_{\mathrm{Mod}_{\dg} (V^!)} (\mathcal{F}(V),\mathcal{F}(V)) &=& \Hom_{V^!} (V^! \otimes_\Bbbk  V^\sharp, V^! \otimes_\Bbbk  V^\sharp ) \\
 & =& \mathrm{RHom}_{V^!} (V^! \otimes_\Bbbk  V^\sharp, V^! \otimes_\Bbbk  V^\sharp )  \\
 &\cong& \mathrm{RHom}_{V^!} (\Bbbk, \Bbbk ) 
 \end{array}
 \end{equation*}
by Lemma \ref{lem:leftrightresol}.
The last line uses the fact that $\mathrm{RHom}$ (the derived hom) only depends on the quasi-isomorphism type of objects.
This is simply the Koszul dual of $V^!$. Hence the morphism level functor of $\cF^{(\bL,\tilde{b})}$ on the endomorphism algebra of $\mathbb{L}$ is a natural map (the double Koszul duality map) $ V \to (V^!)^!$, which induces a quasi-equivalence by Lemma \ref{lem:vdoubledual}.
\end{proof}

Moreover, a completed path algebra $\widehat{A}$ has the following nice property.

\begin{lemma}\label{lem:fd=nil}
	Let \(Q\) be a finite quiver, let \(I\subset kQ\) be a two-sided ideal, and set $A:=kQ/I.$
	Let \(\mathfrak m\subset A\) be the image of the arrow ideal of \(kQ\), and let
	\[
	\widehat A:=\varprojlim_{n} A/\mathfrak m^n
	\]
	be the \(\mathfrak m\)-adic completion of \(A\). Denote by \(\widehat{\mathfrak m}\subset \widehat A\) the completed arrow ideal. Then every finite-dimensional left \(\widehat A\)-module is \(\widehat{\mathfrak m}\)-nilpotent, i.e. $\mathrm{Mod}_{\mathrm{fd}}(\widehat{A}) \cong \mathrm{Mod}_{\mathrm{nil}}(\widehat{A}).$
\end{lemma}
\begin{proof}

		Since \(\widehat A\) is complete with respect to the two-sided ideal
		\(\widehat{\mathfrak m}\), we have
		$\widehat{\mathfrak m}\subset \operatorname{Jac}(\widehat A).$
		Indeed, for any $x\in \widehat{\mathfrak m}$ and $a\in \widehat A$, the
		geometric series
		$$
		1+ax+(ax)^2+\cdots
		$$
		converges in the \(\widehat{\mathfrak m}\)-adic topology and gives the inverse
		of \(1-ax\). Hence \(1-ax\) is invertible for all \(a\in \widehat A\), so
		\(x\in \operatorname{Jac}(\widehat A)\) by Jacobson radical criterion.
		
	Let \(M\) be a finite-dimensional left \(\widehat A\)-module. The descending
	filtration
	$$
	M \supset \widehat{\mathfrak m}M \supset \widehat{\mathfrak m}^{2}M
	\supset \cdots
	$$
	stabilizes, since \(M\) is finite-dimensional over \(k\). Hence, for some
	\(N\geq 0\), we have
	\(\widehat{\mathfrak m}^{N}M=\widehat{\mathfrak m}^{N+1}M\). Setting \(M_N:=\widehat{\mathfrak m}^{N}M\), this gives \(M_N=\widehat{\mathfrak m}M_N\). Since
	\(\widehat{\mathfrak m}\subset \operatorname{Jac}(\widehat A)\), Nakayama's lemma implies \(M_N=0\). Therefore \(\widehat{\mathfrak m}^{N}M=0\), so \(M\) is \(\widehat{\mathfrak m}\)-nilpotent.
\end{proof}

\begin{remark}\label{rem: fd=nil}
	Similar statements hold for completed dg algebras, provided one works cohomologically. In other words, if the differential preserves the completed arrow ideal, then every dg module with finite-dimensional total cohomology is nilpotent with respect to the induced action of the completed arrow ideal on cohomology.
\end{remark}

\subsection{Comparison to global (homological) mirror symmetry}\label{subsec:globalms}

When a global generator is available and $\bL$ is an unobstructed object, the localized mirror construction admits a comparison with a derived completion. This identifies the part of a global algebra seen by the chosen compact brane; a global generator is not needed to define the local construction itself. Here ``local'' refers to the categorical completion, not to an assertion that all ambient Floer operations are determined by a symplectic neighborhood. We retain the unobstructed hypotheses of this section. The following definitions adapt \cite[Definition 2.3.1]{BGO25}, stated there for dgas, to $A_\infty$-algebras.


\begin{defn}\label{def:centcomp} Let $A$ be an $A_\infty$-algebra, and $M$ a module over $A$. 
\begin{itemize}
\item[(i)]
The dga $A^!_M:=\Hom_{\mathrm{Mod}_{A_\infty} (A)} (M, M)$ is called the centralizer of $A$ with respect to $M$. 
\item[(ii)]
$M$ can be viewed as a module over $A^!_M$. We call
$$ A^{!!}_M := (A_M^!)_M^! = \Hom_{\mathrm{Mod}_{\dg} (A_M^!)} (M,M)$$
the derived completion of $A$ along $M$. 
\end{itemize}
\end{defn}

Observe that the centralizer recovers the Koszul dual when $M=\Bbbk$, and in this case, the double dual is known to be equivalent to the derived completion at a point \cite{Booth1}.
We can describe our local (extended) mirror $\tilde{\mathcal{A}}_\mathbb{L}$ in terms of the ``global mirror" using this term.
Suppose there exists an object $G$ that generates $\Fuk(X)$, and
let $\mathcal{A}_G$ denote its endomorphism $\Hom_{\Fuk(X)} (G,G)$. Then $\Fuk(X) \simeq \mathrm{Mod}_{A_{\infty}} (G)$, and hence $\mathcal{A}_G$ can intuitively be thought of as a global noncommutative mirror (determined up to Morita equivalence). 

\begin{prop} If $G$ generates $\Fuk(X)$, then the functor $\cF^{(\bL,\tilde{b})}$ derived-completes the algebra $\Hom_{\Fuk(X)} (G,G)$ in the sense that 
$$\cF^{(\bL,\tilde{b})}:  \Hom_{\Fuk(X)} (G,G) \to \Hom_{V^!} (\cF^{(\bL,\tilde{b})}(G),\cF^{(\bL,\tilde{b})}(G) ) \cong( \Hom_{\Fuk(X)} (G,G) )_M^{!!}$$ 
where $M = CF(G,\mathbb{L})$.
\end{prop}

\begin{proof}
We denote $  \Hom_{\Fuk(X)} (G,G)$ by $E_G$ for simplicity.
Let $M:=CF(G,\mathbb{L})$ be the left $E_G$-module obtained from $\mathbb{L}$.
Since $G$ generates, its associated Yoneda functor is fully faithful, and hence 
$$ V \cong \Hom_{\mathrm{Mod}_{A_\infty} (E_G) } (M,M)= (E_G)_M^{!}.$$
Therefore
\begin{equation*}
\begin{array}{rcl}
(E_G)_M^{!!} &=& \Hom_{\mathrm{Mod}_{A_\infty} (V)} (M,M) \\
&=& \Hom_V (M \otimes_\Bbbk \bar{B}V \otimes_\Bbbk V, M ) \\
&\cong& M \otimes_\Bbbk (M \otimes_\Bbbk \bar{B}V)^\sharp\\
 &\cong& M \otimes_\Bbbk V^! \otimes_\Bbbk M^\sharp
\end{array}
\end{equation*}
where we used finite-dimensionality of $M$. 
On the other hand, by definition of the functor, $\Hom_{V^!} (\cF^{(\bL,\tilde{b})}(G),\cF^{(\bL,\tilde{b})}(G) )$ is given by
$$\Hom_{V^!} (\cF^{(\bL,\tilde{b})}(G),\cF^{(\bL,\tilde{b})}(G) )= \Hom_{V^!} (V^! \otimes_\Bbbk M^\sharp, V^! \otimes_\Bbbk M^\sharp ) \cong M \otimes_\Bbbk V^! \otimes_\Bbbk M^\sharp.$$
\end{proof}
 
\begin{remark}\label{rmk:Gint1pt}
Suppose further that $G \in \Fuk(X)$ satisfies $Hom_{\Fuk(X)} (\mathbb{L}, G) \cong \Bbbk$. Then it is straightforward to check its image under $\cF^{(\bL,\tilde{b})}$ is $V^!=\tilde{\mathcal{A}}_{\mathbb{L}}$. This is analogous to transferring a Lagrangian section $G$ of a SYZ fibration (with $\bL$ being one of fibers) to the structure sheaf on the mirror side.
\end{remark}


\noindent\emph{Scope under bulk deformation.}
The module $M=\CF(G,\bL)$ used in this comparison is obtained from $\bL$ as an unobstructed object. If bulk deformation gives $\bL$ uncancelled curvature, one cannot apply this same uncurved comparison to the undeformed brane without first specifying an unobstructing deformation. Section~\ref{bulk_deform} instead starts with the curved Floer operations and imposes their Maurer--Cartan relations. The distinction concerns the reference brane and its deformation problem; it does not exclude a global algebraic description of the bulk-deformed wrapped category, such as those with bulk parameters in~\cite{EL19}.

\subsection{Reduction to unextended mirror $A_\mathbb{L}$ under certain formality}
Recall from \cite{CHL21} that one also has an $A_\infty$ functor
$$ \mathcal{F}^{(\mathbb{L},b)} : \Fuk(X)  \to \mathrm{Mod}_{\dg} (\cA_{\mathbb{L}})$$
where $b$ consists of degree-1 elements only. In this case, $\mathrm{Mod}_{\dg} (\cA_{\mathbb{L}})$ is simply the dg category of complexes of $\cA_\mathbb{L}$-modules.\footnote{This is intuitively extracting the underlying topological space (``$\mathrm{Spec}\, \cA_{\mathbb{L}}$") from a larger derived space (``$\mathrm{Spec}\,\tilde{\mathcal{A}}_{\mathbb{L}}$").}
As expected, the same spectral sequence argument as in Section \ref{subsubsec:specseq} shows that this suffices, provided $\tilde{\mathcal{A}}_{\mathbb{L}}$ is quasi-isomorphic to $\cA_{\mathbb{L}} = H^0(\tilde{\mathcal{A}}_{\mathbb{L}})$. 

\begin{lemma}\label{lem:equiv}
Suppose 
$\tilde{\mathcal{A}}_{\mathbb{L}} \cong \cA_\mathbb{L}$. Then the (unextended) mirror functor gives rise to a quasi-equivalence
$$\mathcal{D} \mathcal{F}^{(\mathbb{L},b)}: \mathcal{D} \Fuk_\mathbb{L} (X) \to \mathcal{D}_{\mathrm{nil}}(\cA_\bL),$$ where $\mathcal{D}_{\mathrm{nil}}(\cA_\bL):=\operatorname{thick}_{\mathcal D( \cA_{\mathbb L})}(\Bbbk).$
\end{lemma}


\begin{proof}
Let $V=CF(\mathbb{L},\mathbb{L})$ and $M=CF(U,\mathbb{L})$ as before.
Observe that the functor restricted to $\Fuk_{\mathbb{L}} (X)$ can be understood abstractly as 
$$ \mathcal{F}^{(\mathbb{L},b)} : \mathrm{Mod}_{A_\infty} (V) \to \mathrm{Mod}_{\dg} (\cA_{\mathbb{L}}), \qquad M \mapsto H^0 (V^!) \otimes_\Bbbk  M^\sharp.$$
Here, $H^0 (V^!) \otimes_\Bbbk  M^\sharp$ is the complex obtained as follows. Recall that the image of $M$ under $\mathcal{F}^{(\mathbb{L},\tilde{b})}$ can be identified as $V^! \otimes_\Bbbk  M^\sharp$, and as in \eqref{eqn:funcdouble}, the differential on it consists of two operators, $d= d_{V^!} \otimes id + m_1^{\tilde{b},0}$. Viewing $\mathcal{F}^{(\mathbb{L},\tilde{b})} = V^! \otimes_\Bbbk  M^\sharp$ as a double complex this way, $\mathcal{F}^{(\mathbb{L},b)} (M)$ is precisely its $E_1$-page (with the differential $d_{E_1}$), since $H^{\neq 0} (V^!)=0$ by hypothesis. The $E_2$-page computes the cohomology of $\mathcal{F}^{(\mathbb{L},\tilde{b})}$.

Now, as in the proof of Theorem \ref{thm:extfuctorequiv},
\begin{equation*}
\begin{array}{rcl}
\Hom_{H^0 (V^!)} (\mathcal{F}^{(\mathbb{L},b)} (V), \mathcal{F}^{(\mathbb{L},b)} (V)) &=& \Hom_{H^0(V^!)} ( H^0(V^!) \otimes_\Bbbk  V^\sharp, H^0 (V^!) \otimes_\Bbbk  V^\sharp ) \\
 & =& \mathrm{RHom}_{H^0(V^!)} (H^0(V^!) \otimes_\Bbbk  V^\sharp, H^0(V^!) \otimes_\Bbbk  V^\sharp )  \\
 &\cong& \mathrm{RHom}_{V^!} (\Bbbk, \Bbbk ).
 \end{array}
 \end{equation*}
The last line uses the fact that $H^0(V^!) \otimes V^\sharp$ is quasi-isomorphic to $\Bbbk$ which follows from $V^! \otimes_\Bbbk  V^\sharp \stackrel{qis}{\simeq} \Bbbk$ by applying the above spectral sequence for $M=V$.

\end{proof}

A family of extended localized mirrors satisfying the formality come from the core of plumbing spaces. More discussions can be found in Section \ref{subsec: lms}.

\begin{remark}
	Theorem \ref{thm:extfuctorequiv} and Lemma \ref{lem:equiv} can be understood as the mirror counterpart of Theorem 6.8 in \cite{Tod18}, which shows that there is an equivalence of categories $$I_*: \mathrm{mod}_{\mathrm{nil}} A_{E_{\bullet}} \to \langle E_1, \ldots, E_k \rangle \subset \mathrm{Coh}(Y), $$ where $E_{\bullet}=\{E_1, \ldots, E_k\}$ is a simple collection of coherent sheaves and $A_{E_{\bullet}}$ is the corresponding deformation quiver algebra, encoding the relations of $\mathrm{Ext}^*(\oplus_i E_i,\oplus_i E_i)$. In particular, $\mathcal{D}^b(\mathrm{mod}_{\mathrm{nil}} A_{E_{\bullet}}) \cong \operatorname{thick}_{\mathcal D( A_{E_{\bullet}})}(\Bbbk)$, and $I_*$ sends the $i$-th simple representation $S_i$ to the $i$-th simple sheaf $E_i$.
	
	From this perspective, the formal Maurer-Cartan deformation space of a Lagrangian immersion plays the mirror role of Toda's deformation quiver algebra. Thus, even when the mirror manifold is not explicitly known, the localized mirror construction provides a symplecto-geometric deformation space mirror to the moduli stack of semistable coherent sheaves.
\end{remark}

\section{Higher rank vector bundles on the Lagrangian $\bL$}\label{sec:higher}
We now consider a compact Lagrangian immersion equipped with a higher rank local system. By this, we mean a Lagrangian immersion $L \looparrowright \mathbb{L}$, and a flat vector bundle $(\mathcal{E},\nabla)$ on $L$ whose rank is allowed to be greater than one. Higher rank flat vector bundles encode information about the fundamental group $\pi_1(L)$ while rank-one local systems are Abelian and only concern about the homology group $H_1(L)$.

If $L$ has several connected components, equivalently if $\mathbb{L}=\bigoplus_i L_i$ has several irreducible components, then we allow the restrictions $\mathcal{E}|_{L_i}$ to have different ranks.

Let us consider its Fukaya algebra $\CF((\bL,\mathcal{E}),(\bL,\mathcal{E}))$. More specifically, one may use the Morse model by choosing a Morse function $f$ on $L$, in which case
$$\CF((\bL,\mathcal{E}),(\bL,\mathcal{E}))=\bigoplus_{p \in \mathrm{crit}(f)} \mathrm{End} (\mathcal{E}_p) \oplus \bigoplus_{(q_+, q_-) \in L \times_{\iota} L} \Hom (E_{q_+}, E_{q_-}).$$
The $A_\infty$-operations additionally involve parallel transports with respect to $\nabla$ along the boundary of contributing pearl trajectories. For example, a (negative) gradient flow $\gamma$ from $p$ to $q \in \mathrm{crit}(f)$ contributes $P_\gamma^\nabla \circ A  \circ (P_\gamma^\nabla)^{-1} \in   \mathrm{End} (\mathcal{E}_q)$ to the Morse differential of $A \in \mathrm{End} (\mathcal{E}_p)$ (which is part of $m_1 (A)$), where $P_{\gamma}^\nabla$ denotes the parallel transport of $(\mathcal{E},\nabla)$ along $\gamma$. (In \ref{subsec:matcoeff}, we give a simpler formulation of parallel transport with respect to $\nabla$ along the boundary paths of pearl trajectories.)

We can extend the construction of \cite{CHL17} to obtain the localized mirror functor $\cF^{(\mathbb{L},\mathcal{E})}$ associated to $(\mathbb{L},\mathcal{E},b)$ which transforms Lagrangians in $X$ to complexes over the corresponding Maurer-Cartan quiver moduli (see \cite[Definition 3.3]{LT26}).
The purpose of this section is to use Koszul duality to deduce a quasi-equivalence statement for this functor.

Recall that, in the rank-one setting, we used the semisimple ring $\Bbbk = \oplus_i \mathbb{C}$ as the coefficient ring for the Floer complex with one copy of $\mathbb{C}$ assigned to each irreducible component $L_i$ of $\mathbb{L}$.
In the framework of Koszul duality, one may generalize \(\Bbbk\) in either of the following ways:
\begin{enumerate}
\item[(i)] by replacing the coefficient ring of the algebra \(V\) with a different coefficient ring;
\item[(ii)] by replacing the trivial \(V = \CF(\bL,\bL)\)-module \(\Bbbk\) with a more general \(V\)-module.
\end{enumerate}
We begin by choosing a convenient representative of the flat connection $\nabla$ within its gauge-equivalence class.

\subsection{Gauge hypersurfaces for $\nabla$}\label{subsec:gaugehyp}
We first describe the holonomy of $\mathcal{E}$ by concentrating the nontrivial parallel transport near a chosen finite collection of codimension-one chains chosen using Morse theory or handle decomposition. We explain the construction only in the case where $\mathbb{L}$ has a single irreducible component, and the general case follows by applying the same construction to each component separately.

Let $\widetilde{f}$ be a Morse function with exactly one maximum point on $\mathbb{L}$ chosen generically with respect to the original Morse function $f$, and let $h_1,\ldots,h_r$ be the compactified unstable chains parameterized by manifolds with corners associated with its cohomological degree-$1$ critical points. 

We specify holonomy matrices across each of $h_i$ as follows. The complement $\mathbb{L}^\circ:=\mathbb{L}\setminus\bigcup_{i=1}^r h_i$ is the unstable submanifold of the unique maximum point and is simply connected. Indeed, the interiors of the $h_i$ are the cocores of the $1$-handles in the handle decomposition associated with $-f$, and cutting along them removes the generators introduced by these handles. Since $(\mathcal{E},\nabla)$ is flat, it is therefore flatly trivial over $\mathbb{L}^\circ$,
$$
\mathcal{E}|_{\mathbb{L}^\circ}\cong\mathbb{L}^\circ\times\mathbb{C}^d.
$$
After a gauge transformation, the nontrivial contributions of parallel transport are made to be concentrated in arbitrarily small neighborhoods of the hypersurfaces $h_i$. Accordingly, the parallel transport is encoded by automorphisms
$$
\rho_i\in\operatorname{Aut}(\mathbb{C}^d),
$$
with a factor of $\rho_i$ or $\rho_i^{-1}$ whenever a path crosses $h_i$ with positive or negative local intersection number, respectively.
$\rho_i$'s are required to satisfy compatibility relations specified by the unstable chains of degree-$2$ critical points: they are of the form $R(\rho_1,\ldots,\rho_p) = I$ for $p$ codimension-one walls joint at a codimension-two unstable chain.

We refer to $h_1,\ldots,h_r$ as the \emph{gauge hypersurfaces} for $\nabla$. This construction generalizes the gauge hypertori used in \cite{CHL-toric} when $\mathbb{L}\cong T^n$.
The $J$-holomorphic discs and pearl trajectories are required to be transverse to these unstable submanifolds by generic perturbations. In the inductive perturbation scheme of \cite{FOOO09, FOOO-can}, these are included in the set of the initial chains.

\begin{remark}
Under an additional assumption on the holonomy, one may instead use smooth embedded hypersurfaces. Suppose that the holonomy representation of $(\mathcal{E},\nabla)$ factors through $H_1(\mathbb{L};\mathbb{Z})/\operatorname{Tor}$ (which holds true when $\mathcal{E}$ is a line bundle).
Choose smooth maps
$
g_1,\ldots,g_r:\mathbb{L}\rightarrow S^1
$
representing a basis of $H^1(\mathbb{L};\mathbb{Z})$; smooth representatives exist by the Whitney approximation theorem. For each $i$, choose a regular value $z_i\in S^1$ and set
$
h_i:=g_i^{-1}(z_i).
$
Then the $h_i$ are smooth, cooriented hypersurfaces whose homology classes form a basis of the free part of $H_{n-1}(\mathbb{L};\mathbb{Z})$.
Any loop contained in $\mathbb{L}\setminus\bigcup_i h_i$ has zero algebraic intersection with every $h_i$ and hence represents a torsion class in $H_1(\mathbb{L};\mathbb{Z})$. 
\end{remark}

\subsection{The matrix coefficient ring}\label{subsec:matcoeff}
In the approach (i), we enlarge the coefficient ring $\Bbbk$ of $\CF((\mathbb L,\mathcal{E}),(\mathbb L,\mathcal{E}))$ to a suitable matrix algebra so that the resulting algebraic structure more faithfully records the effect of $\mathcal{E}$ on Floer theory. 
By genericity, we may assume that all critical points of $f$ and all self-intersection points of $\mathbb{L}$ lie outside $\bigcup_i h_i$. Consequently, the chosen trivialization of $(\mathcal{E},\nabla)$ over $\mathbb{L}\setminus\bigcup_i h_i$ as in \ref{subsec:gaugehyp} induces an identification
\begin{equation*}
\begin{aligned}
\CF\bigl((\bL,\mathcal{E}),(\bL,\mathcal{E})\bigr)
&=
\bigoplus_{p\in\operatorname{crit}(f)}
\operatorname{End}(\mathcal{E}_p)
\oplus
\bigoplus_{(q_+,q_-)\in L\times_{\iota}L}
\operatorname{Hom}(\mathcal{E}_{q_+},\mathcal{E}_{q_-})
\\
&\cong
\bigoplus_{p\in\operatorname{crit}(f)}
\operatorname{End}(\mathbb{C}^d)\langle p\rangle
\oplus
\bigoplus_{(q_+,q_-)\in L\times_{\iota}L}
\operatorname{End}(\mathbb{C}^d)\langle(q_+,q_-)\rangle ,
\end{aligned}
\end{equation*}
where we have assumed that $\mathbb{L}$ has a single irreducible component and that $\mathcal{E}$ has rank $d$. The case of several irreducible components is treated similarly.
The operations $m_k$ on $\CF\bigl((\bL,\mathcal{E}),(\bL,\mathcal{E})\bigr)$ are defined as before. The parallel-transport factors, however, are now much easier to compute as one only needs to record the intersections of the boundary segments of a pearl trajectory with the hypersurfaces $h_i$.

Recall that the unique maximum $p_{\max}$ of the Morse function represents the unit and spans the degree-$0$ part of the underlying Morse complex. We set
$$
\widetilde{\Bbbk}_{\mathcal{E}}
:=
\operatorname{End}(\mathbb{C}^d)\langle p_{\max}\rangle
\cong
\mathfrak{gl}_d (\mathbb{C}).
$$
Obviously, $\widetilde{\Bbbk}_{\mathcal{E}}$ is closed under $m_2$, which agrees with ordinary matrix multiplication. More generally, if $\bL=\bigcup_i L_i$ has several irreducible components and $\mathcal{E}|_{L_i}$ has rank $d_i$, then
$$
\widetilde{\Bbbk}_{\mathcal{E}}
:=
\bigoplus_i
\operatorname{End}(\mathbb{C}^{d_i})\langle p_{\max,i}\rangle
\cong
\bigoplus_i\mathfrak{gl}_{d_i}(\mathbb{C}).
$$

Even when $\bL$ is exact and $f$ is perfect, the resulting Morse model need not be minimal. Indeed, the classical differential counting isolated gradient trajectories is twisted by conjugation with the parallel transport of $\nabla$. In the above trivialization, this parallel transport is given by an ordered product of the matrices $\rho_i^{\pm1}$, determined by the signed intersections of the trajectory with the hypersurfaces $h_i$. In particular, after passing to a minimal model, its degree-$0$ component is identified with the centralizer in $\widetilde{\Bbbk}_{\mathcal{E}}$ of the subgroup
$$
G\subset \bigoplus_i \operatorname{Aut}(\mathbb{C}^{d_i})
$$
generated by the holonomy matrices $\rho_j$. We denote this algebra by
$$
\Bbbk_{\mathcal{E}}
:=Z_{\widetilde{\Bbbk}_{\mathcal{E}}}(G),
$$
and take it as our new coefficient ring.

From now on, we use the same notation $\CF((\bL,\mathcal{E}),(\bL,\mathcal{E}))$ for its minimal model, and denote it by $V_{\mathcal{E}}:=\CF((\bL,\mathcal{E}),(\bL,\mathcal{E}))$ for simplicity.
The $\Bbbk_{\mathcal{E}}$-bimodule structure on $V_{\mathcal{E}}$ is defined by
$$ U \cdot X \cdot V = m_2(m_2 (U, X) ,V) =  m_2( U, m_2(X ,V))$$
for $U,V \in \Bbbk_{\mathcal{E}}$, where the second equality comes from the $A_\infty$-relation and the unital property. Let us write $V_{\mathcal{E}} : =\CF((\bL,\mathcal{E}),(\bL,\mathcal{E}))$ for simplicity. By the same argument, we have

\begin{lemma} The $A_\infty$-operations on $V_{\mathcal{E}}$ are linear on $V_{\mathcal{E}} \otimes_{\Bbbk_{\mathcal{E}}} \cdots \otimes_{\Bbbk_{\mathcal{E}}} V_{\mathcal{E}}$ in the sense that
$$ m_k (X_1,\cdots, (X_{i-1} \cdot U), X_i ,\cdots, X_k) = m_k (X_1,\cdots, X_{i-1}, (U \cdot X_i) ,\cdots, X_k)$$
for $U \in \Bbbk_{\mathcal{E}}$.
\end{lemma}
Moreover we have the augmentation $V_{\mathcal{E}} \to \Bbbk_{\mathcal{E}}$ by projecting to the unit component. We assume that all the immersed generators have positive degrees, analogously to Assumption \ref{assume:onv}. The bar construction naturally extends to give
$$ B V_{\mathcal{E}} := \bigoplus_k ((V_{\mathcal{E}})_{>0}[1])^{\otimes_{\Bbbk_{\mathcal{E}}} k},$$
where $(V_{\mathcal{E}})_{>0}$ is the augmentation ideal generated by positive degree elements of $V_{\mathcal{E}}$. Notice that $B V_{\mathcal{E}}$ also has a $\Bbbk_{\mathcal{E}}$-bimodule structure and that the bar differential is $\Bbbk_{\mathcal{E}}$-linear. In particular, one can consider the (graded) linear dual as before,
$$ V_{\mathcal{E}}^! := \oplus_i \Hom_{\Bbbk_{\mathcal{E}}} ( (B V_{\mathcal{E}})_i ,\Bbbk_{\mathcal{E}} )  $$
where we view both $(B V_{\mathcal{E}})_i$ and $\Bbbk_{\mathcal{E}}$ as left $\Bbbk_{\mathcal{E}}$-modules. As before, this can also be understood as the set of pre-$A_\infty$ module homomorphisms from the right $V_{\mathcal{E}}$-module $\Bbbk_{\mathcal{E}}$ to itself.  

We next extend the Koszul functor to this higher-rank setup.
For a given (left) $V_{\mathcal{E}}$-module $M$, consider the correspondence
$$ M \mapsto \Hom_{\Bbbk_{\mathcal{E}}}  (M \otimes_{\Bbbk_\mathcal{E}} BV_\mathcal{E}, \Bbbk_{\mathcal{E}})$$
which assigns $M$ the set of pre-$A_\infty$ module homomorphism from $M$ to $\Bbbk_{\mathcal{E}}$.
In particular, if we take $M= V_\mathcal{E}$, then the image of the above functor is cohomologically the set of $A_\infty$-homomorphisms from $V$ to $\Bbbk_\mathcal{E}$ extending the augmentation (which should obviously be the augmentation itself up to scale). Hence we see that $V_{\mathcal{E}}$ maps to $\Bbbk_{\mathcal{E}}$ under this functor.

On the other hand, the usual contracting homotopy (inserting the unit class to the $V_{\mathcal{E}}$-factor) still makes sense in the current setup to give a homotopy equivalence
$$ V_{\mathcal{E}} \otimes_{\Bbbk_\mathcal{E}}  BV_\mathcal{E}  \cong \Bbbk_\mathcal{E}$$
which dualizes to 
$$ V_{\mathcal{E}}^! \otimes_{\Bbbk_\mathcal{E}} V_{\mathcal{E}}^{\sharp} \cong \Bbbk_\mathcal{E}$$ as left $V_{\mathcal{E}}^! $-modules. Hence, we have
 $$\mathrm{RHom}_{V_{\mathcal{E}}^!} (\Bbbk_{\mathcal{E}}, \Bbbk_{\mathcal{E}} ) = \mathrm{Hom}_{V_{\mathcal{E}}^!} (V^! \otimes_{\Bbbk_{\mathcal{E}}}  V^\sharp, \Bbbk_{\mathcal{E}} )   
 \cong   \Hom_{\Bbbk_{\mathcal{E}}} ( V_{\mathcal{E}}^\sharp, \Bbbk_{\mathcal{E}} ) \cong V_{\mathcal{E}} \cong \hom_{V_\mathcal{E}} (V_\mathcal{E},V_\mathcal{E}),$$
 which establishes an equivalence between the sub-Fukaya category generated by $(\mathbb{L},\mathcal{E})$ and the category of finite dimensional $V_\mathcal{E}^!$-modules.
%
%

\subsection{The new ``trivial" module $\Bbbk_{\vec{d}}$}
Approach (ii) requires each local system $\mathcal{E}|_{L_i}$ to be trivial up to gauge equivalence. This condition is automatic when every component $L_i$ of $\bL$ is simply connected, as will be the case in our later study of plumbings of cotangent bundles of spheres. In this situation, the isomorphism class of $\mathcal{E}|_{L_i}$ is determined entirely by its rank, which we denote by $d_i$.


We then assign $\mathbb{C}^{d_i}$ over each idempotent $\pi_i$ of $\Bbbk = \oplus_i \mathbb{C}$, which results in a $\Bbbk$-module $\Bbbk_{\vec{d}}:= \oplus_i \mathbb{C}^{d_i}$. It can also be viewed as a module over $V = \CF(\mathbb{L},\mathbb{L})$, since so is $\Bbbk$ via the augmentation. This has a higher rank, but is still ``trivial" in the sense that higher $A_\infty$-operations (in the $V$-module structure) vanish. As mentioned, we will replace the trivial $V$-module $\Bbbk$ by this new $V$-module $\Bbbk_{\vec{d}}$. Thus we may consider the centralizer of $V$ with respect to  $\Bbbk_{\vec{d}}$ (Definition \ref{def:centcomp}),
$$V_{\Bbbk_{\vec{d}}}^! = \Hom_{V} (\Bbbk_{\vec{d}},\Bbbk_{\vec{d}}) \cong \Hom_\Bbbk ( \Bbbk_{\vec{d}} \otimes_\Bbbk  \bar{B}V, \Bbbk_{\vec{d}}) \cong \Bbbk_{\vec{d}} \otimes_\Bbbk V^! \otimes_\Bbbk \Bbbk_{\vec{d}}^\sharp$$
We denote this by $\tilde{\mathcal{A}}_{\mathbb{L}_\mathcal{E}}$.

For a basis element $X \in \pi_i \cdot V \cdot \pi_j$, define its associated coordinate function $x_{ab}$ (for $1\leq a \leq i$ and $1 \leq b\leq j$) to be an element of $V_{\Bbbk_{\vec{d}}}^!$ such that
$$x_{ab}: e_a \otimes X  ( \in  \mathbb{C}^{d_i} \otimes \langle X \rangle )\mapsto  e_b  (\in  \mathbb{C}^{d_j} )$$
and maps other elements to zero,
where $e_a$ and $e_b$ denote $a$-th and $b$-th standard basis vectors of $\mathbb{C}^{d_i}$ and $\mathbb{C}^{d_j}$, respectively. $\tilde{\mathcal{A}}_{\mathbb{L}_\mathcal{E}}$ is generated by matrices with such a coordinate function in each entry.

As before, there is a natural functor
$$\mathcal{F}^{\mathbb{L}_\mathcal{E}} : \Fuk(X) \to \mathrm{Mod}^{left}_{\dg} (\tilde{\mathcal{A}}_{\mathbb{L}_\mathcal{E}})$$
which can be algebraically described as the composition of the Yoneda functor $\mathcal{Y}_\bL$ and
$$\mathcal{F} : \mathrm{Mod}^{right}_{A_\infty} (V) \to \mathrm{Mod}^{left}_{\dg} (V_{\Bbbk_{\vec{d}}}^!),\quad M \mapsto \Hom_{\mathrm{Mod}_{A_\infty} (V)} (M, \Bbbk_{\vec{d}}) = \Hom_\Bbbk (M \otimes_\Bbbk \bar{B}V,  \Bbbk_{\vec{d}})$$
for $V=\CF(\mathbb{L},\mathbb{L})$.
It is easy to check that the functor sends $\mathbb{L}$ itself to $\Bbbk_{\vec{d}}$, or $\mathcal{F}(V) = \Bbbk_{\vec{d}}$. 


Let us consider $\Fuk_{\mathbb{L}_\mathcal{E}} (X)$, the subcategory of $\Fuk(X)$ split-generated by $\mathbb{L}_\mathcal{E}:=(\mathbb{L},\mathcal{E})$.\footnote{Strictly speaking $\Fuk(X)$ does not contain Lagrangians equipped with higher-rank local systems, but one can easily generalize it to include  those.} 
On the mirror side, we take the subcategory split-generated by $\mathrm{End}(\Bbbk_{\vec{d}}):=\Bbbk_{\vec{d}}  \otimes_\Bbbk  \Bbbk_{\vec{d}}^\sharp$, which is obviously a left module over $V_{\Bbbk_{\vec{d}}}^! = \Hom_{V} (\Bbbk_{\vec{d}},\Bbbk_{\vec{d}}).$
We claim that the functor restricts to a quasi-equivalence between these subcategories.

\begin{lemma}
$\mathcal{F}^{\mathbb{L}_\mathcal{E}}$ induces a quasi-equivalence
\begin{equation}\label{eqn:functhr}
 \mathcal{D} \Fuk_{\mathbb{L}_\mathcal{E}} (X) \stackrel{\simeq}{\longrightarrow} \operatorname{thick}_{\mathcal{D}(\tilde{\mathcal{A}}_{\mathbb{L}_\mathcal{E}})}(\mathrm{End}(\Bbbk_{\vec{d}})).
\end{equation}
\end{lemma}

\begin{proof}
The image of $(\mathbb{L},\mathcal{E})$ under the mirror functor is
\begin{equation*}
\begin{array}{rcl}
 \mathcal{F}^{\mathbb{L}_\mathcal{E}} (\mathbb{L},\mathcal{E})=\Hom_V ( \Bbbk_{\vec{d}} \otimes_\Bbbk V,\Bbbk_{\vec{d}})&=& \Hom_\Bbbk ( \Bbbk_{\vec{d}}  \otimes_\Bbbk V \otimes_\Bbbk \bar{B}V,  \Bbbk_{\vec{d}}) \\
 &=& \Bbbk_{\vec{d}} \otimes_\Bbbk V^! \otimes_\Bbbk V^\sharp \otimes_\Bbbk  \Bbbk_{\vec{d}}^\sharp \\
  &\stackrel{qis}{\simeq}&  \Bbbk_{\vec{d}}  \otimes_\Bbbk  \Bbbk_{\vec{d}}^\sharp =\mathrm{End}(\Bbbk_{\vec{d}}),
\end{array}
\end{equation*}  
and the restricted functor \eqref{eqn:functhr} is well-defined.
Let us next compare their endomorphism spaces.
Observe first that the endomorphism of $\mathbb{L}_\mathcal{E}$ in the Fukaya category is given by
\begin{equation}\label{eqn:endohrobj}
\begin{array}{rcl}
\Hom_{\Fuk(X)} ((\mathbb{L},\mathcal{E}),(\mathbb{L},\mathcal{E}))=\Hom_V (\Bbbk_{\vec{d}} \otimes_\Bbbk V, \Bbbk_{\vec{d}} \otimes_\Bbbk V) &=& \Hom_\Bbbk (\Bbbk_{\vec{d}} \otimes_\Bbbk  V \otimes_\Bbbk \bar{B}V, \Bbbk_{\vec{d}} \otimes_\Bbbk V) \\
&=& \Bbbk_{\vec{d}} \otimes_\Bbbk V \otimes_\Bbbk V^! \otimes_\Bbbk  V^\sharp \otimes_\Bbbk  \Bbbk_{\vec{d}}^\sharp \\
&\stackrel{qis}{\simeq}& \Bbbk_{\vec{d}} \otimes_\Bbbk V \otimes_\Bbbk  \Bbbk_{\vec{d}}^\sharp.
\end{array}
\end{equation}
On the other hand, the endomorphism of $ \mathcal{F}^{\mathbb{L}_\mathcal{E}} (\mathbb{L},\mathcal{E})$ can be computed as
\begin{equation*}
\begin{array}{rl} 
& \Hom_{\mathrm{Mod}_{\dg} (\tilde{\mathcal{A}}_{\mathbb{L}_\mathcal{E}})} ( \mathcal{F}^{\mathbb{L}_\mathcal{E}} (\mathbb{L},\mathcal{E}), \mathcal{F}^{\mathbb{L}_\mathcal{E}} (\mathbb{L},\mathcal{E}))\\
=&\Hom_{V_{\Bbbk_{\vec{d}}}^!} (  \Bbbk_{\vec{d}} \otimes_\Bbbk V^! \otimes_\Bbbk V^\sharp \otimes_\Bbbk  \Bbbk_{\vec{d}}^\sharp \,\,,\,  \Bbbk_{\vec{d}} \otimes_\Bbbk V^! \otimes_\Bbbk V^\sharp \otimes_\Bbbk  \Bbbk_{\vec{d}}^\sharp) \\
=& \Hom_{V_{\Bbbk_{\vec{d}}}^!} ( \underbrace{ \Bbbk_{\vec{d}} \otimes_\Bbbk V^! \otimes_\Bbbk \Bbbk_{\vec{d}}^\sharp }_{=V_{\Bbbk_{\vec{d}}}^!} \otimes_{\mathrm{End} (\Bbbk_{\vec{d}})} \Bbbk_{\vec{d}} \otimes_\Bbbk V^\sharp \otimes_\Bbbk  \Bbbk_{\vec{d}}^\sharp \,\,,\,  \Bbbk_{\vec{d}} \otimes_\Bbbk V^! \otimes_\Bbbk V^\sharp \otimes_\Bbbk  \Bbbk_{\vec{d}}^\sharp) \\
=& \Hom_{\mathrm{End}(\Bbbk_{\vec{d}})} (  \Bbbk_{\vec{d}} \otimes_\Bbbk V^\sharp \otimes_\Bbbk  \Bbbk_{\vec{d}}^\sharp \,\,,\,  \Bbbk_{\vec{d}} \otimes_\Bbbk V^! \otimes_\Bbbk V^\sharp \otimes_\Bbbk  \Bbbk_{\vec{d}}^\sharp) \\
=& \Hom_{\Bbbk} (    V^\sharp \otimes_\Bbbk  \Bbbk_{\vec{d}}^\sharp \,\, , \, V^! \otimes_\Bbbk V^\sharp \otimes_\Bbbk  \Bbbk_{\vec{d}}^\sharp) 
\end{array}
\end{equation*}
where we used
\begin{itemize}
\item[(i)] the base change formula,
\(\text{Hom}_{S}( S \otimes _{R}M,P)\cong \text{Hom}_{R}(M,P)\), and
\item[(ii)] the Morita equivalence (in the last line, since $\Bbbk_{\vec{d}}$ is projective over $\Bbbk$).
\end{itemize}
Moreover, using $V^\sharp   \otimes_\Bbbk V^!  \stackrel{qis}{\simeq} \Bbbk$ from Lemma \ref{lem:leftrightresol}, the above becomes
\begin{equation*}
\begin{array}{rcl}
 \Hom_{\Bbbk} (    V^\sharp \otimes_\Bbbk  \Bbbk_{\vec{d}}^\sharp \,\, , \, V^! \otimes_\Bbbk V^\sharp \otimes_\Bbbk  \Bbbk_{\vec{d}}^\sharp) &\stackrel{qis}\simeq&  \Hom_{\Bbbk} (    V^\sharp \otimes_\Bbbk  \Bbbk_{\vec{d}}^\sharp \,\, , \,    \Bbbk_{\vec{d}}^\sharp)  \\
&=&  \Bbbk_{\vec{d}} \otimes_\Bbbk V \otimes_\Bbbk  \Bbbk_{\vec{d}}^\sharp \,\, .
\end{array}
\end{equation*}
Thus we have 
$$ \Hom_{\mathrm{Mod}(V^!_{\Bbbk_{\vec{d}}})} ( \mathcal{F}^{\mathbb{L}_\mathcal{E}} (\mathbb{L},\mathcal{E}), \mathcal{F}^{\mathbb{L}_\mathcal{E}} (\mathbb{L},\mathcal{E})) \stackrel{qis}{\simeq}  \Bbbk_{\vec{d}} \otimes_\Bbbk V \otimes_\Bbbk  \Bbbk_{\vec{d}}^\sharp$$
which agrees with \eqref{eqn:endohrobj}.
\end{proof}


\section{Application to the Fukaya category of plumbing spaces $T^\ast \mathbb{S}^n$}\label{sec:fukplumb}
In this section, we provide a simple application, namely, a version of homological mirror symmetry for plumbing spaces will be deduced from Theorem \ref{thm:extfuctorequiv}. Moreover, when there is no cycle in the plumbing graph, it implies that the zero section (core) split generates the compact Fukaya category. On the other hand, we also exhibit examples showing that, in the presence of cycles, the core Lagrangian in general fails to split generate the compact Fukaya category.

Let's first briefly recall the construction of plumbing spaces which are Liouville manifolds obtained by performing surgeries on cotangent bundles. 
Let $D$ be a graph with vertex set $V$ and edge set $E$. 
For each vertex $v \in V$, let $M_v$ be a Riemannian manifold, and consider its disk cotangent bundle $D^*M_v$. For each edge $e = (v,w) \in E$, choose open balls $B_v^e \subset M_v$ and $B_w^e \subset M_w$, together with a symplectomorphism between the corresponding disk cotangent bundles $
D^*B_v^e \cong D^*B_w^e$
given by the identification
$ (x_v, y_v) \mapsto (-y_w, x_w)$,
where $x$ and $y$ denote the base and fiber coordinates, respectively.

Gluing all $D^*M_v$ along these identifications for each edge $e \in E$, one obtains a Weinstein domain with corners. After rounding the corners, its completion defines a Liouville manifold, called the \emph{plumbing of cotangent bundles according to the graph $D$}. Moreover, each zero section $M_v \subset D^*M_v$ forms an embedded exact Lagrangian submanifold in the resulting Liouville manifold. In this section, we will mainly consider the plumbing space of $T^*\mathbb{S}^n$ and work with the standard exact plumbing model. We choose the Liouville
form in this model so that it restricts to zero on each core component and on the chosen cocores.

\subsection{Local mirror symmetry for plumbing spaces} \label{subsec: lms}

We describe the local mirror symmetry for plumbing spaces from the perspective of Koszul duality. 
The extended localized mirror of the core Lagrangian is modeled by a version of the Ginzburg differential graded algebra associated to the plumbing graph.

\begin{defn}
	Let $D=(I,E)$ be a graph. We associate a quiver $Q$ by:
	\begin{itemize}
		\item replacing each edge between $i$ and $j$ by a pair of arrows 
		$x_a: i \to j$ and $x_{a^*}: j \to i$,
		\item adding a loop $t_v$ at each vertex $v \in I$.
	\end{itemize}
	
	Fix an integer $n\in \mathbb{Z}$ and assign to each arrow $a$ an integer degree $q_a\in \mathbb{Z}_{\leq 0}$. 
	
	The \emph{Ginzburg differential graded algebra} $G_n(D)$ is the graded path algebra of $Q$ equipped with the following grading and differential.
	
	\begin{itemize}
		\item The grading is given by
		\[
		|x_a|=q_a,\qquad |x_{a^*}|=2-n-q_a,\qquad |t_v|=1-n.
		\]
		
		\item The differential $d$ is the derivation determined by
		\[
		d(x_a)= d(x_{a^*})=0,\qquad 
		d(t_v)=\sum_{h(a)=v} x_ax_{a^*}.
		\]
	\end{itemize}
     
     The \emph{completed Ginzburg differential graded algebra} 
     $\widehat{G}_n(D)$ is defined in the same way, replacing the graded
     path algebra by its completion with respect to the arrow ideal. 
     Equivalently,
     \[
     \widehat{G}_n(D)
     =
     \prod_{k\geq 0} \mathbb{C}Q_k,
     \]
     where $\mathbb{C}Q_k$ is the vector space spanned by paths of length $k$.
     The grading and differential are the extensions of those on
     $G_n(D)$.
\end{defn}


We now explain how the above algebra arises from the extended localized mirror construction for plumbing spaces.

\begin{prop}
	Let $X$ be the plumbing of $T^*\mathbb{S}^n$ according to a graph $D=(I,E)$ and $\bL$ be the core of $X$. Then the extended localized mirror of $\bL$ is the completed Ginzburg differential graded algebra $\widehat{G}_n(D)$ for $n\geq 2.$
	
\end{prop} 
\begin{proof}
	One can compute the extended localized mirror of $\bL$ by taking the Koszul dual. Here we provide a geometric proof using Proposition \ref{prop:kos}.
	
	We distinguish the cases $n=2$ and $n \geq 3$. Let $n=2$, and take
	$$\tilde{b} = \sum_a (x_a X_a + x_{\bar{a}}X_{\bar{a}}) + \sum_v t_v T_v  $$
	to be the formal deformation parameter, where $(X_a,X_{\bar{a}})$ is the immersed sector of degree 1 and $T_v$ is the minimum point of $\mathbb{S}^2_v$ for all $v \in I$.  Formally $\tilde{b}$ is made to be degree one by setting $\deg x_a = \deg x_{\bar{a}} = 0$ and $\deg t_v = -1$.  
	Then $m_0^{\tilde{b}}$ is in degree $2$ and takes the form
	$$m_0^{\tilde{b}} = \sum_v p_v T_v.  $$
	Since the coefficients have non-positive degrees, the generators $X_a$ and $X_{\bar{a}}$ do not appear for degree reasons.  Moreover, each $p_v$ must have degree $0$ and hence are series only in $x_a,x_{\bar{a}}$.
	
	The derivation $d$ is defined by $d x_a = d x_{\bar{a}} = 0$ and $d t_v = p_v$, and then extended by Leibniz rule.  The $A_\infty$ equations imply that $(\C \tilde{Q}, d)$ is a dg algebra.
	
	The only non-trivial term $dt_v = p_v$ counts Floer contributions to the minimal point of the Morse function in each sphere component, which was computed in Theorem 5.2 of \cite{HLT24}. Up to a coordinate change, one has $p_v= \sum_{t(a)=v} \epsilon(x_a) x_{\bar{a}}x_a$. Hence, $(\C \tilde{Q}, d)$ is the completed Ginzburg differential graded algebra $\widehat{G}_2(D)$. 
	
	When $n \geq 3$, we take $\tilde{b}$ as before. In this case, each intersection point gives rise to a pair of immersed sectors $(X_a,X_{\bar{a}})$, whose degrees are $d_a$ and $n- d_a$. Moreover, after shifting the gradings of the Lagrangians and choosing the ambient grading structure appropriately, the degree $d_a$ can be arranged to satisfy $1\leq d_a \leq n-1$. See, for example, Section 3.1 of \cite{JKL-Ginzburg}. 

    Since the chosen Liouville form restricts to zero on every core branch, Stokes' theorem gives zero symplectic area for any holomorphic polygon with boundary on $\bL$. Thus all such polygons are constant, and this part of the calculation takes place in an arbitrarily small neighborhood of the core. Consequently, $m_0^{\tilde{b}}$ counts pearl trajectories, which consist of constant polygons mapping to the immersed sectors together with Morse flow lines. Therefore, $$m_0^{\tilde{b}}= \sum_k \sum_a m_{2k}(x_a X_a, x_{\bar{a}}X_{\bar{a}}, \ldots, x_a X_a, x_{\bar{a}}X_{\bar{a}}).$$ Moreover, by degree considerations, all terms with $k\geq 2$ vanish. As a result, $$m_0^{\tilde{b}}=\sum_v \left( \sum_{t(a)=v} \epsilon(x_a) x_{\bar{a}}x_a \right) T_v,$$ where $\epsilon(x_a)= \pm 1$ depending on the grading structure. Thus, $\tilde{\cA}_\bL \cong \widehat{G}_n(D).$
\end{proof}

\begin{remark}\label{rem: completed}
	In general, the extended localized mirror or Koszul dual obtained from the bar construction is naturally a completed path algebra, since the bar construction allows tensors of arbitrary length. However, in some cases the completion is unnecessary and one obtains the ordinary, non-completed Ginzburg dg algebra.
	
	For example, suppose that the underlying graph is a tree and $n \geq 3$. In this case there are no degree-zero oriented cycles in the corresponding graded quiver. Consequently, in each fixed degree, only finitely many paths can contribute. Hence the bar construction does not produce genuinely infinite degree-preserving sums, and the resulting dg algebra lies in the ordinary graded path algebra rather than its completion. Thus, in this case, the Koszul dual can be identified with the non-completed Ginzburg dg algebra $G_n(D)$.
\end{remark}

When $X$ is the plumbing of $T^*\mathbb{S}^2$ according to an affine $ADE$ Dynkin graph, the above proposition provides a version of local mirror symmetry. 
\begin{cor}\label{cor:localHMS}
	Let $X$ be the plumbing of $T^*\mathbb{S}^2$ according to an extended ADE Dynkin diagram $D$. There exists a fully-faithful functor $$\Phi: \mathcal{D} \Fuk_\mathbb{L} (X) \to \mathcal{D}^b(\mathrm{Coh}(Y)),$$ which sends the $i$-th component of $\bL$ to $\iota_* \cO_{C_i}(-1)[1]$ for $i \neq 0$. Here \(Y\) is the
	crepant resolution of the corresponding ADE singularity
	\(\mathbb C^2/\Gamma\), and the vertices of \(D\) are labeled by
	\(i=0,1,\dots,r\), with \(i=0\) denoting the affine vertex corresponding to the trivial representation of \(\Gamma\) under the McKay correspondence. For \(i\neq 0\), the vertex \(i\) corresponds to the
	exceptional curve \(C_i\subset Y\), and \(\iota:C_i\hookrightarrow Y\)
	denotes the inclusion.
\end{cor}
\begin{proof}
	By the preceding proposition, the local mirror construction produces a
	completed Ginzburg dg algebra \(\tilde{\cA}_{\mathbb L}\) and an extended mirror functor: \[
	\mathcal D\mathcal F^{(\mathbb L,b)}:
	\mathcal D\Fuk_{\mathbb L}(X)
	\longrightarrow
	\mathcal D(\tilde{\cA}_{\mathbb L}).
	\] In the affine $ADE$ surface case, the relations appearing in this completed algebra are the polynomial preprojective relations. Moreover, for the objects considered here, no genuinely
	infinite series appear in the image of the localized mirror functor. We henceforth work with the non-completed model $G_2(D)$. By abuse of notation, we still denote the corresponding algebra by
	\(\tilde{\mathcal A}_{\mathbb L}\), and the localized mirror
	functor takes values in \(\mathcal D(\tilde{\mathcal A}_{\mathbb L})\).
	
	By \cite{Her16}, over a field of characteristic zero, the Ginzburg dga $\tilde{\cA}_\bL$ is formal, which is quasi-isomorphic to a preprojective algebra, if the graph $D$ is non-Dynkin. Then  Lemma \ref{lem:equiv} implies that the unextended mirror functor gives rise to a quasi-equivalence $$\mathcal{D} \mathcal{F}^{(\mathbb{L},b)}: \mathcal{D} \Fuk_\mathbb{L} (X) \to \operatorname{thick}_{\mathcal{D}(\cA_{\mathbb{L}})}(\Bbbk),$$ which sends the $i$-th component of $\bL$ to the $i$-th simple representation $S_i$. 
	
	On the other hand, when $D$ is an affine $ADE$ Dynkin diagram, the derived McKay correspondence shows that the preprojective algebra $\cA_\bL$ is derived equivalent to $\mathcal{D}^b(\mathrm{Coh}(Y))$. More precisely, $$ \varphi: \mathcal{D}^b(\cA_\bL) \to \mathcal{D}^b(\mathrm{Coh}(Y)), \quad M \mapsto M\otimes^\mathbf{L}_{\cA_{\bL}} \mathcal{T} $$ induces an equivalence, where $\mathcal{T}$ is the tilting bundle of $\widetilde{\C^2/\Gamma}$. In particular, one can check that $E_i:= \varphi(S_i)= S_i \otimes^\mathbf{L}_{\cA_{\bL}} \mathcal{T}$ is quasi-isomorphic to $\iota_* \cO_{C_i}(-1)[1]$ for $i \neq 0$ and $\iota_*\cO_{D}$ for $i=0$ (see e.g. Section 4.3 of \cite{Nak99}). Here $C_i$ is the $i$-th exceptional divisor in $Y$, $D$ is the union of the exceptional divisors and $\iota$ is the inclusion.
	
	As a consequence, $\varphi \circ \mathcal{D} \mathcal{F}^{(\mathbb{L},b)}$ gives the desired fully-faithful functor, and the core of $X$ is mirror to the exceptional divisor in $Y$.
\end{proof}


\subsection{Generation of compact Fukaya categories}

In this subsection, we apply Theorem \ref{thm:extfuctorequiv} to the compact exact Fukaya categories of plumbing spaces and study the relation with some known results on split generation. Moreover, we provide examples where the core Lagrangian fails to split generate the compact exact Fukaya category $\cF(X)$ whose objects are pairs $(L,b)$, where $L$ is a compact exact Lagrangian immersion and $b$ is a bounding cochain supported on the positive-action double point of $L$. Readers are referred to, for example, \cite{JKL-im, AFOOO26} for more details. 

\begin{lemma}\label{lem:gen}
	Let $X$ be the plumbing of $T^*\mathbb{S}^n$ according to a graph $D$ and let $\bL$ be the core of $X$. Then
	\begin{enumerate}
		\item If $n=2$ and $D$ is of Dynkin type $A_k$ or $D_k$, the core Lagrangian $\bL$ split-generates the compact Fukaya category $\cF(X)$.
		\item If $n \geq 3$ and $D$ is a tree, the core Lagrangian $\bL$ split-generates the compact exact Fukaya category $\cF(X)$.
	\end{enumerate}
\end{lemma}
\begin{proof}
	Recall that the wrapped Fukaya category of a plumbing space \(X\) was studied
	in \cite{EL17,EL19} for \(n=2\), and in \cite{KL25} for general \(n\).
	In the present setting, these results imply that the wrapped Fukaya category \(\mathcal W(X)\) is pretriangulated equivalent to dg category of modules over the Ginzburg algebra \(\mathrm{Mod}_{\dg} (G_n(D))\).
	
	
	Moreover, under the above equivalence, \cite{JKL-Ginzburg} shows that the compact exact Fukaya category $\cF(X)$ is identified with
	the full subcategory of \(\mathrm{Mod}_{\dg} (G_n(D))\)
	consisting of cohomologically finite-dimensional modules. Namely, there is a pretriangulated equivalence
	\[ \mathcal F(X) \simeq \mathcal{D}_{\mathrm{fd}}(G_n(D)),\]
	where \(\mathcal{D}_{\mathrm{fd}}(G_n(D))\) denotes the subcategory of dg \(G_n(D)\)-modules with finite-dimensional total cohomology. In other words, an object in $\mathcal{D}_{\mathrm{fd}}(G_n(D))$ is a complex $M$ of $G_n(D)$-modules such that $\mathrm{dim}_\C H^*(M) < \infty.$ 
	
	Notice that when $n \geq 3$, the extended localized mirror $\tilde{\cA}_\bL$ is isomorphic to $G_n(D)$ as mentioned in Remark \ref{rem: completed}. Besides, when $n=2$, there is a quasi-equivalence between \(\tilde{\mathcal A}_{\mathbb L}\) and the Ginzburg dga $G_2(D)$, see \cite[Rmk 6.1.5]{BGO25}. Hence, in both cases, we have $\cF(X) \simeq \mathcal{D}_{\mathrm{fd}}(\tilde{\cA}_\bL).$
	
	Combining with Theorem \ref{thm:extfuctorequiv}, the core Lagrangian $\bL$ split-generates $\cF(X)$ if and only if $\mathcal{D}_{\mathrm{nil}}(\tilde{\cA}_\bL)\cong \mathcal{D}_{\mathrm{fd}}(\tilde{\cA}_\bL).$ As $H^*(M)$ is finite-dimensional as a graded vector space, any class of nonzero-degree arrows acts nilpotently on $H^*(M)$. Therefore, the equivalence $\mathcal{D}_{\mathrm{nil}}(\tilde{\cA}_\bL)\cong \mathcal{D}_{\mathrm{fd}}(\tilde{\cA}_\bL)$ holds when the ideal of $H^0(\tilde{\mathcal A}_{\mathbb L})$ generated by the nontrivial arrow classes is nilpotent.
	
	When $n=2$, since \(\tilde{\mathcal A}_{\mathbb L}\) is quasi-isomorphic to $G_2(D)$, $H^0(\tilde{\mathcal A}_{\mathbb L})$ is identified with the preprojective algebra. If $D$ is of $ADE$ Dynkin type, this algebra is finite-dimensional. Hence its arrow ideal, or equivalently the kernel of the
	augmentation
	\[
	H^0(\tilde{\mathcal A}_{\mathbb L}) \longrightarrow \Bbbk,
	\]
	is nilpotent. Therefore every finite-dimensional module is nilpotent.
	
	For $n\geq 3$, each transverse intersection point gives rise to a pair of reverse arrows $(x_a,x_{\bar a})$ of degrees $1-p_a$ and $1-n+p_a$. These two arrows cannot both have degree $0$.
	If $D$ is a tree, the degree-zero arrow ideal has no oriented cycle, i.e. there is no
	nontrivial path made entirely of degree-zero arrows whose head and tail coincide.
	Thus the degree-zero arrow ideal is nilpotent.
	Hence $\mathbb L$ split-generates $\mathcal F(X)$.
\end{proof}

In the proof above, if one works with a non-completed model and has
\[
\cF(X) \simeq \mathcal{D}_{\mathrm{fd}}(\tilde{\cA}_\bL),
\qquad
\mathcal{D}_{\mathrm{nil}}(\tilde{\cA}_\bL)
\subsetneq
\mathcal{D}_{\mathrm{fd}}(\tilde{\cA}_\bL),
\]
then the core Lagrangian fails to split-generate the compact Fukaya category $\cF(X)$. A typical algebraic source of this strict inclusion is the presence of a degree-zero oriented cycle in the plumbing graph: such a cycle may act non-nilpotently on a finite-dimensional module over the ordinary quiver algebra.

However, this mechanism should be interpreted with care. The extended localized mirror naturally gives a completed algebra \(\tilde{\cA}_\bL\), and for a completed quiver algebra, finite-dimensional modules are
automatically nilpotent with respect to the completion ideal, see Remark \ref{rem: fd=nil}. Hence,
$\mathcal{D}_{\mathrm{nil}}(\tilde{\cA}_\bL)
=
\mathcal{D}_{\mathrm{fd}}(\tilde{\cA}_\bL)$ in the completed setting. This equality does not imply that the core Lagrangian split-generates the whole compact Fukaya category; it only says \(\tilde{\cA}_\bL\) describes the formal local category seen by $\bL$. This local category need not coincide with \(\cF(X)\).

Therefore, in the completed setting, the possible failure of split-generation should be understood as a local-to-global issue: the core Lagrangian may only see a local completed chart of the global mirror. The following example illustrates this point.

\begin{example}
    Let $X$ be the plumbing of $T^*\mathbb{S}^2$ according to affine $A_0$ graph. Then $\bL$ does not split generate the compact Fukaya category $\cF(X)$ considered in \cite{JKL-Ginzburg}. In this case, the wrapped Fukaya category $\mathcal{W}(X)$ is equivalent to $\mathrm{Mod}_{\dg}(\C[x,y,(xy-1)^{-1}])$, whereas, up to a coordinate change, the extended localized mirror $\tilde{\cA}_\bL\cong \C[[x,y]].$ 
    
    By definition, see for example \cite{JKL-Ginzburg}, the Clifford torus $(\mathbb{T}, \nabla^{t_1,t_2})$, equipped with the rank-one local system whose holonomies are nonzero complex numbers $(t_1,t_2)$, defines an object in $\cF(X)$. Here \(t_1,t_2\) are the holonomy parameters chosen as
    in \cite[Section 3]{HKL23}. In particular, $(\mathbb{T}, \nabla^{t_1,t_2})$ is a nonzero object. On the other hand, one can check that $\cF^{(\bL,b)}(\mathbb{T}, \nabla^{t_1,t_2})$ is quasi-isomorphic to $\C[[x,y]]/(x-t_1,xy-1-t_2) \cong 0$, since $x-t_1$ is invertible in $\C[[x,y]]$. Therefore, the core Lagrangian does not split generate $\cF(X).$ 
    This is also consistent with the construction in \cite[Section 3]{HKL23}, where both the immersed sphere with bounding cochains $(\bL,b)$ and the Clifford torus with non-trivial holonomy are required to construct the mirror.  
\end{example}

\begin{remark}
	For plumbing spaces, generation results for the compact Fukaya category are known in several cases. 
	When \(n=2\), the result follows from Seidel's work on ADE configurations; see \cite[Lem.~4.15]{Sei00}, \cite[Cor.~5.8]{Sei08}, and the discussion in \cite[Rem.~29]{EL17}. 
	For \(n\geq 3\), the case of tree-like plumbings was studied by Abouzaid--Smith \cite{AS12}. 
	More recently, Jeong, Karabas and Lee, building on \cite{KL25}, have been studying compact Fukaya categories of more general plumbing graphs and $n \geq 3$ \cite{JKL-im, JKL-Ginzburg}. 
	Their work suggests that, in general, the core Lagrangians alone do not generate; instead, one should include core Lagrangians equipped with suitable bounding cochains.
	
	For the plumbings under consideration, a cocore description provides a comparison with $\mathrm{thick}_{\mathrm{CW}(G,G)}(\Bbbk)$. Theorem~\ref{thm:extfuctorequiv} instead constructs the local equivalence directly from $\CF(\bL,\bL)$ and its deformation algebra $\tilde{\cA}_\bL$, without requiring cocores or other Lagrangians dual to $\bL$. The cocore description is therefore a useful comparison, rather than an input to the compact-brane construction.
\end{remark}

The same nodal neighborhood can occur in a different ambient symplectic manifold, including a compact, hence non-exact, one. With compatible local brane and perturbation data, the constant-disc contributions are unchanged; identifying the full deformation equations additionally requires control of nonconstant polygons. Suitable anticanonical compactifications provide such a setting. The following proposition and corollary compare the open and compactified models, using the positivity argument of \cite[Section 2.2 and Appendix A.2]{AAK16}.

\begin{prop}[\cite{AAK16}]\label{prop:zero-intersection}
	Let $\overline X$ be a compact K\"ahler manifold, let $D\subset\overline X$ be a reduced simple normal crossings divisor, and put $X=\overline X\setminus D$. Assume that $D\cdot C\geq 0$
	for every rational curve $C\subset\overline X$. Let $\mathbb L\subset X$ be a compact immersed Lagrangian with transverse double points. Then using common geometric and virtual  perturbation data, the zero-intersection Floer operations in $\overline X$ agree with those in $X$:
	$m_k^{\overline X,[0]}=m_k^X.$
\end{prop}

\begin{cor}\label{cor:cpt}
	In the setting of Proposition~\ref{prop:zero-intersection}, suppose further
	that $D$ is anticanonical and assume that the Floer construction admits a common strictly unital, nonnegatively graded model $C$ with $C^0$ generated by units.
	Then the compactification preserves the weak Maurer-Cartan relations and equips the formal
	deformation space $\cA_{\mathbb L}$ with a superpotential $W$.
\end{cor}
\begin{proof}
	The assumption of Proposition \ref{prop:zero-intersection} implies that there are no discs with boundary on $\bL$ of negative Maslov index, and that those of vanishing Maslov index are disjoint from the anticanonical divisor $D$.
	 
	Decompose Floer operations according to total intersection number with $D$, i.e. $ m_k^{\overline X}=\sum_{d\geq0}m_k^{\overline X,[d]}.$
	By the preceding proposition, $m_k^{\overline X,[0]}=m_k^X$. As $D$ is an anticanonical divisor, 
	$m_k^{\overline X,[d]}(b,\ldots,b)$ has degree $2-2d$.
	Thus the intersection-zero terms recover $m_0^{X,b}$, the terms
	with $d=1$ are unit-valued, and those with $d\geq2$ vanish. In other words, $$m_0^{\overline X,b}= m_0^{X,b} + m_0^{\overline X,[1],b} .$$
	The coefficients of all non-unit generators therefore coincide with those of $m_0^{X,b}$, proving that the non-unit weak Maurer-Cartan relations are unchanged.
\end{proof}

\begin{example}[Projective compactification of the affine $A_1$ smoothing]
	Consider the smoothing of the affine $A_1$ singularity
	$$X=\{(x,y,z)\in\mathbb C^3:xy=(z-c_0)(z-c_1),\ z\neq 0\},$$
	where $c_1<c_0<0$. Its projective closure is
	$$\bar X=\{[X_0:Y_0:Z:W]\in\mathbb P^3:
	X_0Y_0=(Z-c_0W)(Z-c_1W)\}. $$
	Since $c_0\neq c_1$, $\bar X$ is a smooth quadric surface, and hence $\bar X\cong\mathbb P^1\times\mathbb P^1$. 
	
	Let $D_{\mathrm{fill}}=\bar X\cap\{Z=0\},
	D_\infty=\bar X\cap\{W=0\}.$
	Then $X=\bar X\setminus D$, where $D=D_{\mathrm{fill}}+D_\infty$ is a simple normal crossings divisor. Both components have class $(1,1)$, so $[D]=(2,2)=-K_{\bar X}.$
	In particular, $D$ is an ample anticanonical divisor. The holomorphic volume form
	$$\Omega=\frac{dx\wedge dz}{xz}$$
	extends meromorphically to $\bar X$ with simple poles along $D$.
	
	Equip $\bar X$ with the restriction of the Fubini-Study form on $\mathbb P^3$. This is a K\"ahler form, and the two ruling
	classes have the same symplectic area. Write $q$ for their common Novikov weight. Let $L=\bS_0^2\bigcup \bS_1^2\subset X$ be the Lagrangian immersion associated with the affine $A_1$ plumbing. 
	Since $D$ is ample and anticanonical,
	Corollary~\ref{cor:cpt} implies that the non-unit weak Maurer-Cartan relations are the same. Thus the completed deformation algebra $\cA_L$ remains unchanged, while the compactification
	introduces a superpotential
	$W=W_{\mathrm{fill}}+W_\infty,$ where the two terms correspond to contributions intersecting
	$D_{\mathrm{fill}}$ and $D_\infty$, respectively.
	
	The formal deformation space $\cA_L$ of $L$ and its gluing with SYZ torus charts were studied in \cite{HLT24}. In this example, $\cA_L$ is the completed preprojective algebra of affine $A_1$, whose arrows are denoted by $u_i,v_i$. In particular, these gluing maps allow us to compute the superpotential $W$. For a Clifford-type torus enclosing one critical value, Auroux \cite[Eq. (5.10)]{Aur07} computed
	$$
	W_T=z_1+z_2+\frac{q}{z_1}+\frac{q}{z_2},
	$$
	where $z_i$ is the area-weighted holonomy associated to the basic disc class $\beta_i$.
	
	Write $\gamma_i=\partial\beta_i$. We choose gauge hypertori dual to the basis $\gamma_x=\gamma_2$ and $\gamma_y=\gamma_1-\gamma_2$. Suppressing the basic-disc area factors in the coordinate notation, we use holonomy coordinates
	$x=-\operatorname{Hol}(\gamma_x),
	y=-\operatorname{Hol}(\gamma_y),$
	so that $z_1=xy$ and $z_2=-x$. In our conventions, the gluing relations are
	$$
	u_2x=u_1,\qquad v_2=xv_1, \quad y=1+v_1u_1.
	$$
	Consequently, on the commutative torus chart,
	$$W=xy-x+\frac{q}{xy}-\frac{q}{x}=x(y-1)-q\frac{y-1}{xy}=v_2u_1-qv_1u_2(1+v_1u_1)^{-1}.
	$$
	This resulting expression extends to $\cA_L$, since $1+v_1u_1$ is invertible the arrow-adic completion.
\end{example}

\section{Bulk Deformations of the Plumbing of $T^\ast \mathbb{S}^2$
 and Representations of Deformed Preprojective Algebras}\label{bulk_deform}

Let $X$ be the plumbing of $T^*\mathbb{S}^2$ according to a graph $D$, equipped with the standard Liouville form, and $\bL$ be the core of $X$. In \cite{HLT24}, the authors showed that the noncommutative deformation space of $\bL$ in $X$ is isomorphic, up to a change of coordinates, to the preprojective algebra. Moreover, \cite{HLT24} introduced the notion of framed Lagrangian branes $(L^\fr,\mathcal{E})$, and proved that the Maurer-Cartan space of a framed Lagrangian brane is isomorphic to the Nakajima quiver variety at complex moment-map level $\mu_{\C}=0$.

From the viewpoint of symplectic reduction, it is natural to ask how to obtain the variety in the more general moment map level of a coadjoint orbit
$$ (\mu_{\mathbb{C},i})_{i=1}^n \in \prod_{i=1}^n \mathcal{O}_{g_i}. $$ 

In this section, we will apply bulk deformations of $X$ to construct the deformed preprojective algebras and Nakajima quiver variety at non-zero levels. 

Moreover, this construction gives a symplectic counterpart of the non-commutative deformations considered by Kawamata \cite{Kaw24A}. In \cite{Kaw24A}, Kawamata constructs non-commutative schemes deforming the $A_n$-resolution, compares their non-commutative local charts with the non-commutative crepant resolution, and realizes the corresponding derived equivalence by a tilting bundle.  From the mirror symmetry point of view, the non-commutative deformation parameters in Kawamata's construction correspond to bulk deformation parameters in the Fukaya category.  Moreover, the comparison maps between the non-commutative local charts and the non-commutative crepant resolution arise, on the symplectic side, from Fukaya isomorphisms between the corresponding Lagrangian immersions.

\subsection{Bulk-Deformed Maurer-Cartan algebra}

For each sphere component $L_i$ of $\bL$ in a plumbing $X$, we pick a point $p_i$ in $L_i$ that is away from the nodal points and critical points of the Morse function on $L_i$. Let $C_i$ be the cotangent fiber $T^*_{p_i}L_i$ which gives rise to a codimension-two submanifold in $X$. We consider bulk deformations of $X$ by the cycle $\mathbbm{b}:=\sum_i \delta_i C_i$ for each fixed value of $\delta_i$.


The compact core is used here as a reference immersion for a curved deformation problem, not necessarily as an object at zero bounding cochain. The calculation below gives
\[
   m_0^{\bb}(\bL)=\sum_i\delta_i P_i,
\]
where $P_i$ is the degree-two Morse generator on $L_i$. In particular, an embedded core sphere with $\delta_i\neq0$ is obstructed; its individual minimal self-Floer model has no generators with which to cancel this term. The whole immersion, however, has degree-one immersed generators between its branches. They supply the variables in $b=\sum_a x_aX_a$, and the equations $m_0^{b,\bb}=0$ express cancellation of the bulk curvature by boundary deformations. A solution gives an unobstructed brane only over coefficients for which the Floer operations are defined. These equations describe possible unobstructed deformations; the categorical equivalence of Theorem~\ref{thm:extfuctorequiv} remains subject to its own unobstructedness hypothesis.


We use the standard exact Liouville form on the plumbing $X$, whose restriction to each core branch is zero. By Stokes' theorem, a holomorphic polygon with boundary on $\bL$ has zero symplectic area and is constant. Hence the self-Floer operations entering the Maurer-Cartan calculation are computed by constant polygons and Morse trajectories in a small neighborhood of $\bL$. With compatible brane, bulk, and perturbation data, this is also the zero-symplectic-area part of the deformation problem in an ambient manifold containing the same neighborhood, including a non-exact one. Here bulk parameters are retained as formal weights: zero area refers to the disc components, not to the valuation of the substituted parameters. Nonconstant ambient polygons may contribute additional terms, so the neighborhood identification alone does not identify the full ambient Floer theory. Extending the local bulk constraints to admissible ambient cycles is a separate requirement.


For the full filtered bulk-deformed theory, we take $\delta_i\in\Lambda_+$ and impose the usual admissibility conditions on boundary deformations. More generally, the relevant operations are interpreted over coefficient rings completed so that their sums are well-defined.\footnote{For other objects, nonconstant polygons may contribute arbitrarily high powers of the bulk parameters. Polynomial dependence in a particular core calculation does not by itself give a complex-valued specialization of the functor on the entire filtered Fukaya category.} The core calculation below has polynomial dependence on the bulk parameters and, after the formal coordinate normalization~\eqref{eq:coord}, yields polynomial preprojective relations. We distinguish this polynomial presentation from its formal completion. In particular, the appearance of polynomial relations does not by itself establish convergence or algebraicity of the normalizing coordinate change, or of all operations defining the mirror functor.


\begin{thm}\label{thm: dpre}
	The bulk-deformed Maurer-Cartan algebra $\cA_\bL$ of $\bL$ by the cycle $\mathbbm{b}:=\sum_i \delta_i C_i$ is isomorphic to the deformed preprojective algebra associated to the corresponding quiver of $\bL$.
	
	Similarly, for a fixed stability parameter $\zeta$ and a framed Lagrangian brane $(L^{\fr},\mathcal{E})$ equipped with trivial bundles $\mathcal{E}$ of rank $(\vec{n},\vec{d})$, the bulk-deformed Maurer-Cartan space is isomorphic to the corresponding Nakajima quiver variety at the complex moment-map level $\mu_{\C,i} = \delta_i$.
\end{thm}
\begin{proof}
	Using the Morse model, the Floer complex $\CF^*(\bL)$ is generated by the maximal points $M_v$ and minimum points $P_v$ of the Morse function $f_v$ on $\mathbb{S}^2_v$, and the degree-one immersed sectors among the adjacent components. Thus, by the construction introduced in Section 2, the quiver $Q$ assigned to $\bL$ is the double quiver of $D$. We denote the set of arrows by $\mathfrak{A}$, and the immersed sectors by $X_a \in \CF^1(\mathbb{S}^2_{t(a)},\mathbb{S}^2_{h(a)})$ for $a \in \mathfrak{A}$.

	Let $b=\sum_{a\in \mathfrak{A}} x_a X_a$.
	We compute the obstruction term $m_0^{b,\mathbbm{b}}$ of $\bL$. The orientation and spin structure of $\bL$ correspond to an orientation $\Omega$ of the double quiver $Q$, which defines the function $\epsilon:\mathfrak{A}  \to \{\pm 1\}$. 
	
	The outputs of the obstruction term $m_0^{b,\mathbbm{b}}$ have degree 2. In this case, the degree-two generators are the minimum points $P_v$. Thus, $m_0^{b}$ counts the pearl trajectories to $P_v$. By the preceding discussion, all holomorphic discs with boundary on $\bL$ are constant. Thus, the relevant trajectories are unions of Morse flow lines to $P_v$ and constant polygons $(x_{\bar{a}}x_a)^{k}$ for all $x_a$ with $t(x_a)=v$. Moreover, there are flow lines from the base point of $C_v$ to $P_v$.
	As a consequence, we obtain  
	\begin{align*}
		m_0^{b,\mathbbm{b}} & =m_0^{\mathbbm{b}}+\sum_k m_{2k}\left(\sum_{a} x_a X_a, \sum_{\bar{a}} x_{\bar{a}} X_{\bar{a}}, \cdots, \sum_{a} x_a X_a, \sum_{\bar{a}} x_{\bar{a}} X_{\bar{a}}\right) \\
		& = \sum_v \left(\delta_v+\sum_{t(a)=v} \epsilon(a)x_{\bar{a}}x_{a}\left(1+ \sum_j a_j(x_{\bar{a}}x_{a})^j \right)\right)P_v,
	\end{align*}
	where $a_j$ depends on the Kuranishi perturbation, and $(x_{\bar{a}}x_{a})^j$ is contributed by the constant polygon with corners being the immersed sectors $x_a$ and $x_{\bar{a}}$, ordered counterclockwise.
	
	Because of the local anti-symmetric involution in a neighborhood of the transverse intersection point, see \cite[Lem. 8.4.3]{FOOO09} or \cite[Lem. 3.4]{HKL23}, a pair of arrows contributes to the coefficients $x_{\bar{a}}x_{a}(1+ \sum_j a_j(x_{\bar{a}}x_{a})^j)$ at $P_{t(a)}$ and $x_{a}x_{\bar{a}}(1+ \sum_j a_j(x_{a}x_{\bar{a}})^j)= x_a (1+ \sum_j a_j(x_{\bar{a}}x_{a})^j)x_{\bar{a}}$ at $P_{h(a)}$. We define the following change of coordinates: 
	\begin{equation}
		\label{eq:coord}
		\tilde{x}_a= x_{a}\left(1+ \sum_j a_j(x_{\bar{a}}x_{a})^j\right), \quad \tilde{x}_{\bar{a}}= x_{\bar{a}},
	\end{equation} for all $a$ such that $\epsilon(a)=1$. In particular, one can check that the factor is invertible.
	
	We will abuse the notation and replace $\tilde{x}$ by $x$. Thus, $m_0^{b,\mathbbm{b}}= \sum_v \left(\delta_v+\sum_{t(a)=v} \epsilon(a)x_{\bar{a}}x_{a}\right)P_v$, and the noncommutative deformation space of $\bL$ is isomorphic to the deformed preprojective algebra.
	

Similarly, one can compute the bulk-deformed Maurer-Cartan space of a framed Lagrangian brane $(L^\fr,\mathcal{E})$ by observing that each framing component is equipped with a Morse function with a unique maximum point. This produces the framed analogue of the above computation. Evaluating the resulting equations on the matrix coefficients of $\mathcal E$ gives, with the moment-map convention in the statement,
\[
   \mu_{\C,i}(B,I,J)=\delta_i\,\Id_{V_i},
\]
where $V_i$ is the fiber on the $i$-th core component and $I,J$ denote the framing maps. Thus the bulk contribution shifts the complex moment-map level, while the ranks of the bundles give the dimension vectors. Taking the quotient with the chosen stability parameter $\zeta$ gives the stated framed Maurer--Cartan space.

\end{proof}


\noindent\emph{Polynomial presentations and completion.}
The polynomial presentations in normalized coordinates used below must be distinguished from completions at the zero-arrow representation. For example, for $\lambda\in\C^\times$,
\[
   \C\langle\!\langle x,y\rangle\!\rangle/(xy-yx+\lambda)=0,
\]
because $xy-yx+\lambda$ has an invertible constant term, whereas the polynomial quotient in Example~\ref{ex:ncC2} is a nonzero algebra. Thus a nonzero complex fiber of the polynomial family is not obtained by completing that fiber at the former zero-deformation point. This is the algebraic counterpart of the loss of the original unobstructed brane. Complex bulk parameters below refer to the polynomial presentations. Identifying their complex-valued points with geometric branes, or evaluating mirror-functor formulas there, requires finiteness or convergence for the relevant operations and coordinate changes. Likewise, specialization of Novikov weights is legitimate only for finite or convergent expressions, not for arbitrary Novikov series.

Moreover, we can also consider the matrix coefficients in the Maurer-Cartan algebra as noncommutative free variables. This gives a noncommutative version of the Nakajima quiver variety. 
We will soon see this in the examples of affine $A_n$ quivers. 

We now apply the above theorem to the smoothing of $A_n$-singularity, which is a conic fibration over $\C^*$. Moreover, we compute the local noncommutative charts of the formal deformation space of the core Lagrangian under bulk deformation and their gluing. This agrees with the recent result of Kawamata \cite{Kaw24A}.

Recall that the smoothing of the $A_n$ singularity is the following surface
$$X=\{(x,y,z)\in \mathbb{C}^3\mid xy = p(z) \text{ and }z\neq 0\}$$ where $p(z)=\prod_{k=0}^{n}(z-c_{k})$ is a monic polynomial and $-\infty < c_{n}<\cdots < c_{0}<0$ are distinct real numbers. This is a smoothing of the $A_n$ singularity. $X$ is equipped with the K\"ahler form inherited from the standard form on $\mathbb{C}^3.$ As is known, the projection map $\pi: (x,y,z)\mapsto z$ can be viewed as a Lefschetz fibration. It admits a natural Hamiltonian $\cS^1$-action given by \begin{align*}
	e^{it}\cdot(x,y,z) =  (e^{it} x, e^{-it} y, z)
\end{align*}
whose associated moment map is $\mu(x,y,z)=\frac{1}{2}(|x|^2-|y|^2)$. 

$X$ admits a special Lagrangian torus fibration 
$$F:X \to \R^2, \, F(x,y,z)=(\mu(x,y,z), \ln \, |z|^2).$$ 
We denote the singular fibers $F^{-1}(0, \ln |c_i|^2)$ by $\cS_i$ for $i=0,\cdots, n$, which are immersed Lagrangian spheres.

Let $C_{q}=\left\{(x,y,z)\in X\mid \mu(x,y,z) = 0 \text{ and }z= q\right\}$.  Consider $\bL := \bigcup_{q\in c}C_{q}$, where $c$ is a loop passing through all the zeros of $p(z)$. Thus, $\bL$ is a union of $n+1$ Lagrangian spheres $\bigcup_{i=0}^{n} \bS^2_i$, see Figure \ref{fig:An}. 

\begin{figure}[htbp]
	\centering
	\includegraphics[width=0.7\textwidth]{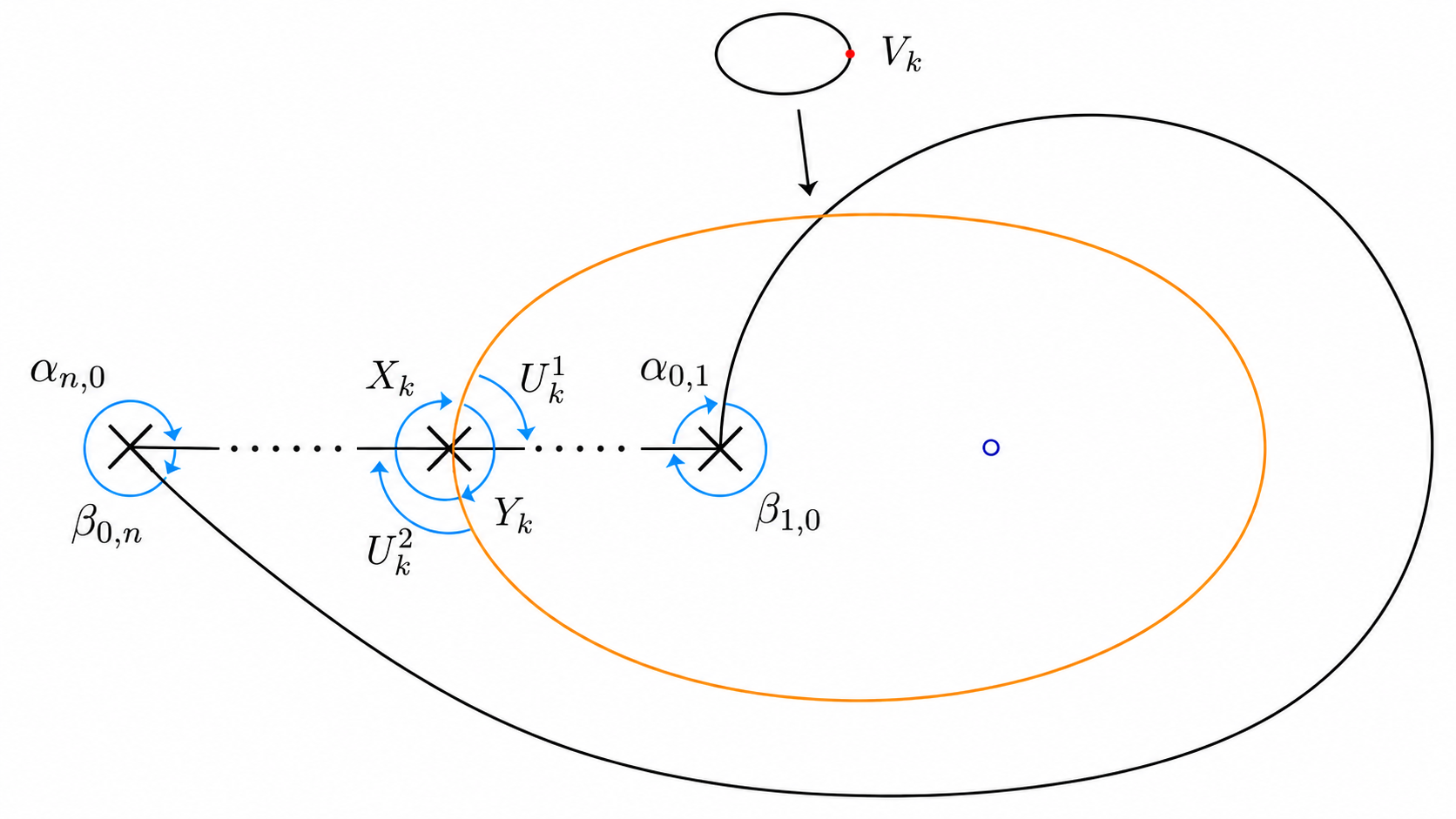}
	\caption{The image of $\bL$ under the conic fibration of $X$. The degree-one immersed sector from $\bS^2_{k+1}$ to $\bS^2_k$ is denoted by $\alpha_{k,k+1},$ while the sector in the opposite direction is denoted by $\beta_{k+1,k}.$ In addition, the degree-one self-immersed sectors of $\cS_k$ are denoted by $X_k$ and $Y_k$.}
	\label{fig:An}
\end{figure}

By the Lagrangian neighborhood theorem, applied compatibly at the nodes, a sufficiently small neighborhood of the core in the corresponding affine $A_n$ plumbing is symplectomorphic to a neighborhood of $\bL$ in $X$. This is a local identification, not an embedding of the completed plumbing. The bulk-deformation calculation gives the following generalization of the affine $A_n$ construction of \cite{HLT24}.

\begin{cor}\label{cor: An}
	The bulk deformed Maurer-Cartan algebra of $\bL$ is isomorphic to $$\cA_{\bL}:= k[\delta_0,\delta_1,\ldots,\delta_n] Q/(\alpha_{i,i+1} \beta_{i+1,i}- \beta_{i,i-1}\alpha_{i-1,i}- \delta_i)_i, $$ while the bulk deformed Maurer-Cartan algebra of $\cS_i$, for $i=0, \ldots,n$, is isomorphic to $$\cA_i:= k[\delta_0,\delta_1,\ldots,\delta_n]\langle x_i,y_i\rangle/(x_iy_i-y_ix_i-(\delta_0+\cdots + \delta_n)).$$ Here, we write $\mathbbm{b}:=\sum_i \delta_i C_i$, and $\delta_0, \ldots,\delta_n$ are the bulk parameters.
\end{cor}

\begin{remark}
	These calculations describe the constant-disc contribution for the nodal neighborhood. In another ambient symplectic manifold, the same contribution is obtained when the brane data and the restrictions of the bulk insertions agree. Identifying the full bulk-deformed Maurer-Cartan relations additionally requires admissible extensions of the local bulk cycles and control of nonconstant polygons. Under the corresponding compactification hypotheses, an anticanonical divisor may introduce a superpotential without changing the non-unit relations.
\end{remark}

To construct the noncommutative mirror of $X$, one can consider the gluing of deformation spaces of $\cS_i$ through that of $\bL$. This results in a quiver algebroid stack, which corresponds to the noncommutative variety considered by Kawamata \cite{Kaw24A}.
\begin{prop} \label{prop:A_n}
	There exist preisomorphism pairs between $(\bL,b)$ and $(\cS_i,b_i),i=0,\cdots, n$:
	$$\alpha_i \in \CF_{\cA_i \otimes \cA_\bL(U_{i})}((\cS_i,b_i),(\bL,b)),\quad \beta_i \in \CF_{\cA_\bL(U_{i}) \otimes \cA_i}((\bL,b),(\cS_i,b_i))$$
	and a quiver stack $\hat{\cY}$ corresponding to the noncommutative deformations of the crepant resolution of $A_n$-singularity, with charts $\cA_\bL$ and $\cA_i,i=0,\cdots,n$, that solves the isomorphism equations for $(\alpha_i,\beta_i)$:
	\begin{align}
		\label{equ:stackrelation1}
		m_{1,\hat{\cY}}^{b_i,b}(\alpha_i) = 0,& \quad  m_{1,\hat{\cY}}^{b,b_i}(\beta_i) = 0;\\
		\label{equ:stackrelation2}
		m_{2,\hat{\cY}}^{b_i,b,b_i}(\alpha_i,\beta_i)  = \mathbf{1}_{\cS_i},&\quad  m_{2,\hat{\cY}}^{b,b_i,b}(\beta_i,\alpha_i) = \mathbf{1}_{\bL}.
	\end{align}
    Moreover, the gluing of the $\cA_i$, for $i = 0, \ldots, n$, yields the noncommutative variety introduced by Kawamata. Here, $\cA_\bL(U_{i})$ is the localization of $\cA_\bL$ at the set of arrows
	$$\{\alpha_{0,1},\cdots, \alpha_{i-1,i}, \beta_{n+1,n}, \cdots, \beta_{i+2,i+1}\}$$ for $i=0,\cdots,n$, respectively.
\end{prop}
\begin{proof}
	In the proof, we use the same notations as shown in Figure \ref{fig:An}. We denote the bulk parameters by $\delta_0, \ldots,\delta_n$, $A_{k}$ and $A_{k}'$  denote areas of the polygons with vertices $U_k^2,\beta_{k+1,k},\cdots ,\beta_{0n},V_k$ and $U_k^1,\alpha_{k,k-1},\cdots, \alpha_{01},V_k$. The isomorphism pairs are defined using normalized immersed sectors $U_k^1,U_k^2$ and $V_k$, 
	$$
	\alpha_{k}= T^{-A_{k}'} (\alpha_{01}\cdots \alpha_{k-1,k})^{-1}U_k^1+ T^{-A_{k}}(\beta_{0,n}\beta_{n,n-1}\cdots \beta_{k+2,k+1})^{-1}U_k^2, \quad \beta_k = V_k.$$ In fact, these isomorphism pairs are detected by computing $m_1^{b,b_k}(U_k^1+U_k^2)$ and $m_1^{b_k,b}(V_k)$. More precisely,
	\begin{align*}
		m_1^{b,b_k}(U_k^1+U_k^2)= (T^{A_k'} \alpha_{01} \cdots \alpha_{k-1,k} -T^{A_k} \beta_{0,n}\beta_{n,n-1}\cdots \beta_{k+2,k+1}) V_k' + (x_k-\alpha_{k,k+1}) \mathcal{X}_k  +(y_k- \beta_{k+1,k}) \mathcal{Y}_k ,
	\end{align*} where $V_k'$ is a degree one intersection point complementary to $V_k$ in the base, $\mathcal{X}_k, \mathcal{Y}_k \in \CF^1(\cS_i, \bL)$. This suggests the above choices of \(\alpha_k\) and \(\beta_k\), together with the identifications $x_k=\alpha_{k,k+1}, y_k=\beta_{k+1,k}.$ These identifications constitute part of the local data of the quiver algebroid stack. We now explain how they enter the transition maps between the local charts. One can check that over the quiver algebroid stack, $\alpha_k$ and $\beta_k$ satisfy Equations \eqref{equ:stackrelation1} and \eqref{equ:stackrelation2}.
	
	For simplicity, we assume the area terms vanish, i.e. $A_i=A_i'=0$, which suffices for the purpose of constructing the quiver algebroid stack.

	Let $t_0:=\delta_0 + \delta_1 + \cdots +\delta_n$ and $t_i:= -(\delta_0 + \cdots + \hat{\delta_i} + \cdots + \delta_n)$ for $i=1, \cdots,n.$
	
	The transition representation maps  $G_{k(n+1)}:\cA_\bL|_{U_k}\rightarrow\cA_{k}$ and $G_{(n+1)k}:\cA_k \to \cA_\bL|_{U_k}$ for $k=0, \ldots, n$ are defined as \begin{align*}
		G_{k(n+1)}:=& \begin{cases} 
			\alpha_{i,i+1}\mapsto 1 & \forall i<k  \\
			\alpha_{k,k+1} \mapsto x_k \\
		\alpha_{i,i+1}\mapsto x_ky_k+ (i-k-1)t_0+t_{k+1}+ \cdots t_i& \forall i \geq k+1\\
		\beta_{i+1,i}\mapsto 1 & \forall i \geq k+1  \\
		\beta_{k+1,k}\mapsto y_k &  \\
		\beta_{i+1,i} \mapsto x_ky_k-(k-i)t_0-t_{i+1}-\cdots-t_k& \forall i < k 
	\end{cases},\\ G_{(n+1)k}:=&\begin{cases}
		x_{k}\mapsto (\alpha_{01}\cdots \alpha_{k-1,k})\alpha_{k,k+1}(\beta_{0,n} \cdots \beta_{k+2,k+1})^{-1} &  \\
		y_k\mapsto (\beta_{0,n} \cdots \beta_{k+2,k+1})\beta_{k+1,k}(\alpha_{01}\cdots \alpha_{k-1,k})^{-1}
	\end{cases}.
\end{align*}
	In particular, we obtain the transition maps between $\cA_k$ and $\cA_{k+1}$:
	$$G_{(n+1)k} \circ G_{(n+1)(k+1)}:=\begin{cases}
		x_{k+1} \mapsto (\alpha_{01}\cdots \alpha_{k,k+1})\alpha_{k+1,k+2}(\beta_{0,n} \cdots \beta_{k+3,k+2})^{-1} \mapsto x_k^2y_k +t_{k+1}x_k\\
		y_{k+1} \mapsto (\beta_{0,n} \cdots \beta_{k+3,k+2})\beta_{k+2,k+1}(\alpha_{01}\cdots \alpha_{k,k+1})^{-1} \mapsto x_{k}^{-1}
	\end{cases}.$$ 

The gerbe terms $c_{0k0}$ at the vertices of $Q$  are defined by  $$c_{0k0}(v_i)= \begin{cases}
	\alpha_{01}\cdots \alpha_{i-1,i} & i \leq k\\
	\beta_{0,n}\beta_{n,n-1}\cdots \beta_{i+1,i} & i>k   \\ 	 
\end{cases} $$ where $v_i$ denotes the $i$-th vertex of $Q$, for $i=0,\ldots,n$. The other gerbe terms are trivial.
\end{proof}

Note that for the smoothing of the $A_n$ singularity, one can compute the transition maps directly by carefully deforming the immersed Lagrangian spheres $\cS_i$ \cite{LLL24}.

\begin{remark}
	Notice that in \cite[Lemma 5.3]{Kaw24A}, Kawamata wrote down a map which compared the noncommutative crepant resolution of $\C^2/\Z_n$ with noncommutative deformations of the usual commutative resolutions. This map coincides with the transition representation $G_{k(n+1)}$, which comes from solving Fukaya isomorphism equations over the quiver algebroid stacks. Later, we will produce more examples of these constructions, where the deformations of the noncommutative crepant resolutions coincide with the noncommutative bulk-deformation space of an immersed Lagrangian. 
\end{remark}

\begin{lemma}\label{lem:tilting}
	Let $F_i$ be the cocores for $i=0,\ldots,n$. Then $\cF^{(\mathcal{L},b)}(\oplus_k F_k)$ is the tilting bundle on the noncommutative scheme $\mathcal{X}$ obtained by gluing the localized mirrors $\cA_i$.
\end{lemma}

\begin{proof}
    We compute $\cF^{(\mathcal{L},b)}(\oplus_k F_k)$ explicitly and show that it coincides with the tilting bundle of the noncommutative deformation of the crepant resolution of the $A_n$ singularity constructed by Kawamata \cite{Kaw24A}. 
    
    As $\cS_k$ intersects $F_j$ transversally at a point $P_{kj}$, we have $\cF^{(\cS_k, b_k)}(F_j) = \cA_k \langle P_{kj} \rangle,$ which is the structure sheaf over noncommutative $\C^2$, i.e. a free rank-one $\cA_k$-module. Observe that $\widetilde{\alpha}_{ij}$ forms Fukaya isomorphism between $\cS_i$ and $\cS_j$ over the quiver algebroid stack $\hat{\mathcal{Y}}$, which induces transition maps between $\cF^{(\cS_i, b_i)}(F_k)$ and $\cF^{(\cS_j, b_j)}(F_k)$. More precisely, we have the transition maps \begin{align*}
    	\rho_{i,i+1}:=m_2^{b_i,b_{i+1},b_i}(\widetilde{\alpha}_{i,i+1},-):& \,\cF^{(\cS_{i+1}, b_{i+1})}(F_k) \to  \cF^{(\cS_i, b_i)}(F_k) \\
    	&\, P_{i+1,k} \mapsto m_2^{b_i,b_{i+1},0}(\widetilde{\alpha}_{i,i+1}, P_{i+1,k}).
    \end{align*}

    By a local trivialization of the conic fibration $p: X \to \C^*$, we observe that, generically, the bulk cycles $\mathbbm{b}:=\sum_i \delta_i C_i$ do not intersect the holomorphic disc bounded by the immersed spheres. Hence, the transition maps $m_2^{b_i,b_{i+1},b_i}(\widetilde{\alpha}_{i,i+1},-)$ are not affected by the bulk deformation. In particular, if $i+1 \neq k$, the holomorphic discs do not pass through the immersed sectors. Hence, $m_2^{b_i,b_{i+1},0}(\widetilde{\alpha}_{i,i+1}, P_{i+1,k})=P_{i,k}.$ Moreover, if $i+1=k$, there exists a holomorphic disc with boundary passing through the immersed sector $x_i X_i$, which implies $$m_2^{b_i,b_{i+1},b_i}(\widetilde{\alpha}_{i,i+1}, P_{i+1,i+1})=m_3(x_iX_i,\widetilde{\alpha}_{i,i+1}, P_{i+1,i+1})=x_i^{-1} P_{i,i+1}.$$

    In summary, $\cF^{(\mathcal{L},b)}(F_k)$ is a line bundle over the noncommutative variety such that $\cF^{(\mathcal{L},b)}(F_k) \mid_{U_i} \cong \cA_i$, $\rho_{i,i+1}= id$ if $i+1\neq k$ and $\rho_{k-1,k}=x_{k-1}^{-1}$ for all $i,k$. Hence, $\cF^{(\mathcal{L},b)}(\oplus_k F_k)$ is isomorphic to the tilting bundle over the noncommutative scheme $\mathcal{X}$.
\end{proof}

\begin{example}
	\label{eg: D4}
	Let $Q$ be the doubled quiver of affine $D_4$ as shown in Figure \ref{fig:affD_4} (a). We consider the deformed preprojective algebra $\cA_{0}$ of $Q$, in other words, the elements in $\cA_{0}$ satisfy $\sum_{i=1}^{4} a_i b^i=-\delta_5$ and $b^ia_i=\delta_i$ for $i=1,\cdots, 4$, where $\delta_k$ is the deformation parameter. In the following, we will write down an affine local chart for $\cA_{0}$ explicitly, whose dimension vector $\vec{w}$ is $(1,1,1,1,2)$ where the order of components agrees with the subscripts of the vertices. 
	
	\begin{claim}
		The affine local chart $\cA_2$ of $\cA_{0}$ is the noncommutative deformation of $\C^2$, i.e. $$\cA_2\cong \C[\delta_1,\ldots,\delta_5]\langle X,T \rangle/(XT-TX+\delta_1+\cdots +2\delta_5).$$
	\end{claim}

\begin{center}
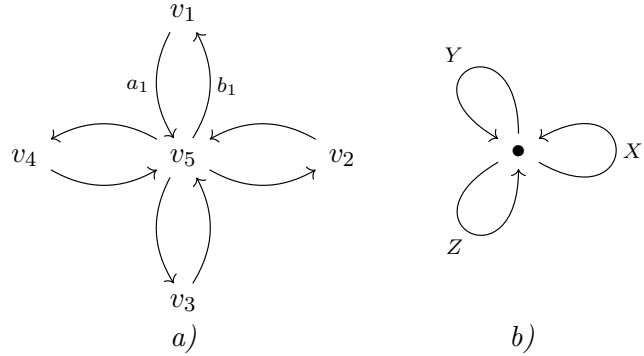

	\begin{tabular}{cc}
		\begin{tikzcd}[row sep=huge,column sep=large]
			& v_{1} \arrow[d, "a_{1}"', bend right] &                    \\
			v_{4} \arrow[r,  bend right] &  v_{5} \arrow[u, "b_{1}"', bend right] \arrow[r, bend right] \arrow[d, bend right] \arrow[l, bend right] & v_{2} \arrow[l, bend right] \\& v_{3} \arrow[u, bend right]&                                      
		\end{tikzcd} &
		\begin{tikzcd}[row sep=huge]
			\bullet \arrow["X"', loop, distance=4em, in=30, out=330] \arrow["Y"', loop, distance=4em, in=150, out=90] \arrow["Z"', loop, distance=4em, in=270, out=210]
		\end{tikzcd} \\
		a) & b)\\
	\end{tabular}
	\captionof{figure}{Double quiver of affine $D_{4}$ and the underlying quiver of $\cA_{2}$}
	\label{fig:affD_4}
\end{center}

    We localize $\cA_{0}$ at the matrix of arrows $\begin{pmatrix}
    	a_1 & a_2 
    \end{pmatrix}$, that is, we insert the reverse arrows $\alpha^1$ and $\alpha^2$ so that \begin{equation} \label{eq:inv1}
    	\begin{pmatrix}
    		a_1 & a_2
    	\end{pmatrix} \cdot \begin{pmatrix}
    		\alpha^1 \\
    		\alpha^2
    	\end{pmatrix} =e_5,
    \end{equation}
    \begin{equation} \label{eq:inv2}
    	\begin{pmatrix}
    		\alpha^1 \\
    		\alpha^2
    	\end{pmatrix} \cdot \begin{pmatrix}
    		a_1 & a_2
    	\end{pmatrix}= \begin{pmatrix}
    		e_1 & 0\\
    		0 & e_2\\
    	\end{pmatrix}.
    \end{equation}
    We then localize it at $\mathrm{diag}(b^2a_1,b^3a_1,b^4a_1$), and denote the resulting algebra by $S^{-1}\cA_0$. 
    
    First, we can normalize the invertible arrows to be 1, which gives rise to the following map: 
    
    $$ G_{20}:=
    \begin{cases}
    	a_1 \mapsto \begin{psmallmatrix}
    		1 \\
    		0
    	\end{psmallmatrix} \\
    	a_2 \mapsto \begin{psmallmatrix}
    		0 \\
    		1
    	\end{psmallmatrix} \\
    	a_3 \mapsto \begin{psmallmatrix}
    		-YX+\delta_3 \\
    		X
    	\end{psmallmatrix} \\
    	a_4 \mapsto \begin{psmallmatrix}
    		-ZX'+\delta_4 \\
    		X'
    	\end{psmallmatrix} \\
    	b^1 \mapsto \begin{psmallmatrix}
    		\delta_1 & W
    	\end{psmallmatrix} \\
    	b^2 \mapsto \begin{psmallmatrix}
    		1 & \delta_2
    	\end{psmallmatrix} \\
    	b^3 \mapsto \begin{psmallmatrix}
    		1 & Y
    	\end{psmallmatrix} \\
    	b^4 \mapsto \begin{psmallmatrix}
    		1 & Z
    	\end{psmallmatrix}
    \end{cases}.$$

Because $\sum_i a_i b^i= -\delta_5 \cdot \unit,$ we know $X'=-1-X$, $X'Z=-\delta_5-\delta_2-XY$, $-ZX'+\delta_4=-\delta_5-\delta_1-\delta_3+YX$ and $W=YXY-ZXY-\delta_3 Y+(-\delta_5-\delta_2-\delta_4)Z$. 

Let $\cA_2$ be the algebra generated by $X,Y,Z$ satisfying the above relations, and  rewrite the above map:

$$G_{20}: S^{-1}\cA_{0} \to Mat(\cA_2), \quad G_{20}:=
\begin{cases}
	a_1 \mapsto \begin{psmallmatrix}
		1 \\
		0
	\end{psmallmatrix} \\
	a_2 \mapsto \begin{psmallmatrix}
		0 \\
		1
	\end{psmallmatrix} \\
	a_3 \mapsto \begin{psmallmatrix}
		-YX+\delta_3 \\
		X
	\end{psmallmatrix} \\
	a_4 \mapsto \begin{psmallmatrix}
		-\delta_5-\delta_1-\delta_3+YX \\
		-1-X
	\end{psmallmatrix} \\
	b^1 \mapsto \begin{psmallmatrix}
		\delta_1 & YXY-ZXY-\delta_3 Y+(-\delta_5-\delta_2-\delta_4)Z
	\end{psmallmatrix} \\
	b^2 \mapsto \begin{psmallmatrix}
		1 & \delta_2
	\end{psmallmatrix} \\
	b^3 \mapsto \begin{psmallmatrix}
		1 & Y
	\end{psmallmatrix} \\
	b^4 \mapsto \begin{psmallmatrix}
		1 & Z
	\end{psmallmatrix}
\end{cases}.$$

One can construct a right inverse map, denoted by $G_{20}$, of $G_{20}$ as follows:

$$G_{02}: \cA_2 \to S^{-1}\cA_{0}, \quad  G_{02}:=
\begin{cases}
	X \mapsto (b^2a_1)^{-1}(\alpha^2a_3)(b^3a_1) \\
	Y \mapsto (b^3a_1)^{-1}(b^3a_2)(b^2a_1) \\
	Z \mapsto (b^4a_1)^{-1}(b^4a_2)(b^2a_1) 
\end{cases}.$$

We show that if $\cA_2$ is an affine local chart of $\cA_0$—meaning that $G_{02}$ and $G_{20}$ are representations up to gerbe terms—then $\cA_2$ is a noncommutative deformation of $\C^2$. In particular, if $G_{02}$ is a representation map, it satisfies $$G_{02}(XY)=G_{02}(X)G_{02}(Y),$$ and the same holds for the others. We next analyze the restrictions on $X, Y, Z$.

First, we compute $G_{02}(XY)$ and $G_{02}(YX)$:
\begin{align*}
	G_{02}(XY)&= (b^2a_1)^{-1}(\alpha^2a_3b^3a_2)(b^2a_1) \\
	&= (b^2a_1)^{-1}\alpha^2 (-\delta_5-a_1b^1-a_2b^2-a_4b^4)a_2(b^2a_1) \\
	&= -\delta_5-\delta_2 - (b^2a_1)^{-1}\alpha^2a_4b^4a_2(b^2a_1)\\
	&= -2\delta_5-\delta_1-\delta_2-\delta_4+(-2\delta_5-\delta_1-\delta_3-\delta_4)(b^2a_1)^{-1}\alpha_2a_3b^3a_1- (b^2a_1)^{-1}b^2a_3b^3a_1,
\end{align*}
where the second equality comes from $\sum_i a_ib^i= \delta_5$ and the third equality uses $\alpha^2a_1=0$ and $b^ka_k=\delta_k$ for $k=1,\ldots,4$. 

\begin{align*}
	G_{02}(YX)&=(b^3a_1)^{-1}(b^3a_2\alpha_2a_3)(b^3a_1) \\
	&=(b^3a_1)^{-1}(b^3a_3)(b^3a_1)- (b^3a_1)^{-1}(b^3a_1\alpha_1a_3)(b^3a_1) \\
	&=\delta_3 - \alpha_1a_3b^3a_1 \\
	&= \delta_3 - (b^2a_1)^{-1}b^2a_1\alpha_1a_3b^3a_1 \\
	&=\delta_3+\delta_2 (b^2a_1)^{-1}\alpha_2a_3b^3a_1- (b^2a_1)^{-1}b^2a_3b^3a_1,
\end{align*} where the second equality comes from $a_1\alpha_1+a_2\alpha_2=e_5.$ Thus, $$G_{02}(XY)-G_{02}(YX)=(-2\delta_5-\delta_1-\delta_2-\delta_3-\delta_4)X-2\delta_5-\delta_1-\ldots-\delta_4.$$

Similarly, we can compute $G_{02}(XZ)$ and $G_{02}(ZX):$
\begin{align*}
	G_{02}(XZ)=&(b^2a_1)^{-1}(\alpha^2a_3)(b^3a_1)(b^4a_1)^{-1}(b^4a_2)(b^2a_1) \\
	=&(b^2a_1)^{-1}\alpha^2(-a_2b^2-a_4b^4)a_1(b^4a_1)^{-1}(b^4a_2)(b^2a_1)\\
	=&-(b^4a_1)^{-1}(b^4a_2b^2a_1)-(b^2a_1)^{-1}\alpha^2a_4b^4a_2b^2a_1\\
	=&-(b^4a_1)^{-1}(b^4a_2b^2a_1)-(b^2a_1)^{-1}\alpha^2(-\delta_5-a_2b^2-a_3b^3)a_2b^2a_1\\
	=&-(b^4a_1)^{-1}(b^4a_2b^2a_1)+\delta_5+\delta_2+(b^2a_1)^{-1}\alpha^2a_3b^3a_2b^2a_1\\
	=&-(b^4a_1)^{-1}(b^4a_2b^2a_1)+\delta_5+\delta_2+(b^2a_1)^{-1}\alpha^2a_3b^3(-\delta_5- a_1b^1-a_3b^3-a_4b^4 )a_1\\
	=&-(b^4a_1)^{-1}(b^4a_2b^2a_1)+\delta_5+\delta_2+(-\delta_5-\delta_1-\delta_3)(b^2a_1)^{-1}\alpha^2a_3b^3a_1-(b^2a_1)^{-1}\alpha^2a_3b^3a_4b^4a_1\\
	=&-(b^4a_1)^{-1}(b^4a_2b^2a_1)+\delta_5+\delta_2+(-\delta_5-\delta_1-\delta_3)(b^2a_1)^{-1}\alpha^2a_3b^3a_1\\
	& -(b^2a_1)^{-1}\alpha^2(-\delta_5-a_2b^2-a_4b^4)a_4b^4a_1\\
	=&-(b^4a_1)^{-1}(b^4a_2b^2a_1)+(-\delta_5-\delta_1-\delta_3)(b^2a_1)^{-1}\alpha^2a_3b^3a_1\\ &+(\delta_4+\delta_5)(b^2a_1)^{-1}\alpha^2a_4b^4a_1-\delta_1-(b^2a_1)^{-1}b^2a_3b^3a_1\\
	&=-\delta_5-\delta_4-\delta_1-(b^4a_1)^{-1}(b^4a_2b^2a_1)+(-2\delta_5-\delta_1-\delta_3-\delta_4)(b^2a_1)^{-1}\alpha^2a_3b^3a_1\\ &-(b^2a_1)^{-1}b^2a_3b^3a_1.\\
	\end{align*}
\begin{align*}
	G_{02}(ZX)&=(b^4a_1)^{-1}(b^4a_2\alpha^2a_3)(b^3a_1)\\
	&=(b^4a_1)^{-1}b^4(e_5-a_1\alpha^1)a_3b^3a_1\\
	&=(b^4a_1)^{-1}b^4a_3b^3a_1-\alpha^1a_3b^3a_1\\
	&=(b^4a_1)^{-1}b^4(-\delta_5-a_1b^1-a_2b^2-a_4b^4)a_1-\alpha^1a_3b^3a_1.\\
	&=-\delta_5-\delta_1-\delta_4-(b^4a_1)^{-1}b^4a_2b^2a_1-\alpha^1a_3b^3a_1\\
	&=-\delta_5-\delta_1-\delta_4+\delta_2(b^2a_1)^{-1}\alpha^2a_2b^3a_1-(b^4a_1)^{-1}b^4a_2b^2a_1-(b^2a_1)^{-1}b^2a_3b^3a_1.\\
\end{align*}
Thus, $G_{02}(XZ)-G_{02}(ZX)=(-2\delta_5-\delta_1-\delta_2-\delta_3-\delta_4)(b^2a_1)^{-1}\alpha^2a_3b^3a_1=(-2\delta_5-\delta_1-\delta_2-\delta_3-\delta_4)X.$

Furthermore, we obtain $$G_{02}(X(Y-Z))-G_{02}((Y-Z)X)= -2\delta_5-\delta_1-\ldots-\delta_4.$$ 

Recall that $\cA_2$ is generated by $X,Y,Z$ satisfying $Z+XZ-XY+\delta_5+\delta_2= Z-X(Y-Z)+\delta_5+\delta_2=0$. Setting $T:=Y-Z$, we obtain $$\cA_2\cong \C[\delta_1,\ldots,\delta_5]\langle X,T\rangle /(XT-TX+2\delta_5+\delta_1+\ldots+\delta_4). $$ In other words, the Claim holds and $\cA_2$ is the noncommutative deformation of the coordinate ring of $\C^2$. Moreover, $\cA_2$ serves as an affine local chart of $\cA_0$ via the representation maps $G_{02}$ and $G_{20}$.

Similarly, one can compute other affine local charts of $\cA_0$ by localizing other collections of arrows.

\end{example}

\begin{remark}
	By Theorem \ref{thm: dpre}, the deformed preprojective algebra $\cA_{0}$ appearing in the example has a geometric meaning, which is the bulk-deformed Maurer-Cartan algebra of the zero section $\bL$ in the plumbing space, while $\cA_2$ serves as a local chart of the mirror space.
\end{remark}

By the notion of a `local chart' in quiver stack, we can define a quiver algebra with relations to be commutative if all its local charts are commutative.

We know that the quiver algebras of affine type $\tilde{A}_n$ and $\tilde{D}_4$ at complex moment-map level zero are commutative. Moreover, in Corollary \ref{cor: An} for $\tilde{A}_n$, it remains commutative if $\sum_{i=0}^n a_i = 0$.
\begin{conjecture}
	All quiver algebras of affine types at complex moment-map level zero are commutative.
	Moreover, if we stay in the hyperplane $\sum_i \mu_i^{\C} = 0$, then the quiver algebra remains commutative.
	Thus, the total moment-map value $\sum_i \mu_i^{\C}$ measures the degree of noncommutativity.
\end{conjecture}

\subsection{Noncommutative Framed Torsion-Free Sheaves}

Motivated by the famous ADHM construction \cite{ADHM78}, Donaldson established a one-to-one correspondence between the isomorphism classes of quiver representations and the isomorphism classes of framed holomorphic vector bundles over $\C\bP^2$. Here, a framed bundle means that the trivialization of the bundle is fixed along the line at infinity $l_\infty \subset \C\bP^2$. Subsequently, Nakajima \cite{Nak99,Nak07} generalized the correspondence over the compactification of the asymptotic locally Euclidean (ALE) spaces introduced in \cite{KN90}. We will denote such a compactification by $\widetilde{\bP^2/\Gamma},$ where $\Gamma$ is a finite subgroup in $\mathrm{SL}_2(\C).$

An analogue of this correspondence also exists in the noncommutative setting. In \cite{KKO01}, Kapustin-Kuznetsov-Orlov introduced the noncommutative projective plane $\bP^2_{\tau}$ and showed that there is a one-to-one correspondence between isomorphism classes of framed sheaves over $\bP^2_{\tau}$ and deformed ADHM data. In \cite{BGK02}, Baranovsky-Ginzburg-Kuznetsov further generalized this result to the noncommutative surface $\bP^2_{\tau,\Gamma}$ associated to the finite subgroup $\Gamma \subset \mathrm{SL}_2(\C)$. 


In this subsection, we identify the images of bulk-deformed framed Lagrangian branes with the noncommutative monadic complexes of~\cite{KKO01,BGK02}, extending the construction of~\cite{HLT24}. The construction uses a compact reference immersion to transform the framed objects considered below.


\begin{thm}\label{thm: monad}
	Let $M$ be the plumbing of $T^*\mathbb{S}^2$ according to the affine ADE Dynkin diagram, and $\bL$ be the core Lagrangian. Consider a framed Lagrangian brane $(L^\fr,\mathcal{E})$ of rank $(\vec{m},\vec{n})$, and let $b_0$ be a point in the bulk-deformed Maurer-Cartan space of $(L^\fr,\mathcal{E})$. Then the image of $(L^\fr,\mathcal{E},b_0)$ under the localized mirror functor $\cF^{(\bL,b)}$ can be identified with the monadic complex of torsion-free sheaves over the noncommutative surface $\widetilde{\C^2/\Gamma}$, which extends uniquely (up to isomorphism) to a framed torsion-free sheaf on the noncommutative projective surface $\bP^2_{\tau,\Gamma}$.
	
	In other words, the framed Lagrangian brane $(L^\fr,\mathcal{E},b_0)$ is mirror to the framed torsion-free sheaf over $\bP^2_{\tau,\Gamma}$.
\end{thm}

Before giving the proof, we first recall the definition of noncommutative surfaces $\widetilde{\C^2/\Gamma}$ and $\bP^2_{\tau,\Gamma}$. 

Let $L$ be a two-dimensional symplectic vector space with basis $\{x,y\}$ such that $\omega(x,y)=1$, and $\Gamma$ be a finite subgroup of $\mathrm{Sp}(L)$. 

\begin{defn}
	The noncommutative surface $\bP^2_{\tau,\Gamma}$ is defined to be the following graded algebra $$A^{\tau}:=(\C\langle x,y,z \rangle\#\Gamma)/ ( [x,z]=[y,z]=0,[x,y]=\tau z^2 ) ,$$ where $\tau$ is an element in the center $Z(\C \Gamma)$ of the group algebra, $\Gamma$ acts naturally on $x,y$, and $\Gamma$ acts trivially on $z$. The algebra $A^\tau$ has a grading defined by  $\mathrm{deg}\,x=\mathrm{deg}\,y=\mathrm{deg}\,z=1$ and $\mathrm{deg}\, a=0$ for all $a \in \C\Gamma$.
\end{defn}

\begin{remark}
	Recall that given a variety $X= \mathrm{Proj}(A)$, where $A$ is a graded Noetherian ring, one can associate to a finitely generated graded $A$-module a coherent sheaf over $X$. Moreover, if  $A$ is generated by finitely many elements in $A_1$ over $A_0$, this assignment induces an equivalence  $$\mathrm{Mod}^{fg}_A/ \mathrm{Mod}^{fg}_{A,\mathrm{torsion}} \to \mathrm{Coh}(X).$$ Hence, one can regard a graded algebra as a (noncommutative) projective variety. 
	
	In particular, when $\tau=0$ and $\Gamma$ is trivial, $A^\tau \cong \C[x,y,z]$ is the homogeneous coordinate ring of $\bP^2$. For more details on the noncommutative surface $\bP^2_{\tau,\Gamma}$, the reader is referred to Section 1 of \cite{BGK02}.
\end{remark}

Next, we introduce the noncommutative $\widetilde{\C^2/\Gamma}$, which can be regarded as the `coordinate ring' of $\bP^2_{\tau,\Gamma} \setminus \bP^1_{\tau,\Gamma}.$

\begin{defn}
	The noncommutative surface $\widetilde{\C^2/\Gamma}$ is defined to be $$B^\tau:= A^{\tau}/(z-1) \cong (\C \langle x, y\rangle \#\Gamma)/([x,y]=\tau).$$
\end{defn}

\begin{remark}
	Let us clarify the name noncommutative $\widetilde{\C^2/\Gamma}$. In \cite{CBH98}, the authors proved that $B^\tau$ is Morita equivalent to the deformed preprojective algebra. In our context, this is $\cA_{\bL}^\tau$ where $\bL$ is the zero section in the plumbing space corresponding to $\Gamma$ and $\tau$ corresponds to the bulk deformation parameter. 
	
	On the other hand, when $\tau=0$, the preprojective algebra $\cA_\bL$ is derived equivalent to the category of coherent sheaves over $\widetilde{\C^2/\Gamma}$ by the derived McKay correspondence \cite{KV00,BKR01}. Hence, we refer to $B^\tau$ as the noncommutative $\widetilde{\C^2/\Gamma}$.
\end{remark}

\begin{proof}[Proof of Theorem \ref{thm: monad}]
	Let $B^\tau$ be the noncommutative surface $\widetilde{\C^2/\Gamma}$. By Proposition 5.3.9 and Proposition 5.3.10 in \cite{BGK02}, every torsion-free sheaf over $B^\tau$ admits a unique (up to isomorphism) extension to a framed torsion-free sheaf over $\bP^2_{\tau,\Gamma}$. Thus, it suffices to prove that $\cF^{(\bL,b)}(L^\fr,\mathcal{E},b_0)$ is isomorphic to the following monadic complex of torsion-free sheaves over $B^\tau$, introduced in Equation 4.1.2 of \cite{BGK02}, under the Morita equivalence:
	\begin{equation}\label{eq: monad1}
		C^\bullet : 0 \to B^\tau \otimes_{\C[\Gamma]} V \xrightarrow[]{\alpha} B^\tau \otimes_{\C[\Gamma]} ((L\otimes_\C V)\oplus W) \xrightarrow[]{\beta}B^\tau \otimes_{\C[\Gamma]} V \to 0,
	\end{equation} where $L$ is the natural representation of the finite subgroup $\Gamma \subset \mathrm{SL}_2(\C)$ with basis $\{x,y\}$. Here $\alpha=\begin{pmatrix}
		B \cdot - (h_x \cdot x+ h_y \cdot y) \otimes \mathrm{Id}_V \\
		J \cdot
	\end{pmatrix}$ and $\beta= \begin{pmatrix}
	B\cdot- (h_x \cdot x+ h_y \cdot y) \otimes \mathrm{Id}_V, & I \cdot
\end{pmatrix}$, where $\{h_x,h_y\}$ is the dual basis of $L^*$ and $(B,I,J) \in \Hom_\Gamma (V, L \otimes V) \oplus \Hom_\Gamma (W,V) \oplus \Hom_\Gamma (V,W).$

On the other hand, the generic bulk cycles do not intersect the holomorphic discs counted by $m_1^{b,b_0}$. Hence, the image of $(L^\fr,\mathcal{E},b_0)$ under the localized mirror functor $\cF^{(\bL,b)}$ can be computed as in \cite[Thm. 5.13]{HLT24}:
\begin{equation}\label{eq: monad2}
	\oplus_v \cA_\bL e_v \otimes \C^{m_v} \langle M_v \rangle \xrightarrow{d_1} \oplus_{a} \cA_{\bL} e_{t(a)} \otimes \C^{m_{h(a)}} \langle X_a \rangle  \bigoplus \oplus_v \cA_\bL e_v \otimes \C^{n_v}\langle J_v \rangle \xrightarrow{d_2} \oplus_v \cA_\bL e_v \otimes \C^{m_v} \langle P_v \rangle,
\end{equation} where $M_v$, $P_v$ are the maximal points and minimum points of the Morse function $f_v$ on $\mathbb{S}^2_v$ respectively, while $X_a \in \CF^1(\mathbb{S}^2_{t(a)},\mathbb{S}^2_{h(a)})$ and $J_v \in \CF^1(\bL_v, F_v)$ are degree $1$ immersed sectors. 
Furthermore, the first differential has the form $$d_1(\eta M_v):=m_{1}^{\bb,b_0}(\eta M_v)= \sum_{t(\bar{a})=v} (B_{\bar{a}} \eta) X_{\bar{a}} + \sum_{h(a)=v} ( \eta x_a )X_a + j_v\eta J_v,$$ and the second one is $$d_2(\eta' X_a):= m_{1}^{\bb,b_0}(\eta' X_a)=  \eta' x_{\bar{a}} P_{h(a)}+ B_{\bar{a}} \eta' P_{t(a)}; $$  $$d_2(\eta'' J_v):=m_{1}^{\bb,b_0}(\eta'' J_v) = i_v \eta'' P_v.$$ Notice that after a suitable change of coordinates, the relations of the deformation space $\cA_\bL$ and the differentials become polynomials. Thus we can extend the coefficients to the complex numbers and obtain Equation \ref{eq: monad2} over $\C$.
	
	To identify the complex, we need to be more precise about the equivalence. Let $N_1,\ldots,N_k$ be the irreducible representations of $\Gamma$. Then by Wedderburn theorem, there exists idempotent $f_i \in \C[\Gamma]$ such that $f_i \C[\Gamma] \cong N_i$ as a right $\C[\Gamma]$-module. One can check that $f:= f_1 + \cdots +f_k$ is a full idempotent of $B^{\tau}$, i.e. $B^{\tau} f B^\tau= B^\tau$. Hence, it induces a Morita equivalence $F:B^\tau-\mathrm{mod} \to fB^\tau f-\mathrm{mod}$, which sends a left $B^\tau$-module $M$ to $fM$.

	Moreover, \cite{CBH98} proved that $fB^\tau f$ is isomorphic to the deformed preprojective algebra $\cA_\bL$, which is the bulk-deformed Maurer-Cartan algebra of the Lagrangian. In particular, it sends $\tau$ to the bulk deformation parameter $\delta$, and idempotent $f_i$ to the trivial path $e_i$ at the $i$-th vertex. Therefore, $F$ induces a Morita equivalence between $B^\tau$ and $\cA_\bL,$ and it remains to prove that the monadic complex $C^\bullet$, Equation \ref{eq: monad1}, is equivalent to $\cF^{(\bL,b)}(L^\fr,\mathcal{E})$ under the functor $F$.
	
	By construction, one can choose $\vec{m}$ so that $V\cong \oplus_i (N_i)^{\oplus m_i}$ as a left $\Gamma$-module. Then \begin{align*}
		F(B^\tau \otimes_{\C[\Gamma]}V)&= f B^\tau \otimes_{\C[\Gamma]}V\cong f B^\tau \otimes_{\C[\Gamma]} \oplus_i (N_i)^{\oplus m_i}\\
		&\cong f B^\tau \otimes_{\C[\Gamma]} \oplus_i (\C[\Gamma]f_i)^{\oplus m_i} \cong \oplus_i \left( f B^\tau \otimes_{\C[\Gamma]} \C[\Gamma]f_i \right)^{\oplus m_i} \\
		&\cong \oplus_i (fB^\tau f)f_i \otimes_\C \C^{m_i} \cong \oplus_i A_\bL e_i \otimes_\C \C^{m_i}= \CF^0((\bL,b),(L^\fr,E,b_0)).
	\end{align*} A similar isomorphism holds in the other degrees. Hence, $F(C^\bullet)$ is isomorphic to $\cF^{(\bL,b)}(L^\fr,\mathcal{E})$ as a graded vector space.

    Next, we prove that the differentials also coincide. For the linear part of the differential $f\otimes (B,I,J)$, it suffices to show $(B,I,J)$ decomposes. Recall that $(B,I,J)$ are morphisms of $\Gamma$-representations, i.e. $$(B,I,J) \in \Hom_\Gamma (V, L \otimes V) \oplus \Hom_\Gamma (W,V) \oplus \Hom_\Gamma (V,W).$$ Thus, each component has a decomposition according to the McKay correspondence: \begin{align*}
    	\Hom_\Gamma (V, L \otimes V)&\cong  \Hom_\Gamma( \oplus_i N_i^{\oplus m_i}, \oplus_j (L \otimes N_j)^{\oplus m_j} ) \\
    	& \cong \oplus_{i,j} \Hom_\C(\C^{\oplus m_i}, \C^{\oplus m_j}) \otimes \Hom_\Gamma(N_i, L\otimes N_j).
    \end{align*}
By the McKay correspondence, this is precisely the representation space of the McKay quiver.

Similarly, if $W \cong \oplus_j N_i^{\oplus n_j},$ $\Hom_\Gamma(W,V)\cong \oplus_{i} \Hom_\C(\C^{\oplus n_i}, \C^{\oplus m_i}) \otimes \Hom_\Gamma(N_i,N_i) \cong \Hom_\C(\C^{\oplus n_i}, \C^{\oplus m_i})$ by Schur's Lemma. Moreover, $\Hom_\Gamma(V,W)\cong \oplus_i \Hom_\C(\C^{\oplus m_i},\C^{\oplus n_i} )$. Hence, $(B,I,J)$ decomposes into the framed quiver representation of the McKay quiver, which can be identified with the Maurer-Cartan element $b_0$ of $(L^\fr,\mathcal{E})$, where $\mathcal{E}$ is a trivial bundle of rank $(\vec{m},\vec{n})$. Therefore, the linear part of the differential can be identified under the functor $F$.

In addition, the morphism $F(h_x\cdot x+ h_y \cdot y)$ also respects the decomposition under the isomorphism of the graded vector space. Indeed, by Lemma 3.3 in \cite{CBH98}, the isomorphism between $fB^\tau f$ and $\cA_\bL$ is defined by sending the bimodule generators of $fB^\tau_1 f$ over $fB_0^\tau f$ to the arrows in the deformed preprojective algebras, where $B^\tau_k$ is the degree $k$ component of $B^\tau$. Hence, multiplication by $f(h_x\cdot x+ h_y \cdot y)$ corresponds to multiplication by the sum of arrows. One can check that under the decomposition $\Hom_{\cA_\bL}(\cA_\bL,\cA_\bL) \cong \oplus_{i,j} \Hom_{\cA_\bL}(\cA_\bL e_i,\cA_\bL e_j)$, the differential also coincides. Hence, the result holds.
\end{proof}

\begin{remark}
	 When the underlying graph $D$ is affine $A_0$ diagram, i.e. the graph with a single vertex and one self-loop, $\cF^{(\bL,b)}(L^\fr,\mathcal{E},b_0)$ reduces to the monadic complex of torsion-free sheaves over the noncommutative plane $\C^2$, where the correspondence was first developed in \cite{KKO01}. In particular, the bulk-deformed Maurer-Cartan equation of $\bL$ is the deformed ADHM equation.
\end{remark}

Notice that the deformed preprojective algebra $\cA_{\bL}$ is isomorphic to the spherical algebra of $B^\tau$; that is $\cA_\bL \cong fB^\tau f$, where $f$ is the multiplicity-free idempotent. See \cite[Theorem 3.4]{CBH98}. 
Moreover, the brane $(\bL,\mathcal{E})$ contains the information of multiplicities. 
We should be able to recover the deformed skew group ring $B^\tau$ from $(\bL,\mathcal{E})$.
Below, we provide an affirmative answer by constructing an isomorphism between the higher-rank Maurer-Cartan algebra $\cA_{(\bL,\mathcal{E})}$ and the deformed skew group ring $B^\tau$ when $\tau=0$.

More precisely, given an affine $ADE$ Dynkin diagram, let $\vec{n}$ be the primitive vector such that $C \vec{n}=\vec{0},$ where $C$ is the Cartan matrix of the graph $D$. We would like to compare $\cA_{(\bL,\mathcal{E})}$, where the trivial vector bundle $\mathcal{E}$ is of rank $\vec{n}$, with the algebra $B^\tau$.  Recall that $\vec{n}$ coincides with the dimension vector of the irreducible representations of the corresponding finite subgroup $\Gamma $, i.e. $\vec{n}:=(n_1, \ldots, n_r)$, where $n_i= \mathrm{dim}\, N_i$ and $N_i$ ranges over the irreducible representations of $\Gamma$.

\begin{prop}
	Let $\tau$ be zero, $B$ be the skew group ring $\C[x,y] \# \Gamma$, and $\cA_{(\bL,\mathcal{E})}$ be the higher-rank Maurer-Cartan algebra, where $\mathcal{E}$ is the trivial vector bundle of rank $\vec{n}$ over $\bL$. Then $\cA_{(\bL,\mathcal{E})}$ is isomorphic to $B$.
\end{prop}

\begin{proof}
	Let $\mathcal{C}$ be the dg algebra $\wedge^\bullet \C^2$ with trivial differential, whose multiplication is given by the wedge product. Since $\Gamma$ is a finite subgroup of $\mathrm{SL}_2(\C)$, $\mathcal{C}$ admits a natural $\Gamma$ action. Moreover, $\mathcal{C}$ can be viewed as an $A_\infty$-category with a unique object $u$, whose endomorphism is  $\wedge^\bullet \C^2$.
	
	The constructions introduced in Section 3.5 of \cite{Lam15} induce $A_\infty$-structures (differential graded structures) on $\mathcal{C}/\Gamma$ and $\Hom_\Gamma(\C\Gamma, \mathcal{C} \otimes \C\Gamma)$. In particular, $\mathcal{C}/\Gamma$ contains $u \otimes N_i$ as objects for all $i$ and $\Hom_{\mathcal{C}/\Gamma}^\bullet(u\otimes N_i, u \otimes N_j):= \Hom_{\Gamma}(N_i, \mathcal{C}^\bullet(u,u) \otimes N_j)$ as morphism spaces. Moreover, we have the following decomposition of $\Hom_\Gamma(\C\Gamma, \mathcal{C}^1 \otimes \C\Gamma)$: $$ \oplus_{i,j}\Hom_\C(\C^{\oplus n_i}, \C^{\oplus n_j}) \otimes \Hom_\Gamma(N_i, \C^2\otimes N_j) \cong \oplus_{i,j}\Hom_\C(\C^{\oplus n_i}, \C^{\oplus n_j}) \otimes \Hom_{\mathcal{C}/\Gamma}^1(u\otimes N_i, u \otimes N_j).$$ Hence, the $A_\infty$-structure of $\Hom_\Gamma(\C\Gamma, \mathcal{C} \otimes \C\Gamma)$ can be induced by that of $\mathcal{C}/\Gamma$ and matrix multiplication.
	
	Next, we show that $\mathcal{C}/\Gamma$ is isomorphic to $\CF(\bL,\bL)$ as an $A_\infty$-algebra. This implies that $\CF((\bL,\mathcal{E}),(\bL,\mathcal{E}))$ is isomorphic to $\Hom_\Gamma(\C\Gamma, \mathcal{C} \otimes \C\Gamma)$, and likewise for their Koszul duals. By McKay correspondence, each irreducible representation corresponds to a vertex in the affine $ADE$ Dynkin diagram. This suggests that $u \otimes N_i$ corresponds to a sphere component $L_i$ of $\bL$. By the Schur Lemma and the above decomposition, we see that $\mathcal{C}/\Gamma$ and $\CF(\bL,\bL)$ are isomorphic as graded vector spaces.
	
	Notice that, after the change of coordinates on the Koszul dual side described in Theorem~\ref{thm: dpre}, the Koszul dual of $\CF(\bL,\bL)$ is isomorphic to the Koszul dual of an $A_\infty$-algebra with the same underlying graded vector space as $\CF(\bL,\bL)$, but whose only nontrivial $A_\infty$-operation is $m_2$. By abuse of notation, we will still denote this $A_\infty$-algebra by $\CF(\bL,\bL)$. Thus, it remains to show that $m_2$ coincides. Here, we will check the operator $m_2$ on the degree-one components; the remaining cases are similar. 
	
	Observe that $\oplus_{i,j}\Hom_{\mathcal{C}/\Gamma}^1(u\otimes N_i, u \otimes N_j)$ is one-dimensional along each arrow of the McKay quiver. By definition, $m_2: \Hom_{\mathcal{C}/\Gamma}^1(u\otimes N_i, u \otimes N_j) \otimes \Hom_{\mathcal{C}/\Gamma}^1(u\otimes N_j, u \otimes N_k) \to \Hom_{\mathcal{C}/\Gamma}^2(u\otimes N_i, u \otimes N_k) \cong \Hom_\Gamma (N_i, N_k),$ which is zero unless $i=k$. This follows from the Schur Lemma and the identification $\wedge^2 \C^2 \cong \C$ as a trivial $\Gamma$-representation, as $\Gamma$ is a finite subgroup of $\mathrm{SL}_2(\C)$. Furthermore, there exists a basis $\{a_{ij}\}$ of $\oplus_{i,j}\Hom_{\mathcal{C}/\Gamma}^1(u\otimes N_i, u \otimes N_j)$ such that $m_2(a_{ij},a_{ji})= m_2(a_{ik},a_{ki})$ for every vertex $i$ and adjacent vertices $j,k$.
	
	This shows $\mathcal{C}/\Gamma$ is isomorphic to $\CF(\bL,\bL)$ as an $A_\infty$-algebra, by mapping $a_{ij}$ to the basis of $\CF^1(L_i,L_j).$ By Proposition 3.5.5 and Proposition 3.5.6 in \cite{Lam15}, we know $\Hom_\Gamma(\C\Gamma, \mathcal{C} \otimes \C\Gamma) \cong (\mathcal{C}^{op} \# \Gamma)^{op}$ and taking Koszul dual commutes with smash product. As a consequence, $$H^0(\Hom_\Gamma(\C\Gamma, \mathcal{C} \otimes \C\Gamma)^!) \cong \C[x,y]\#\Gamma \cong H^0(\CF((\bL,\mathcal{E}),(\bL,\mathcal{E}))^!) \cong \cA_{(\bL,\mathcal{E})}.$$
\end{proof}

\subsection{Noncommutative Deformation of the Resolved Conifold}
In this subsection, we show that a bulk deformation of the conifold smoothing gives rise to a noncommutative deformation of the noncommutative crepant resolution of the conifold. We then construct local charts for this deformation using the framework of quiver algebroid stacks \cite{LNT23}. This provides a three-dimensional analogue of Kawamata's noncommutative surfaces.

Let $X:=\{(u_1,v_1,u_2,v_2,z) \in \C^4 \times \C^* \mid u_1v_1=z-1, u_2v_2=z+1\}$ be the smoothing of the conifold. The space $X$ admits a double conic fibration over $\C^*$ given by projection onto the last coordinate, see Figure \ref{fig:conifold}. Moreover, the space $X$ carries a Lagrangian torus fibration $\pi: X \to \R^3$ given by $\pi(u_1,v_1,u_2,v_2,z):= \left(
\frac{1}{2}\bigl(|u_1|^2-|v_1|^2\bigr),
\frac{1}{2}\bigl(|u_2|^2-|v_2|^2\bigr),
ln |z|
\right),$ see for example \cite{CPU16,FHLY17}. We denote the preimage of $0$ by $\bL,$ which consists of two 3-spheres $\bS^3$ intersecting cleanly along two circles. 

\begin{figure}[htbp]
	\centering
	\includegraphics[width=0.6\textwidth]{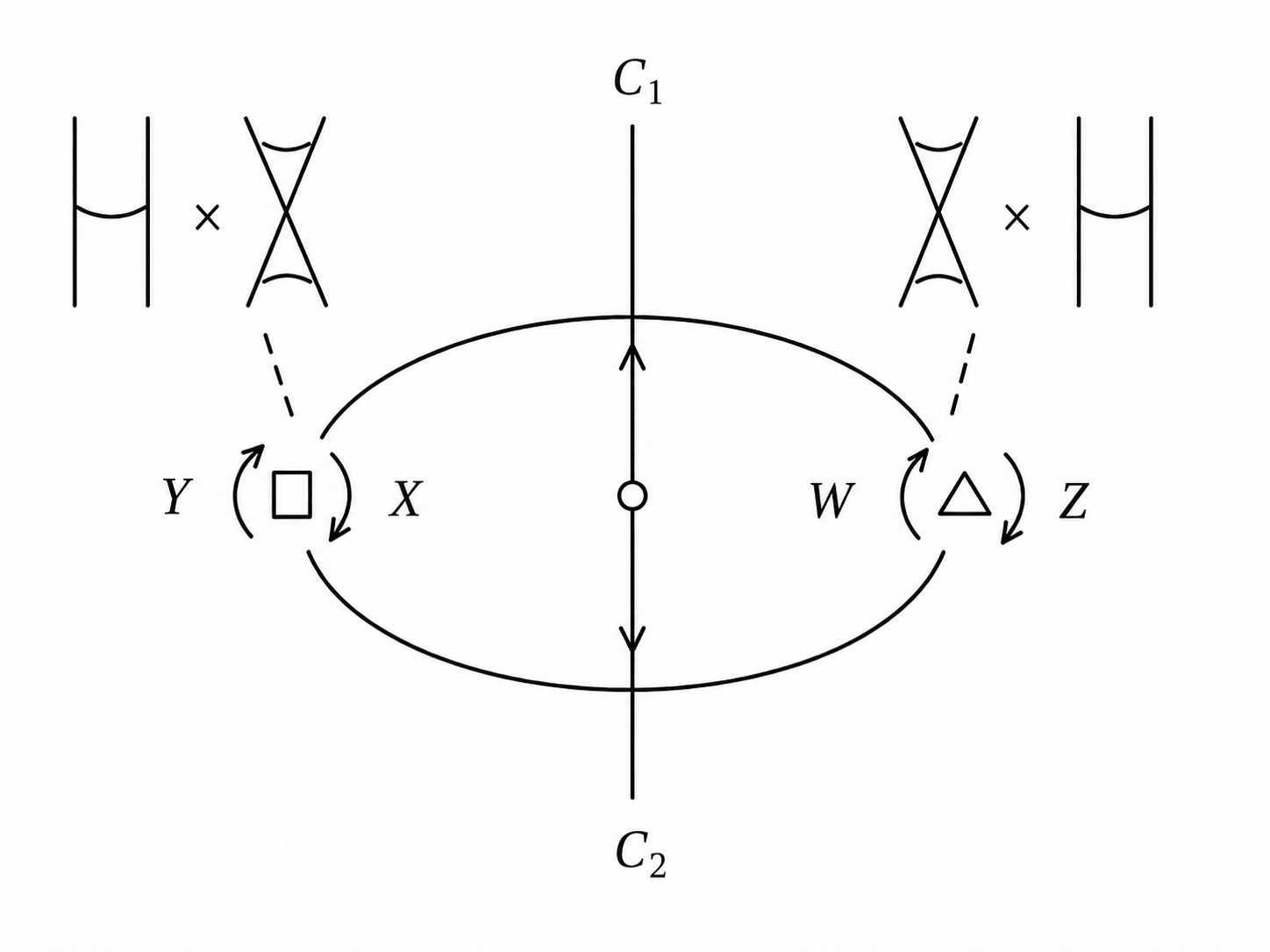}
	\caption{The double conic fibration of $X$, with singular fibers over $\pm 1$.}
	\label{fig:conifold}
\end{figure}

We will use two basic bulk cycles $C_1$ and $C_2$. For each $z\in \C^*$, let
\[
C_1(z)=\{u_1v_1=z-1\},\qquad
C_2(z)=\{u_2v_2=z+1\},
\] so that the fiber $X_z=C_1(z)\times C_2(z)$.  For a regular value \(z\), let \(S_i(z)\subset C_i(z)\) be the standard
vanishing circle. Then \(S_1(z)\times S_2(z)\) is a two-torus inside
\(X_z\). We choose the diagonal circle
$\Delta(z)\subset S_1(z)\times S_2(z),
$ representing the class \((1,1)\in H_1(S_1(z)\times S_2(z))\).

We now choose two rays in the base, which start from the origin in the base and go to infinity, and they are chosen so that they meet the two components of \(\bL\) respectively. The first bulk cycle \(C_1\) is obtained by choosing, in a regular fiber over one of the rays, a real three-dimensional submanifold whose compact circle direction is \(\Delta(z)\), and then parallel transporting it along the ray. Thus, locally over the ray, \(C_1\) is given by the parallel transport of $\Delta(z)\times \R^2$. Topologically,
$C_1\simeq S^1\times \R^3.$
Its intersection with the corresponding component of \(\bL\) has
fiberwise direction given by the diagonal circle.

The second bulk cycle \(C_2\) is defined analogously, using the ray
associated with the other \(3\)-sphere component of \(\bL\). Namely, over the ray, we choose the same type of fiberwise real three-dimensional submanifold. 
\begin{prop}\label{prop:coni}
	The bulk-deformed Maurer-Cartan algebra $\cA_\bL$ of $\bL$ by the cycle $\mathbbm{b}:= \delta_1 C_1 + \delta_2 C_2$ equals the deformed noncommutative crepant resolution of the conifold. More precisely, $$\cA_\bL \cong \C[\delta_1,\delta_2] Q/(\partial_{x_e} \Phi: e \in E), $$ where $\Phi= (1+\delta_1)xyzw-(1+\delta_2)wzyx+(\delta_1-\delta_2) zw+(\delta_1-\delta_2) xy$ and $\partial_{x_e}$ denotes the cyclic derivative.
\end{prop}
\begin{proof}
	Let $\bL_i$ denote the Lagrangian sphere component of $\bL$ which intersects the bulk cycle $C_i$. The two components $\bL_1$ and $\bL_2$ intersect cleanly along two disjoint circles, which we denote by $S^1_1$ and $S^1_{-1}$, lying over $z=1$ and $z=-1$, respectively. 
	
	The formal deformation space of $\bL$ before introducing bulk deformations was studied in Section 6.1 of \cite{FHLY17}. More precisely, using the Morse model, the authors constructed an $A_\infty$-structure on $\CF^*(\bL,\bL)$. They observed that, for a suitable choice of Morse functions $f_i$ on the components $\bL_i$, there is a unique Morse trajectory connecting $S^1_1$ and $S^1_{-1}$. Moreover, they showed that the obstruction equation is given by $
	m_0^b=\sum_e \partial_{x_e}\bigl(\Phi|_{\delta_1=\delta_2=0}\bigr),$
	where the relevant contributions come from these Morse trajectories together with the flow lines along the clean intersection circles.
	
	After bulk-deformation, 	$m_0^{b,\mathbbm{b}}$ receives additional contributions from pearl configurations with interior marked points constrained to the bulk cycles. Observe that $C_i$ intersects $\bL_i$ along a clean diagonal circle, denoted by $K_i$. Thus, as before, for the chosen Morse functions, there are Morse trajectories from the clean intersection $K_i$ to $S^1_{1}$ and $S^1_{-1}$. This is because $K_i$ intersects the stable manifold of $S^1_{\pm 1}$ transversely in one point.
	
	Thus, in addition to the original contributions to $m_0^b$, the
	bulk-deformed term $m_0^{b,\mathbbm b}$ also counts configurations in which a constant disk component carries an interior marked point constrained to $C_i$ and is attached to the above Morse trajectory. Besides, the configuration obtained by attaching the Morse trajectory from \(K_i\) to the Morse flow along the clean circle \(S^1_{\pm 1}\) becomes a stable pearl configuration. These contribute to the remaining terms of \(m_0^{b,\mathbbm b}\), i.e. $\sum_e\partial_{x_e} \left( \delta_1xyzw-\delta_2wzyx+(\delta_1-\delta_2) zw+(\delta_1-\delta_2) xy \right)$.
	
\end{proof}

Moreover, using the techniques of quiver algebroid stacks (See Section 1.1 of \cite{LNT23} for an introduction), we can produce the local noncommutative charts of deformation of this NCCR $\cA_\bL,$ which serves as a three-dimensional analogue of Kawamata's construction \cite{Kaw24A}. 

\begin{lemma}\label{lem:conichart}
	There exist local noncommutative charts $\cA_1$ and $\cA_2$ of $\cA_\bL$, which glue into a noncommutative deformation of the resolved conifold. Moreover, the local noncommutative charts are noncommutative deformations of $\C^3$.
\end{lemma}  
\begin{proof}
	
	The local noncommutative charts of $\cA_\bL$ can be constructed via arrow localization. More precisely, $\cA_1$ and $\cA_2$ are obtained by localizing at the arrows $z,x$ respectively.
	
	Motivated by the commutative cases and the above example, we introduce normalized generators in the
	localized algebra $\mathcal A_{0}[z^{-1}]$. Namely, after inverting the arrow $z$, we use the invertibility of $z$ to perform
	a gauge normalization, and introduce
	$z^{-1}x, yz, wz $ as the noncommutative analogues of affine coordinates on the corresponding
	local chart. This suggests the following representations between $\cA_1$ and $\cA_0[z^{-1}]$, which are inverse to each other up to gerbe terms:
	$$G_{01}:=\begin{cases}
		x_1 \mapsto z^{-1}x \\
		y_1 \mapsto yz \\
		z_1 \mapsto wz 
	\end{cases}, \quad G_{10}:=\begin{cases}
		z \mapsto 1 \\
		x \mapsto x_1\\
		y \mapsto y_1\\
		w \mapsto z_1,
	\end{cases}$$  where $G_{01}$ is representation from $\cA_1$ to $\cA_0$.
	
	Thus, once this gauge normalization is chosen, the presentation of
	$\mathcal A_1$ is determined by the localized relations of
	$\mathcal A_{0}[z^{-1}]$. Explicitly,
	one obtains
	$\mathcal A_1 \cong
	\mathbb C\langle x_1,y_1,z_1\rangle/I_1,$
	where $I_1$ is the two-sided ideal generated by
	\[
	\begin{aligned}
		&(1+\delta_1)x_1y_1-(1+\delta_2)y_1x_1+(\delta_1-\delta_2),\\
		&(1+\delta_2)x_1z_1-(1+\delta_1)z_1x_1-(\delta_1-\delta_2)x_1,\\
		&(1+\delta_1)y_1z_1-(1+\delta_2)z_1y_1+(\delta_1-\delta_2)y_1.
	\end{aligned}
	\]
	Similarly, we can localize at the arrow $x$. In this chart, the normalized
	generators are chosen as
	$x^{-1}z, yx, wx,$
	which play the role of affine coordinates on the second local chart. This gives the following representations:
	$$G_{02}:=\begin{cases}
		x_2 \mapsto x^{-1}z \\
		y_2 \mapsto yx \\
		z_2 \mapsto wx 
	\end{cases}, \quad G_{20}:=\begin{cases}
		z \mapsto x_2 \\
		x \mapsto 1\\
		y \mapsto y_2\\
		w \mapsto z_2.
	\end{cases}$$ As before, these two representations are inverse to each other up to the prescribed gerbe terms.
	
	Once this gauge normalization is chosen, the presentation of the second local
	chart algebra $\mathcal A_2$ is determined. More explicitly,
    one obtains $\mathcal A_2 \cong
	\mathbb C\langle x_2,y_2,z_2\rangle/I_2,$
	where $I_2$ is the two-sided ideal generated by
	\[
	\begin{aligned}
		&(1+\delta_2)x_2y_2-(1+\delta_1)y_2x_2-(\delta_1-\delta_2)x_2,\\
		&(1+\delta_1)x_2z_2-(1+\delta_2)z_2x_2+(\delta_1-\delta_2),\\
		&(1+\delta_2)z_2y_2-(1+\delta_1)y_2z_2-(\delta_1-\delta_2)z_2.
	\end{aligned}
	\]
	
	Moreover, we have the following transition maps between the local charts:$$G_{21}:= G_{20} \circ G_{01}=\begin{cases}
		x_1 \mapsto x_2^{-1} \\
		y_1 \mapsto y_2x_2 \\
		z_1 \mapsto z_2x_2 
	\end{cases}, \quad G_{12}:= G_{10} \circ G_{02}=\begin{cases}
		x_2 \mapsto x_1^{-1} \\
		y_2 \mapsto y_1x_1 \\
		z_2 \mapsto z_1x_1 
	\end{cases}.$$ The local charts $\mathcal A_1$ and
	$\mathcal A_2$, together with the above transition maps, form a noncommutative
	deformation of the standard two-chart affine cover of the resolved conifold.
\end{proof}

\begin{remark}
	Note that, as in the two-dimensional case, the quiver algebra remains commutative along a distinguished one-parameter direction in the bulk-parameter space. In the present case, this direction is the diagonal direction $\delta_1=\delta_2$.
\end{remark}

\bibliography{mybib}{}

@article{GHL23,
      title={Hochschild cohomology of the second kind: Koszul duality and Morita invariance}, 
      author={Ai Guan and Julian Holstein and Andrey Lazarev and K. Rasmussen},
      year={2023},
      Journal = {arXiv preprint},
      note = {\href{https://arxiv.org/abs/2312.16645}{arXiv:2312.16645}},
}

@article{AFOOO26,
      title={Quantum cohomology and split generation in Lagrangian Floer theory}, 
      author={M. Abouzaid and K. Fukaya and Y. -G. Oh and H. Ohta and K. Ono},
      year={2026},
      Journal = {arXiv preprint},
      eprint={2606.12257},
      archivePrefix={arXiv},
      primaryClass={math.SG},
      url={https://arxiv.org/abs/2606.12257}, 
      note   = {\href{https://arxiv.org/abs/2606.12257}{arXiv:2606.12257}}
}

@article{Ho,
author = {Hansol Hong},
title = {{Maurer-Cartan deformation of Lagrangians}},
volume = {21},
journal = {Journal of Symplectic Geometry},
number = {1},
publisher = {International Press},
pages = {1 -- 71},
year = {2023},
doi = {10.4310/JSG.2023.v21.n1.a1},
URL = {https://dx.doi.org/10.4310/JSG.2023.v21.n1.a1
}
}

@article{BGO25,
      title={Reflexive dg categories in algebra and topology}, 
      author={Matt Booth and Isambard Goodbody and Sebastian Opper},
      year={2025},
      eprint={2506.11213},
      archivePrefix={arXiv},
      primaryClass={math.RT},
      url={https://arxiv.org/abs/2506.11213}, 
      note          = {\href{https://arxiv.org/abs/2506.11213}{arXiv:2506.11213}}
}

@article{Booth1,
	author = {Booth, Matt},
	doi = {10.1007/s00209-021-02892-7},
	fjournal = {Mathematische Zeitschrift},
	issn = {0025-5874},
	journal = {Math. Z.},
	mrclass = {14B10 (14A30 14B20 14F08 18N40)},
	mrnumber = {4381228},
	number = {3},
	pages = {3023--3082},
	title = {The derived deformation theory of a point},
	url = {https://mathscinet.ams.org/mathscinet-getitem?mr=4381228},
	volume = {300},
	year = {2022}}

@article{CHL-toric,
	title = {Localized mirror functor constructed from a {L}agrangian torus},
	journal = {Journal of Geometry and Physics},
	volume = {136},
	pages = {284-320},
	year = {2019},
	issn = {0393-0440},
	author = {Cheol-Hyun Cho and Hansol Hong and Siu-Cheong Lau},
}

@article {AAK16,
    AUTHOR = {Abouzaid, Mohammed and Auroux, Denis and Katzarkov, Ludmil},
     TITLE = {Lagrangian fibrations on blowups of toric varieties and mirror
              symmetry for hypersurfaces},
   JOURNAL = {Publ. Math. Inst. Hautes \'{E}tudes Sci.},
  FJOURNAL = {Publications Math\'{e}matiques. Institut de Hautes \'{E}tudes
              Scientifiques},
    VOLUME = {123},
      YEAR = {2016},
     PAGES = {199--282},
      ISSN = {0073-8301,1618-1913},
   MRCLASS = {53D37 (14F05 14J33 14M25)},
  MRNUMBER = {3502098},
MRREVIEWER = {Shanzhong\ Sun},
       DOI = {10.1007/s10240-016-0081-9},
       URL = {https://doi.org/10.1007/s10240-016-0081-9},
}

@article{ADHM78,
title = {Construction of instantons},
journal = {Physics Letters A},
volume = {65},
number = {3},
pages = {185-187},
year = {1978},
issn = {0375-9601},
doi = {https://doi.org/10.1016/0375-9601(78)90141-X},
url = {https://www.sciencedirect.com/science/article/pii/037596017890141X},
author = {Michael F. Atiyah and Vladimir G. Drinfeld and Nigel J. Hitchin and Yuri I. Manin}
}

@article {Aur07,
    AUTHOR = {Auroux, Denis},
     TITLE = {Mirror symmetry and {$T$}-duality in the complement of an
              anticanonical divisor},
   JOURNAL = {J. G\"{o}kova Geom. Topol. GGT},
  FJOURNAL = {Journal of G\"{o}kova Geometry Topology. GGT},
    VOLUME = {1},
      YEAR = {2007},
     PAGES = {51--91},
      ISSN = {1935-2565},
   MRCLASS = {53D40 (14J32 14J45 53C38)},
  MRNUMBER = {2386535},
MRREVIEWER = {Richard\ P.\ Thomas},
}

@article {AS12,
    AUTHOR = {Abouzaid, Mohammed and Smith, Ivan},
     TITLE = {Exact {L}agrangians in plumbings},
   JOURNAL = {Geom. Funct. Anal.},
  FJOURNAL = {Geometric and Functional Analysis},
    VOLUME = {22},
      YEAR = {2012},
    NUMBER = {4},
     PAGES = {785--831},
      ISSN = {1016-443X,1420-8970},
   MRCLASS = {53D37 (53D12)},
  MRNUMBER = {2984118},
MRREVIEWER = {Rafael\ Santamar\'ia},
       DOI = {10.1007/s00039-012-0162-y},
       URL = {https://doi.org/10.1007/s00039-012-0162-y},
}

@article {BGK02,
    AUTHOR = {Baranovsky, Vladimir and Ginzburg, Victor and Kuznetsov, Alexander},
     TITLE = {Quiver varieties and a noncommutative {${\Bbb P}^2$}},
   JOURNAL = {Compositio Math.},
  FJOURNAL = {Compositio Mathematica},
    VOLUME = {134},
      YEAR = {2002},
    NUMBER = {3},
     PAGES = {283--318},
      ISSN = {0010-437X,1570-5846},
   MRCLASS = {14D20 (14A22 16S38)},
  MRNUMBER = {1943905},
MRREVIEWER = {Olivier\ G.\ Schiffmann},
       DOI = {10.1023/A:1020930501291},
       URL = {https://doi.org/10.1023/A:1020930501291},
}

@article{CHJL25,
      title={Closed-string mirror symmetry for dimer models}, 
      author={Dahye Cho and Hansol Hong and Hyeongjun Jin and Sangwook Lee},
      year={2025},
      Journal = {arXiv preprint},
      note = {\href{https://arxiv.org/abs/2511.06699}{arXiv:2511.06699}},
}

@article {CHL21,
    AUTHOR = {Cho, Cheol-Hyun and Hong, Hansol and Lau, Siu-Cheong},
     TITLE = {Noncommutative homological mirror functor},
   JOURNAL = {Mem. Amer. Math. Soc.},
  FJOURNAL = {Memoirs of the American Mathematical Society},
    VOLUME = {271},
      YEAR = {2021},
    NUMBER = {1326},
     PAGES = {v+116},
      ISSN = {0065-9266,1947-6221},
      ISBN = {978-1-4704-4761-8; 978-1-4704-6630-5},
   MRCLASS = {14J33 (14A22 14F08 53D37)},
  MRNUMBER = {4277886},
MRREVIEWER = {Roman\ Golovko},
       DOI = {10.1090/memo/1326},
       URL = {https://doi.org/10.1090/memo/1326},
}

@article {CPU16,
    AUTHOR = {Chan, Kwokwai and Pomerleano, Daniel and Ueda, Kazushi},
     TITLE = {Lagrangian torus fibrations and homological mirror symmetry
              for the conifold},
   JOURNAL = {Comm. Math. Phys.},
  FJOURNAL = {Communications in Mathematical Physics},
    VOLUME = {341},
      YEAR = {2016},
    NUMBER = {1},
     PAGES = {135--178},
      ISSN = {0010-3616,1432-0916},
   MRCLASS = {53D37 (14F05 14J33 53D12 81T60)},
  MRNUMBER = {3439224},
MRREVIEWER = {Hai-Long\ Her},
       DOI = {10.1007/s00220-015-2477-7},
       URL = {https://doi.org/10.1007/s00220-015-2477-7},
}

@article {EL19,
    AUTHOR = {Tolga Etg{\"u} and Yankı Lekili},
     TITLE = {Fukaya categories of plumbings and multiplicative
              preprojective algebras},
   JOURNAL = {Quantum Topol.},
  FJOURNAL = {Quantum Topology},
    VOLUME = {10},
      YEAR = {2019},
    NUMBER = {4},
     PAGES = {777--813},
      ISSN = {1663-487X,1664-073X},
   MRCLASS = {57R58 (16E45 53D37)},
  MRNUMBER = {4033516},
       DOI = {10.4171/qt/131},
       URL = {https://doi.org/10.4171/qt/131},
}

@article{EL17,
author = {Tolga Etg{\"u} and Yankı Lekili},
title = {{Koszul duality patterns in Floer theory}},
volume = {21},
journal = {Geometry \& Topology},
number = {6},
publisher = {MSP},
pages = {3313 -- 3389},
year = {2017},
}

@article{FHLY17,
  title={Mirror of {A}tiyah flop in symplectic geometry and stability conditions},
  author={Yu-Wei Fan and Hansol Hong and Siu-Cheong Lau and Shing-Tung Yau},
  journal={Adv. Theor. Math. Phys.},
  volume = {22},
  year={2018},
  number = {5},
}

@incollection {FOOO-can,
    AUTHOR = {Fukaya, Kenji and Oh, Yong-Geun and Ohta, Hiroshi and Ono,
              Kaoru},
     TITLE = {Canonical models of filtered {$A_\infty$}-algebras and {M}orse
              complexes},
 BOOKTITLE = {New perspectives and challenges in symplectic field theory},
    SERIES = {CRM Proc. Lecture Notes},
    VOLUME = {49},
     PAGES = {201--227},
 PUBLISHER = {Amer. Math. Soc., Providence, RI},
      YEAR = {2009},
      ISBN = {978-0-8218-4356-7},
}

@article{FOOO09,
author = {Fukaya, Kenji and Oh, Yong-Geun and Ohta, Hiroshi and Ono, Kaoru},
year = {2009},
month = {01},
pages = {},
title = {Lagrangian intersection Floer theory: anomaly and obstruction. Part I, II},
journal = {AMS/IP Studies in Advanced Mathematics, v.46 (2009)}
}

@article {FOOOtoric,
    AUTHOR = {Fukaya, Kenji and Oh, Yong-Geun and Ohta, Hiroshi and Ono,
              Kaoru},
     TITLE = {Lagrangian {F}loer theory and mirror symmetry on compact toric
              manifolds},
   JOURNAL = {Ast\'erisque},
  FJOURNAL = {Ast\'erisque},
    NUMBER = {376},
      YEAR = {2016},
     PAGES = {vi+340},
      ISSN = {0303-1179,2492-5926},
      ISBN = {978-2-85629-825-1},
   MRCLASS = {53D40 (14M25 53D37 53D45)},
  MRNUMBER = {3460884},
MRREVIEWER = {Christopher\ T.\ Woodward},
}

@article {HKL23,
    AUTHOR = {Hong, Hansol and Kim, Yoosik and Lau, Siu-Cheong},
     TITLE = {Immersed two-spheres and {SYZ} with application to
              {G}rassmannians},
   JOURNAL = {J. Differential Geom.},
  FJOURNAL = {Journal of Differential Geometry},
    VOLUME = {125},
      YEAR = {2023},
    NUMBER = {3},
     PAGES = {427--507},
}

@article{HLT24,
      title={Mirror Construction for {N}akajima Quiver Varieties}, 
      author={Jiawei Hu and Siu-Cheong Lau and Ju Tan},
      year={2024},
      Journal = {arXiv preprint},
      note = {\href{https://arxiv.org/abs/2404.16172}{arXiv:2404.16172}}, 
}

@article{JKL-Ginzburg,
      title={Proper modules over {G}inzburg dg algebras and compact {F}ukaya categories of plumbings}, 
      author={Wonbo Jeong and Dogancan Karabas and Sangjin Lee},
      Journal = {arXiv preprint},
      year={2026},
      eprint={2605.08969},
      archivePrefix={arXiv},
      primaryClass={math.SG},
      url={https://arxiv.org/abs/2605.08969}, 
      note          = {\href{https://arxiv.org/abs/2605.08969}{arXiv:2605.08969}}
}

@article{JKL-im,
      title={Generation of immersed {L}agrangians by cocores}, 
      author={Wonbo Jeong and Dogancan Karabas and Sangjin Lee},
      year={2026},
      Journal = {arXiv preprint},
      note = {\href{https://arxiv.org/abs/2605.09082}{arXiv:2605.09082}},
}

@article{KL25,
   title={The wrapped {F}ukaya category of plumbings},
   volume={23},
   ISSN={1540-2347},
   url={http://dx.doi.org/10.4310/JSG.250820172316},
   DOI={10.4310/jsg.250820172316},
   number={4},
   journal={Journal of Symplectic Geometry},
   publisher={International Press of Boston},
   author={Karabas, Dogancan and Lee, Sangjin},
   year={2025},
   pages={853–949} }

@article{kon17,
      title={Higher rank local systems in {L}agrangian {F}loer theory}, 
      author={Momchil Konstantinov},
      year={2017},
      Journal = {arXiv preprint},
      note = {\href{https://arxiv.org/abs/1701.03624}{arXiv:1701.03624}},
}

@article {BKR01,
    AUTHOR = {Bridgeland, Tom and King, Alastair and Reid, Miles},
     TITLE = {The {M}c{K}ay correspondence as an equivalence of derived
              categories},
   JOURNAL = {J. Amer. Math. Soc.},
  FJOURNAL = {Journal of the American Mathematical Society},
    VOLUME = {14},
      YEAR = {2001},
    NUMBER = {3},
     PAGES = {535--554},
      ISSN = {0894-0347,1088-6834},
   MRCLASS = {14E15 (14F05 14J30 18E30 19L47)},
  MRNUMBER = {1824990},
MRREVIEWER = {Ana\ Jerem\'{\i}as L\'{o}pez},
       DOI = {10.1090/S0894-0347-01-00368-X},
       URL = {https://doi.org/10.1090/S0894-0347-01-00368-X},
}

@article {CBH98,
    AUTHOR = {Crawley-Boevey, William and Holland, Martin P.},
     TITLE = {Noncommutative deformations of {K}leinian singularities},
   JOURNAL = {Duke Math. J.},
  FJOURNAL = {Duke Mathematical Journal},
    VOLUME = {92},
      YEAR = {1998},
    NUMBER = {3},
     PAGES = {605--635},
      ISSN = {0012-7094,1547-7398},
   MRCLASS = {14B07 (16G10)},
  MRNUMBER = {1620538},
MRREVIEWER = {Michel\ Van den Bergh},
       DOI = {10.1215/S0012-7094-98-09218-3},
       URL = {https://doi.org/10.1215/S0012-7094-98-09218-3},
}

@article {Her16,
    AUTHOR = {Hermes, Stephen},
     TITLE = {Minimal model of {G}inzburg algebras},
   JOURNAL = {J. Algebra},
  FJOURNAL = {Journal of Algebra},
    VOLUME = {459},
      YEAR = {2016},
     PAGES = {389--436},
      ISSN = {0021-8693,1090-266X},
   MRCLASS = {16G20 (16E35 16E45)},
  MRNUMBER = {3503979},
MRREVIEWER = {Jian\ Min\ Chen},
       DOI = {10.1016/j.jalgebra.2016.03.041},
       URL = {https://doi.org/10.1016/j.jalgebra.2016.03.041},
}

@article {HJL25,
    AUTHOR = {Hong, Hansol and Jin, Hyeongjun and Lee, Sangwook},
     TITLE = {Orbifold {K}odaira-{S}pencer maps and closed-string mirror
              symmetry for punctured {R}iemann surfaces},
   JOURNAL = {J. Lond. Math. Soc. (2)},
  FJOURNAL = {Journal of the London Mathematical Society. Second Series},
    VOLUME = {111},
      YEAR = {2025},
    NUMBER = {5},
     PAGES = {Paper No. e70179, 54},
      ISSN = {0024-6107,1469-7750},
   MRCLASS = {53D37 (53D40)},
  MRNUMBER = {4911865},
MRREVIEWER = {Yuhan\ Sun},
       DOI = {10.1112/jlms.70179},
       URL = {https://doi.org/10.1112/jlms.70179},
}

@article{KN90,
journal = {Mathematische Annalen},
number = {2},
pages = {263-308},
title = {Yang-{M}ills instantons on {ALE} gravitational instantons.},
author = {Kronheimer, Peter B. and Nakajima, Hiraku},
url = {http://eudml.org/doc/164736},
volume = {288},
year = {1990},
}

@phdthesis{Lam15,
  author  = {Yat-Tin Lam},
  title   = {Calabi--Yau Categories and Quivers with Superpotential},
  school  = {University of Oxford},
  year    = {2015},
  note    = {\url{https://people.maths.ox.ac.uk/joyce/theses/LamDPhil.pdf}}
}

@article {LLL24,
    AUTHOR = {Lau, Siu-Cheong and Lee, Tsung-Ju and Lin, Yu-Shen},
     TITLE = {S{YZ} mirror symmetry for del {P}ezzo surfaces and affine
              structures},
   JOURNAL = {Adv. Math.},
  FJOURNAL = {Advances in Mathematics},
    VOLUME = {439},
      YEAR = {2024},
     PAGES = {Paper No. 109488, 57},
      ISSN = {0001-8708,1090-2082},
   MRCLASS = {14J33 (14J27 32Q25 53D37)},
  MRNUMBER = {4690503},
MRREVIEWER = {Yu-Wei\ Fan},
       DOI = {10.1016/j.aim.2024.109488},
       URL = {https://doi.org/10.1016/j.aim.2024.109488},
}

@article{LNT23,
      title={Mirror {S}ymmetry for Quiver Algebroid Stacks}, 
      author={Siu-Cheong Lau and Junzheng Nan and Ju Tan},
      year={2023},
      eprint={2206.03028},
      journal={to appear in J. Symplectic Geom.}
}

@article{LT26,
      title={Mirror construction of {H}ecke correspondence between {N}akajima quiver varieties}, 
      author={Siu-Cheong Lau and Ju Tan},
      year={2026},
      Journal = {arXiv preprint},
      Note = {\href{http://arxiv.org/abs/2601.13555}{arXiv:2601.13555}},
}

@article{Nak94,
author = {Hiraku Nakajima},
title = {{Instantons on ALE spaces, quiver varieties, and Kac-Moody algebras}},
volume = {76},
journal = {Duke Mathematical Journal},
number = {2},
publisher = {Duke University Press},
pages = {365 -- 416},
year = {1994},
}

@article{Nak98,
author = {Hiraku Nakajima},
title = {{Quiver varieties and Kac-Moody algebras}},
volume = {91},
journal = {Duke Mathematical Journal},
number = {3},
publisher = {Duke University Press},
pages = {515 -- 560},
year = {1998},
}

@book {Nak99,
	AUTHOR = {Nakajima, Hiraku},
	TITLE = {Lectures on {H}ilbert schemes of points on surfaces},
	SERIES = {University Lecture Series},
	VOLUME = {18},
	PUBLISHER = {American Mathematical Society, Providence, RI},
	YEAR = {1999},
	PAGES = {xii+132},
}

@article {Nak01,
    AUTHOR = {Nakajima, Hiraku},
     TITLE = {Quiver varieties and finite-dimensional representations of
              quantum affine algebras},
   JOURNAL = {J. Amer. Math. Soc.},
  FJOURNAL = {Journal of the American Mathematical Society},
    VOLUME = {14},
      YEAR = {2001},
    NUMBER = {1},
     PAGES = {145--238},
      ISSN = {0894-0347,1088-6834},
   MRCLASS = {17B37 (14D21)},
  MRNUMBER = {1808477},
MRREVIEWER = {Olivier\ G.\ Schiffmann},
       DOI = {10.1090/S0894-0347-00-00353-2},
       URL = {https://doi.org/10.1090/S0894-0347-00-00353-2},
}

@article{Nak07,
author = {Nakajima, Hiraku},
year = {2007},
month = {01},
pages = {},
title = {Sheaves on {ALE} Spaces and Quiver Varieties},
volume = {7},
journal = {Moscow Mathematical Journal},
}

@article {SYZ96,
    AUTHOR = {Strominger, Andrew and Yau, Shing-Tung and Zaslow, Eric},
     TITLE = {Mirror symmetry is {$T$}-duality},
   JOURNAL = {Nuclear Phys. B},
  FJOURNAL = {Nuclear Physics. B. Theoretical, Phenomenological, and
              Experimental High Energy Physics. Quantum Field Theory and
              Statistical Systems},
    VOLUME = {479},
      YEAR = {1996},
    NUMBER = {1-2},
     PAGES = {243--259},
      ISSN = {0550-3213,1873-1562},
   MRCLASS = {32J17 (14J32 32J81 81T30)},
  MRNUMBER = {1429831},
MRREVIEWER = {Mark\ Gross},
       DOI = {10.1016/0550-3213(96)00434-8},
       URL = {https://doi.org/10.1016/0550-3213(96)00434-8},
}

@book {Sei08,
    AUTHOR = {Seidel, Paul},
     TITLE = {Fukaya categories and {P}icard-{L}efschetz theory},
    SERIES = {Zurich Lectures in Advanced Mathematics},
 PUBLISHER = {European Mathematical Society (EMS), Z\"{u}rich},
      YEAR = {2008},
     PAGES = {viii+326},
      ISBN = {978-3-03719-063-0},
   MRCLASS = {53D40 (16E45 32Q65 53D12)},
  MRNUMBER = {2441780},
MRREVIEWER = {Timothy\ Perutz},
       DOI = {10.4171/063},
       URL = {https://doi.org/10.4171/063},
}

@article {Sei00,
    AUTHOR = {Seidel, Paul},
     TITLE = {Graded {L}agrangian submanifolds},
   JOURNAL = {Bull. Soc. Math. France},
  FJOURNAL = {Bulletin de la Soci\'et\'e{} Math\'ematique de France},
    VOLUME = {128},
      YEAR = {2000},
    NUMBER = {1},
     PAGES = {103--149},
      ISSN = {0037-9484,2102-622X},
   MRCLASS = {53D40 (53D12 57R58)},
  MRNUMBER = {1765826},
MRREVIEWER = {David\ E.\ Hurtubise},
       URL = {http://www.numdam.org/item?id=BSMF_2000__128_1_103_0},
}

@article{Kaw24A,
      title={Non-Commutative Deformations of Derived {McKay} Correspondence for {A(n)} singularities}, 
      author={Yujiro Kawamata},
      Journal = {arXiv preprint},
      year={2024},
      Note = {\href{http://arxiv.org/abs/2410.15340}{arXiv:2410.15340}},
}

@article {KV00,
    AUTHOR = {Kapranov, M. and Vasserot, E.},
     TITLE = {Kleinian singularities, derived categories and {H}all
              algebras},
   JOURNAL = {Math. Ann.},
  FJOURNAL = {Mathematische Annalen},
    VOLUME = {316},
      YEAR = {2000},
    NUMBER = {3},
     PAGES = {565--576},
      ISSN = {0025-5831,1432-1807},
   MRCLASS = {14F05 (14J17 14L30 18E30 19A49 32S25)},
  MRNUMBER = {1752785},
MRREVIEWER = {I.\ Dolgachev},
       DOI = {10.1007/s002080050344},
       URL = {https://doi.org/10.1007/s002080050344},
}

@article {KKO01,
    AUTHOR = {Kapustin, Anton and Kuznetsov, Alexander and Orlov, Dmitri},
     TITLE = {Noncommutative instantons and twistor transform},
   JOURNAL = {Comm. Math. Phys.},
  FJOURNAL = {Communications in Mathematical Physics},
    VOLUME = {221},
      YEAR = {2001},
    NUMBER = {2},
     PAGES = {385--432},
      ISSN = {0010-3616,1432-0916},
   MRCLASS = {58B34 (14A22 32L25 81R60)},
  MRNUMBER = {1845330},
MRREVIEWER = {Alexander\ E.\ Polishchuk},
       DOI = {10.1007/PL00005576},
       URL = {https://doi.org/10.1007/PL00005576},
}

@article {CHL17,
    AUTHOR = {Cho, Cheol-Hyun and Hong, Hansol and Lau, Siu-Cheong},
     TITLE = {Localized mirror functor for {L}agrangian immersions, and
              homological mirror symmetry for {$\Bbb{P}^1_{a,b,c}$}},
   JOURNAL = {J. Differential Geom.},
  FJOURNAL = {Journal of Differential Geometry},
    VOLUME = {106},
      YEAR = {2017},
    NUMBER = {1},
     PAGES = {45--126},
      ISSN = {0022-040X,1945-743X},
   MRCLASS = {53D37 (14F05 14J33 53D40)},
  MRNUMBER = {3640007},
MRREVIEWER = {Guangbo\ Xu},
       DOI = {10.4310/jdg/1493172094},
       URL = {https://doi.org/10.4310/jdg/1493172094},
}

@article {Tod18,
    AUTHOR = {Toda, Yukinobu},
     TITLE = {Moduli stacks of semistable sheaves and representations of
              {E}xt-quivers},
   JOURNAL = {Geom. Topol.},
  FJOURNAL = {Geometry \& Topology},
    VOLUME = {22},
      YEAR = {2018},
    NUMBER = {5},
     PAGES = {3083--3144},
      ISSN = {1465-3060,1364-0380},
   MRCLASS = {14D22 (14D15 14D23)},
  MRNUMBER = {3811778},
MRREVIEWER = {Abdelmoubine\ Amar\ Henni},
       DOI = {10.2140/gt.2018.22.3083},
       URL = {https://doi.org/10.2140/gt.2018.22.3083},
}
\bibliographystyle{alpha}
\end{document}